\documentclass[a4paper,11pt,intlimits,oneside]{amsart}

\usepackage{amssymb,latexsym,amsmath,mathrsfs}
\usepackage{amsfonts,amsbsy,bm}
\usepackage{color}
\usepackage[margin=2cm]{geometry}

\usepackage{amsmath}
\usepackage{amssymb}
\usepackage{amsfonts}
\usepackage{amsthm,amscd}

\newtheorem{theorem}{Theorem}[section]
\newtheorem{proposition}[theorem]{Proposition}
\newtheorem{corollary}[theorem]{Corollary}
\newtheorem{lemma}[theorem]{Lemma}
\newtheorem{definition}[theorem]{Definition}

\theoremstyle{definition}
\newcommand{\comment}[1]{}

\numberwithin{equation}{section}

\def\supp{{\rm supp}}

\theoremstyle{definition}

\newcommand{\Be}{\begin{equation}}
\newcommand{\Ee}{\end{equation}}
\newcommand{\Bea}{\begin{eqnarray}}
\newcommand{\Eea}{\end{eqnarray}}
\newcommand{\Bes}{\begin{equation*}}
\newcommand{\Ees}{\end{equation*}}
\newcommand{\Beas}{\begin{eqnarray*}}
\newcommand{\Eeas}{\end{eqnarray*}}
\newcommand{\Ba}{\begin{array}}
\newcommand{\Ea}{\end{array}}

\begin{document}
\title[Two-weight inequlities for multilinear maximal functions]{Two-weight inequalities for multilinear maximal functions in Orlicz spaces}
\author{}
\author[J.M. Tanoh Dje and B. F.  Sehba]{Jean-Marcel Tanoh Dje and Beno\^it F. Sehba}
\address{Unit\'e de Recherche et d'Expertise Num\'erique, 
Universit\'e virtuelle de C\^ote d'Ivoire, Abidjan.}
\email{{\tt tanoh.dje@uvci.edu.ci}}
\address{Department of Mathematics, University of Ghana,  P.O. Box L.G 62 Legon, Accra, Ghana.}
\email{{\tt bfsehba@ug.edu.gh}}

\subjclass{42B25,42B35}
\keywords{Orlicz space,   Carleson sequence, Maximal operator,  dyadic grids,  Weight}

\date{}

\begin{abstract}
 In this note, we provide various two-weight norm estimates of the multi-linear fractional maximal function and weighted maximal function between different Orlicz spaces. More precisely, we obtain Sawyer-type characterizations and norm estimates for these operators.

\end{abstract}

\maketitle

\section{Introduction}
Denote by $\mathcal{Q}$ the set of all non-degenerate cubes with sides parallel to the coordinate axes in $\mathbb R^d$. If $Q$ is a cube, then we denote by $|Q|$ its Lebesgue measure. For $\omega$ a weight on $\mathbb R^d$, and $E$ a  measurable subset of $\mathbb R^d$, we use the notation $\omega(E):=\int_E\omega(x)dx$. 
\vskip .2cm
For $1< p<\infty$, and $\omega$ a weight, the weighted Lebesgue space $L_\omega^p(\mathbb R^d)$ is the set of all measurable functions $f$ such that
$$
\|f\|_{p,\omega}=\|f\|_{L^p(\omega)}:=\left(\int_{\mathbb R^d}|f(x)|^p\omega(x)dx\right)^{1/p}<\infty.
$$
\vskip .2cm
For $0\leq \alpha<d$, the fractional maximal function $\mathcal M_\alpha$ is defined by
\Bes
\mathcal{M}_{\alpha}f(x):=\sup_{Q\in \mathcal{Q}}\frac{\chi_Q(x)}{|Q|^{1-\alpha /d}}\int_Q|f(y)|dy.
\Ees
When $\alpha=0$, this is just the Hardy-Littlewood maximal function denoted $\mathcal{M}$.
\medskip

One of the most studied problems in this type of analysis is the problem of finding the conditions on the pair of weights $(\sigma,\omega)$ such that
\Be\label{eq:maxineqpb}\|\mathcal M_\alpha (\sigma f
)\|_{q,\omega}\le C\|f\|_{p,\sigma}.\Ee
Here the constant $C$ doesn't depend on $f$.
\medskip

    Eric Sawyer's answer (in \cite{Sawyer5}) to the above question is that (\ref{eq:maxineqpb}) holds if and only if it holds for characteristic functions of cubes.
    \medskip
    
    Several authors have also considered quantitative analysis of (\ref{eq:maxineqpb}), i.e. an estimate of the constant $C$ in terms of positive quantities involving the weights and the exponents $p$ and $q$ (see for example \cite{CaoXue,HytPerez,Kerman,kokokrbec,Moen1,Pereyra,RahmSpencer, Israel} and the references therein). These quantitative inequalities are usually in terms of Muckenhoupt classes  and their variations.
    \medskip 
    
    There are also formulations of problem (\ref{eq:maxineqpb}) in the setting of Orlicz spaces. Here again, several results have been obtained by various authors (see for example \cite{Bloom,jmtafeuseh,kokokrbec,LQ}). In particular, in \cite{jmtafeuseh}, we formulated and proved Sawyer-type characterization of the above question in the Orlicz setting. In the same paper, we proved several quantitative inequalities that involved classes that generalize Muckenhoupt classes and the $B_p$ class introduced in \cite{HytPerez}.
    \medskip 
    
    The aim of this paper is to consider the above problem for the multilinear maximal functions in the setting of Orlicz spaces. Our results extend those of \cite{ChenDamian,jmtafeuseh,LiSun,sehbaf}. A key argument in our work will be an extension to the multilinear setting of the Carleson Embedding Lemma obtained in \cite{jmtafeuseh}. The difficulties encountered here are essentially the same as in the power functions case; for example, one does not know if the Sawyer-type characterization is necessary and sufficient in general. In \cite{ChenDamian,sehbaf} a Sawyer-type characterization is shown to be sufficient for the boundedness of the maximal functions but this condition is shown to be also necessary only under an additional condition that was called in \cite{ChenDamian} "reverse H\"older inequality". In this paper, we will provide an analogue of this condition and an example of weights for which it is true. 

\section{The setting of interest }

We recall that a growth function is any function 
$\Phi: \mathbb{R}_{+}\longrightarrow \mathbb{R}_{+}$, not identically null, that is continuous, nondecreasing and onto. 
\vskip .2cm
We observe that if $\Phi$ is growth function, then  $\Phi(0)=0$ and $\lim_{t \to +\infty}\Phi(t)=+\infty$.
\vskip .2cm
The function $t\mapsto \exp(t)-1$ is an example of growth function.
\vskip .2cm
Let $\Phi$ be a growth function. We say $\Phi$ satisfies the condition

\begin{itemize}
\item[(a)]   $\Delta_{2}$
(or $\Phi \in \Delta_{2}$), if there is a constant $K>0$ such that 
\begin{equation}\label{eq:delta2}
\Phi(2t) \leq K \Phi(t),~ \forall~ t > 0.\end{equation}
\item[(b)]   $\Delta'$  (or $\Phi \in \Delta^{\prime}$), if there is a constant $C> 0$ such that
 \begin{equation}\label{eq:deprim}
  \Phi(st) \leq C \Phi(s)\Phi(t),~~\forall~s,t > 0. \end{equation}
 \end{itemize}
We have that if $\Phi \in \Delta^{\prime}$, then  $\Phi \in \Delta_{2}$.
\vskip .2cm
Let  $\Phi$ be a convex growth function such that $\Phi \in \Delta_{2}$. We say $\Phi$ satisfies the condition $\nabla_{2}$ (or $\Phi \in \nabla_{2}$), if there exists a constant $C >0$ such that
\begin{equation}\label{eq:conditiondedinis}
\int_0^t\frac{\Phi(s)}{s^2}ds\le C\frac{\Phi(t)}{t}, ~~ \forall~t>0.\end{equation}
Let $\alpha > 1$. Then the function $t\longmapsto t^{\alpha}\ln(e+t)$ is a growth function satisfying the condition $\nabla_{2}$.
\vskip .2cm
Let $q>0$ and  $\Phi$ be a growth function. We say  $\Phi$ is of upper-type $q$, if there exists a constant $C_{q}>0$  such that
\begin{equation}\label{eq:sui8n}
\Phi(st)\leq C_{q}t^{q}\Phi(s),~~ \forall~s>0, ~~ \forall~t\geq 1.  \end{equation}
We denote by $\mathscr{U}^{q}$, the set of all growth functions of upper-type $q\geq 1$ such that the function $t\mapsto \frac{\Phi(t)}{t}$ is nondecreasing on $\mathbb{R}_{+}^{*}=\mathbb{R}_{+}\setminus\{0\}$. We then set
$$   \mathscr{U}:=\bigcup_{q\geq 1}\mathscr{U}^{q}.$$
Any element in  $\mathscr{U}$ is a homeomorphism of $\mathbb{R}_{+}$ onto $\mathbb{R}_{+}$. 
\vskip .2cm
We say two growth function $\Phi_{1}$ and  $\Phi_{2}$  are equivalent, if there exists a constant $c> 0$ such that  
\begin{equation}\label{eq:equivalent}
c^{-1}\Phi_{1}(c^{-1}t) \leq \Phi_{2}(t)\leq c\Phi_{1}(ct), ~~ \forall~ t > 0.\end{equation}
We can assume that any  $\Phi\in \mathscr{U}$ is  a convex function that belongs to the class  $\mathscr{C}^{1}(\mathbb{R}_{+})$, and satisfies
\begin{equation}\label{eq:ealent}
\Phi'(t)\approx \frac{\Phi(t)}{t}, ~~ \forall~ t > 0,\end{equation}
(see for example \cite{djesehb}).

Let  $q\geq 1$ and $\Phi$ a growth function. We say $\Phi\in \widetilde{\mathscr{U}}^{q}$ if the following conditions are satisfied 
\begin{itemize}
\item[(a)]  $\Phi \in \Delta^{\prime}$, 
\item[(b)] $\Phi\in \mathscr{U}^{q}$, 
\item[(c)] there exists a constant  $C>0$  such that    
\begin{equation}\label{eq:ibide2}
\Phi\left(\frac{s}{t}\right)\leq C\frac{\Phi(s)}{t^{q}},  ~~ \forall~ s,t \geq 1.
\end{equation}
\end{itemize}
We set
$$   \widetilde{\mathscr{U}}:=\bigcup_{q\geq 1}\widetilde{\mathscr{U}}^{q}.$$
Let $\alpha \geq 1$ and  $\beta> 0$. Then the function $t\mapsto t^{\alpha}\log^{\beta}(1+t)$ is a growth function that belongs to the class  $\widetilde{\mathscr{U}}$.
\vskip .3cm

Let $\Phi \in \mathscr{U}$ and let $\sigma$ be a weight on $\mathbb{R}^{d}$. The weighted Orlicz space $L^{\Phi}(\sigma)$ is the set of all equivalent classes (in the usual sencse) of measurable functions $f:\mathbb{R}^d\rightarrow\mathbb{C}$ such that 
\begin{equation}\label{eq:in4aq5ale1}	\|f\|_{L_{\sigma}^{\Phi}}^{lux}:=\inf\left\{\lambda>0 : \int_{\mathbb{R}^{d}}\Phi\left(\frac{|f(x)|}{\lambda}\right)\sigma(x)dx \leq 1  \right\}< \infty.
        \end{equation} 
The map $  f\mapsto \|f\|_{L_{\sigma}^{\Phi}}^{lux}$ is a norm on $L^{\Phi}(\sigma)$, and endowed with this norm,
  $L^{\Phi}(\sigma)$ is Banach space (see \cite{raoren}).

\section{Presentation of the results}
Let $\Psi,\Phi_1,\cdots,\Phi_n   \in \mathscr{U}$  and let $\sigma_1,\cdots,\sigma_n,\omega$ be weights. We say an operator $$T:L^{\Phi_1}(\sigma_1)\times\cdots\times L^{\Phi_n}(\sigma_n)\longrightarrow L^{\Psi}(\omega)$$ is bounded, if there exists a constant $C>0$ such that for any $$(f_1,\cdots,f_n)\in L^{\Phi_1}(\sigma_1)\times\cdots\times L^{\Phi_n}(\sigma_n),$$
\Be\label{eq:opernormdef}\|T(f_1,\cdots,f_n)\|_{L_{\omega}^{\Psi}}^{lux}\le C\prod_{j=1}^n\|f_j\|_{L_{\sigma_j}^{\Phi_j}}^{lux}.\Ee
In this case, we denote by  $\|T\|_{\left(\prod_{j=1}^nL_{\sigma_j}^{\Phi_j}\right)\longrightarrow L_{\omega}^{\Psi}}$, the infimum of the constants $C$ such that (\ref{eq:opernormdef}) holds.
\vskip .1cm
We will be using the notation $$\mathbb{R}_+^*:=\{x\in \mathbb{R}: x>0\}.$$
\subsection{Carleson embedding results}
In the sequel, $\mathcal{D}^{\beta}$ denotes a dyadic grid. We recall the following definition of a  $(\sigma, \Phi)-$Carleson sequence as introduced in  \cite{jmtafeuseh}.
 
 \begin{definition}
 Let $\{\lambda_{Q}\}_{Q \in \mathcal{D}^{\beta}}$  be a sequence of positive numbers, $\Phi$ an increasing function and $\sigma$ a weight on $\mathbb{R}^{d}$. We say  $\{\lambda_{Q}\}_{Q \in \mathcal{D}^{\beta}}$ is a $(\sigma, \Phi)-$Carleson, if there exists a constant $\Lambda>0$ such that
 \begin{equation}\label{eq:suit564c25son}
 \forall~ R\in \mathcal{D}^{\beta},~~ 
 \sum_{Q \subset R, Q \in \mathcal{D}^{\beta}}\lambda_{Q} \leq \dfrac{\Lambda}{\Phi\left( \frac{1}{\sigma(R)}\right)}.\end{equation}
 \end{definition}
\vskip .1cm
The small constant $\Lambda$ for wich (\ref{eq:suit564c25son}) holds is called the Carleson constant of the sequence and will be denoted $\Lambda_{Carl}$.
\vskip .2cm
Our main result here is the following which extends \cite[Lemma 3.7]{sehbaf}.

\begin{theorem}\label{pro:main16yh6}
Let $\{\lambda_{Q}\}_{Q \in \mathcal{D}^{\beta}}$ be a sequence of positive number indexed over the dyadic grid $\mathcal{D}^{\beta}$. Let $\sigma_{1},...,\sigma_{n}$ be weights on $\mathbb{R}^{d}$, and let $\Phi_{1},...,\Phi_{n} \in \widetilde{\mathscr{U}}\cap \nabla_{2}$. Suppose  $\Psi\in \mathscr{U}$ and $\Phi$ is a one-to-one correspondence from $\mathbb{R}_{+}$ to $\mathbb{R}_{+}$ such that  
 $\Phi^{-1}=\Phi_{1}^{-1}\times...\times \Phi_{n}^{-1}$ and $t\mapsto \frac{\Psi(t)}{\Phi(t)}$ is increasing on $\mathbb{R}_{+}^{*}$. Define  $\nu_{\overrightarrow{\sigma}}:=  \frac{1}{\Phi\left(\prod_{i=1}^{n}\Phi_{i}^{-1}\left(\frac{1}{\sigma_{i}}\right)  \right)}.$ 
 \vskip .2cm
 If there exists a constant $\Lambda>0$ such that
 \begin{equation}\label{eq:s564c25son}
 \forall~ R\in \mathcal{D}^{\beta} ,~~ 
 \sum_{Q \subset R, Q \in \mathcal{D}^{\beta}}\lambda_{Q} \leq \dfrac{\Lambda}{\Psi\circ\Phi^{-1}\left(\frac{1}{\nu_{\overrightarrow{\sigma}}(R)}\right)},\end{equation}
 then for any $  (f_{1},...,f_{n})\in L^{\Phi_{1}}(\sigma_{1})\times...\times L^{\Phi_{n}}(\sigma_{n})$ with $f_{i}\not\equiv 0$,  $i=1,..., n$,
 \begin{equation}\label{eq:suitc5frmaron}
  \sum_{Q\in \mathcal{D}^{\beta}}\lambda_{Q}\Psi\left(\prod_{i=1}^{n}\frac{1}{\sigma_{i}(Q)}\int_{Q}\frac{|f_{i}(x)|}{\|f_{i}\|_{L_{\sigma_{i}}^{\Phi_{i}}}^{lux}}\sigma_{i}(x)dx\right) \leq \Lambda C_{n,\Phi, \Psi,\Phi_{1},...,\Phi_{n}}. 
 \end{equation}
 \end{theorem}
It is natural to ask if condition (\ref{eq:suitc5frmaron}) implies (\ref{eq:s564c25son}). It can be proved that this is the case when $\Psi(t)=t^{q}$, $\Phi(t)=t^{p}$ and $\Phi_{i}(t)=t^{p_{i}}$, provided the following condition holds
\Be\label{eq:multireverseholder}
\prod_{i=1}^n\sigma_i(Q)^{p/{p_i}}\lesssim \nu_{\overrightarrow{\sigma}}(Q), \textrm{for any}\quad Q\in\mathcal{Q}
\Ee
with in this case, $\nu_{\overrightarrow{\sigma}}=\prod_{i=1}^n\sigma_i^{p/{p_i}}$.
\vskip .2cm
Condition (\ref{eq:multireverseholder}) has shown to be useful in the study of multilinear weighted inequalities (see for example \cite{ChenDamian,sehbaf}). 

We can make the same observations in this setting. We will prove the following which provides the reverse in the above theorem.
\begin{proposition}\label{pro:main16aqyh6}
 Let $\{\lambda_{Q}\}_{Q \in \mathcal{D}^{\beta}}$ be a sequence of positive numbers, and $\sigma_{1},...,\sigma_{n}$ be weights on $\mathbb{R}^{d}$. Assume that $\Phi_{1},...,\Phi_{n},\Psi\in \mathscr{U}$ and $\Phi$ is a one-to-one correspondence from $\mathbb{R}_{+}$ to $\mathbb{R}_{+}$ such that 
  $\Phi^{-1}=\Phi_{1}^{-1}\times...\times \Phi_{n}^{-1}$ and $t\mapsto \frac{\Psi(t)}{\Phi(t)}$ is increasing on $\mathbb{R}_{+}^{*}$. Define $\nu_{\overrightarrow{\sigma}}:=  \frac{1}{\Phi\left(\prod_{i=1}^{n}\Phi_{i}^{-1}\left(\frac{1}{\sigma_{i}}\right)  \right)}.$ If there exist constants  $C_{1}$ and $C_{2}>0$ such that
  \begin{equation}\label{eq:surmaron}
  \frac{1}{\nu_{\overrightarrow{\sigma}}(Q)} \leq C_{1}\Phi\left(\prod_{i=1}^{n}\Phi_{i}^{-1}\left(\frac{1}{\sigma_{i}(Q)}\right)  \right),~~\forall~Q \in \mathcal{D}^{\beta}, 
  \end{equation}
 and  \begin{equation}\label{eq:suitc5faqrmaron}
  \sum_{Q\in \mathcal{D}^{\beta}}\lambda_{Q}\Psi\left(\prod_{i=1}^{n}m_{\sigma_{i}}\left(\frac{f_{i}}{\|f_{i}\|_{L_{\sigma_{i}}^{\Phi_{i}}}^{lux}},Q\right)\right) \leq C_{2}, ~~\forall~ 0\not\equiv f_{i}\in L^{\Phi_{i}}(\sigma_{i}), 
 \end{equation}
 then $\{\lambda_{Q}\}_{Q \in \mathcal{D}^{\beta}}$ is a $(\nu_{\overrightarrow{\sigma}}, \Psi \circ \Phi^{-1})-$Carleson.
  \end{proposition}
It is easy to see that when $\Phi(t)=t^{p}$ and $\Phi_{i}(t)=t^{p_{i}}$, $i=1,\cdots,n$, condition (\ref{eq:surmaron}) is just (\ref{eq:multireverseholder}).
\vskip .2cm
Let us recall that for $1<p<\infty$, the weight $\omega$ is said to belong to the Muckenhoupt class $A_p$, if \begin{equation}\label{eq:emukodou}
[\omega]_{A_{p}}= \sup_{Q \subset \mathcal{Q}}\left( \frac{1}{|Q|}\int_{Q}\omega(x)dx\right)\left( \frac{1}{|Q|}\int_{Q}(\omega(x))^{1-p'}dx\right)^{p-1}< \infty.\end{equation}
We say the weight $\omega$ belongs to the class $A_1$ if \begin{equation}\label{eq:em1kodou}
[\omega]_{A_{1}}= \sup_{Q \subset \mathcal{Q}}\left( \frac{1}{|Q|}\int_{Q}\omega(x)dx\right)\left(\inf ess_{x \in Q} \omega(x)\right)^{-1}< \infty.\end{equation}
The class $A_\infty$ is then defined as $$A_\infty:=\bigcup_{p\ge 1}A_p.$$
\vskip .2cm
It is easy to see that condition (\ref{eq:multireverseholder}) holds if $\sigma_i=\sigma, \forall i\in \{1,\cdots,n\}$. It is also proved in  \cite{CruzMoen} that (\ref{eq:multireverseholder}) holds when
$$\sigma_i\in A_\infty, \forall i\in \{1,\cdots,n\}.$$

One easily checks that condition (\ref{eq:surmaron}) also holds when $\sigma_i=\sigma, \forall i\in \{1,\cdots,n\}$. We will prove later on that (\ref{eq:surmaron}) holds if  $$\sigma_i\in A_1, \forall i\in \{1,\cdots,n\}.$$

\subsection{Weighted inequalities for the weighted multilinear maximal function.}

 Let $\sigma_{1},...,\sigma_{n}$ be weights on $\mathbb{R}^{d}$ and
$f_{1},...,f_{n}$ be complex-valued measurable functions on $\mathbb{R}^{d}$. Put $\overrightarrow{\sigma}:=(\sigma_{1},...,\sigma_{n})$ and $\overrightarrow{f}:=(f_{1},...,f_{n})$.   
 The weighted multilinear Hardy-Littlewood maximal function  $\mathcal{M}_{\overrightarrow{\sigma}}(\overrightarrow{f})$ of $\overrightarrow{f}$ is defined by 
 
\begin{equation}\label{eq:ega5li6nlf}
\mathcal{M}_{\overrightarrow{\sigma}}(\overrightarrow{f})(x):= \sup_{Q \in \mathcal{Q}}\prod_{i=1}^{n}\frac{\chi_{Q}(x)}{\sigma_{i}(Q)}\int_{Q}|f_{i}(y)|\sigma_{i}(y) dy,
~~\forall~ x \in \mathbb{R}^{d}.\end{equation}
If $n=1$, then we recover the weighted Hardy-Littlewood maximal function $\mathcal{M}_{\sigma}$ and if moreover,  $\sigma \equiv 1$, we recover the usual Hardy-Littlewood maximal function.

A weight  $\sigma$ on $\mathbb{R}^{d}$ is said to be doubling, if there exists a constant $C>1$ such that
\begin{equation}\label{eq:cdubdedinis}
\ell(Q^{\prime})\leq 2\ell(Q)\Rightarrow\sigma(Q^{\prime}) \leq C\sigma(Q),~~ \forall~Q,Q' \in \mathcal{Q}~~ \text{with}~~ Q \subset Q'
.\end{equation}
\vskip .2cm

Let  $\sigma_{1},...,\sigma_{n},\omega$ be weights on $\mathbb{R}^{d}$ and let  $\Psi,\Phi,\Phi_{1},...,\Phi_{n}$ be one-to-one correspondences from $\mathbb{R}_{+}$ to itself.  Let us set $\overrightarrow{\Phi}:=(\Phi_{1},...,\Phi_{n})$,  $\overrightarrow{\sigma}:=(\sigma_{1},...,\sigma_{n})$ and $\nu_{\overrightarrow{\sigma}}:= \frac{1}{\Phi\left(\prod_{i=1}^{n}\Phi_{i}^{-1}\left(\frac{1}{\sigma_{i}}\right)  \right)}$.
\vskip .1cm
We say the pair of weights $(\overrightarrow{\sigma},\omega)$  belongs to the class
$M_{\overrightarrow{\Phi},\Psi}$ (or $(\overrightarrow{\sigma}, \omega)\in M_{\overrightarrow{\Phi},\Psi}$), if  
\begin{equation}\label{eq:suqsfi5t35an1}
[\overrightarrow{\sigma},\omega]_{M_{\overrightarrow{\Phi},\Psi}}:=\sup_{Q\in \mathcal{Q} }\omega(Q)\Psi \circ \Phi^{-1}\left( \frac{1}{\nu_{\overrightarrow{\sigma}}(Q)}\right)<\infty. \end{equation}
When $n=1$, we recover the definition of the class $M_{\Phi,\Psi}$ introduced in \cite{jmtafeuseh}.
\vskip .2cm
We will prove the following result.

\begin{theorem}\label{pro:maing1}
Let $\sigma_{1},...,\sigma_{n},\omega$ be weights on $\mathbb{R}^{d}$,  $\Phi_{1},...,\Phi_{n} \in \widetilde{\mathscr{U}}\cap \nabla_{2}$, and $ \Psi \in \widetilde{\mathscr{U}}$. Assume that $\Phi:\mathbb{R}_{+}\longrightarrow\mathbb{R}_{+}$ is a one-to-one correspondence such that 
 $\Phi^{-1}=\Phi_{1}^{-1}\times...\times \Phi_{n}^{-1}$ and $t\mapsto \frac{\Psi(t)}{\Phi(t)}$ is increasing on $\mathbb{R}_{+}^{*}$.  If for any $1 \leq i \leq n$, $\sigma_{i}$ is doubling, and $(\overrightarrow{\sigma}, \omega)\in M_{\overrightarrow{\Phi},\Psi}$, then there exists a constant $C:=C_{n,\sigma_{1},...,\sigma_{n},\Psi,\Phi_{1},...,\Phi_{n}}>0$ such that for any  $(f_{1},...,f_{n}) \in L^{\Phi_{1}}(\sigma_{1})\times...\times L^{\Phi_{n}}(\sigma_{n})$, 
 \begin{equation}\label{eq:eaqoicz}
 \|\mathcal{M}_{\overrightarrow{\sigma}}(f_{1},...,f_{n}) \|_{L_{\omega}^{\Psi}}^{lux} \leq C \Psi^{-1}\left([\overrightarrow{\sigma},\omega]_{M_{\overrightarrow{\Phi},\Psi}}\right)\prod_{i=1}^{n}\|f_{i}\|_{L_{\sigma_{i}}^{\Phi_{i}}}^{lux}. \end{equation}
\end{theorem}

When $\sigma_{1}=\sigma_2=\cdots=\sigma_n=\sigma$ Theorem  \ref{pro:maing1} reduces to the following.

\begin{corollary}\label{pro:maingp1}
Let $\sigma,\omega$ be two weights on $\mathbb{R}^{d}$,  $\Phi_{1},...,\Phi_{n} \in \mathscr{U}\cap \nabla_{2}$,  $\Psi \in \widetilde{\mathscr{U}}$. Assume that $\Phi:\mathbb{R}_{+} \longrightarrow \mathbb{R}_{+}$ is a one-to-one correspondence such that 
 $\Phi^{-1}=\Phi_{1}^{-1}\times...\times \Phi_{n}^{-1}$ and $t\mapsto \frac{\Psi(t)}{\Phi(t)}$ is increasing on $\mathbb{R}_{+}^{*}$. If $\sigma$ is doubling, then the following are equivalent.
\begin{itemize}
\item[(i)] $(\overrightarrow{\sigma}, \omega)\in M_{\overrightarrow{\Phi},\Psi}$,  
\item[(ii)]  $\mathcal{M}_{\overrightarrow{\sigma}}: L^{\Phi_{1}}(\sigma) \times...\times L^{\Phi_{n}}(\sigma)\longrightarrow  L^{\Psi}(\omega)$ boundedly. 
\end{itemize}
Moreover,  
\begin{equation}\label{eq:esqicz}
\|\mathcal{M}_{\overrightarrow{\sigma}}\|_{\prod_{i=1}^{n}L_{\sigma}^{\Phi_{i}}\longrightarrow L_{\omega}^{\Psi}} \approx  \Psi^{-1}\left( [\overrightarrow{\sigma},\omega]_{M_{\overrightarrow{\Phi},\Psi}}\right).    \end{equation}
\end{corollary}
We also obtain the following result.
\begin{theorem}\label{thm:LiSuaqn} Let   $\sigma_1,\ldots,\sigma_n, \omega$ be weights on $\mathbb{R}^{d}$,  let $\Phi_1,...,\Phi_{n} \in \mathscr U \cap \nabla_{2}$ and $\Psi\in \widetilde{\mathscr U}$. Define
$$[ {\vec{\sigma}},\omega]_{K_{\vec {\Phi},\Psi}}:=\sup_{Q\in \mathcal Q}\omega(Q)\left(\prod_{i=1}^n\Psi\circ\Phi_i^{-1}\left(\frac{1}{\sigma_i(Q)}\right)\right).$$
If for each $1 \leq i \leq n$, the function $t\mapsto\frac{\Psi(t)}{\Phi_i(t)}$ is nondecreasing on  $\mathbb{R}^{*}_{+}$ and  $\sigma_{i}$ is doubling, then    $\mathcal{M}_{\overrightarrow{\sigma}}$ is bounded from $L^{\Phi_1}(\sigma_1)\times\cdots\times L^{\Phi_n}(\sigma_n)$ to $L^\Psi(\omega)$ if and only if $[ {\vec{\sigma}},\omega]_{K_{\vec {\Phi},\Psi}}$ is finite. 
Moreover,   
\begin{equation}\label{eq:esqicz}
\|\mathcal{M}_{\overrightarrow{\sigma}}\|_{\prod_{i=1}^{n}L_{\sigma}^{\Phi_{i}}\longrightarrow L_{\omega}^{\Psi}} \approx  \Psi^{-1}\left( [\overrightarrow{\sigma},\omega]_{K_{\overrightarrow{\Phi},\Psi}}\right).    \end{equation}
\end{theorem}

\subsection{Sawyer-type characterizations for the multilinear fractional maximal function.}
Recall that the multilinear fractional maximal function is defined by $$\mathcal{M}_\alpha(\vec{f})(x):=\sup_{Q\in \mathcal{Q}}|Q|^{\alpha /d}\prod_{i=1}^n\frac{\chi_Q(x)}{|Q|}\int_Q|f_i(y)|dy$$
provided $0\le \alpha<nd$. Here $\vec{f}=(f_1,\cdots,f_n)$ with each $f_i$ being a measurable function. When $\alpha=0$, $\mathcal{M}_0=\mathcal{M}$ is the multilinear Hardy-Littlewood maximal function. These operators appear in the study of multilinear fractional integral operators and are also related to multilinear Calder\'on-Zygmund theory (see for example \cite{Grafakos1,Grafakos2,Grafakos3,KS,LOPTT,Moen2}).
\vskip .2cm
Sawyer-type inequalities for the above operator were obtained in \cite{ChenDamian,LiSun,sehbaf}. In the sequel, when $\sigma_1,\cdots,\sigma_n$ are weights on $\mathbb{R}^d$, we write $\vec{\sigma}=(\sigma_1,\cdots,\sigma_n)$, and $\vec{\sigma}\cdot\vec{f}=(\sigma_1f_1,\cdots,\sigma_nf_n)$.
\vskip .1cm
The following is an extension of the main result in \cite{LiSun} that was reproved in \cite[Theorem 2.1]{sehbaf}.
\begin{theorem}\label{thm:LiSun} Let $n$ be a nonnegative integer. Given $\Phi_i\in \mathscr U \cap \nabla_{2}$, $i=1,2,\cdots,n$, and $\Psi\in \widetilde{\mathscr U}$, suppose that $0\le \alpha<nd$, and $t\mapsto\frac{\Psi(t)}{\Phi_i(t)}$ is nondecreasing for any $i=1,\cdots,n$. Let $\sigma_1,\ldots,\sigma_n$ and $v$ be weights. Define
$$[ {\vec{\sigma}},v]_{L_{\vec {\Phi},\Psi}}:=\sup_{Q\in \mathcal Q}\left(\prod_{i=1}^n\Psi\circ\Phi_i^{-1}\left(\frac{1}{\sigma_i(Q)}\right)\right)\left(\int_Q\Psi\left(\mathcal {M}_\alpha(\sigma_1\chi_Q,\ldots,\sigma_n\chi_Q)(x)\right)v(x)dx\right).$$
Then $\mathcal{M}_\alpha(\vec{\sigma}.)$ is bounded from $L^{\Phi_1}(\sigma_1)\times\cdots\times L^{\Phi_n}(\sigma_n)$ to $L^\Psi(v)$ if $[ {\vec{\sigma}},v]_{L_{\vec {\Phi},\Psi}}$ is finite. Moreover,
$$\|\mathcal M_\alpha(\vec{\sigma}.)\|_{\left(\prod_{i=1}^nL^{\Phi_i}(\sigma_i)\right)\rightarrow L^\Psi(v)}\lesssim \Psi^{-1}\left([ {\vec{\sigma}},v]_{L_{\vec {\Phi},\Psi}}\right).$$

\end{theorem}
If we choose $\Psi$ to be a power function, then the above condition is also necessary.
\begin{theorem}\label{thm:LiSunphiq} Let $n$ be a nonnegative integer. Given $\Phi_i\in \mathscr U$, $i=1,2,\cdots,n$, and $1<q<\infty$, suppose that $0\le \alpha<nd$, and that $t\mapsto\frac{t^q}{\Phi_i(t)}$ is nondecreasing for any $i=1,\cdots,n$. Let $\sigma_1,\ldots,\sigma_n$ and $v$ be weights. Define
$$[{\vec{\sigma}},v]_{L_{\vec {\Phi},q}}:=\sup_{Q\in \mathcal Q}\left(\prod_{i=1}^m\left(\Phi_i^{-1}\left(\frac{1}{\sigma_i(Q)}\right)\right)^q\right)\left(\int_Q\left(\mathcal {M}_\alpha(\sigma_1\chi_Q,\ldots,\sigma_n\chi_Q)(x)\right)^qv(x)dx\right).$$
Then $\mathcal M_\alpha(\vec{\sigma}.)$ is bounded from $L^{\Phi_1}(\sigma_1)\times\cdots\times L^{\Phi_n}(\sigma_n)$ to $L^q(v)$ if and only if $[ {\vec{\sigma}},v]_{L_{\vec {\Phi},\Psi}}$ is finite. Moreover,
$$\|\mathcal M_\alpha(\vec{\sigma}.)\|_{\left(\prod_{i=1}^nL^{\Phi_i}(\sigma_i)\right)\rightarrow L^\Psi(v)}\approx [ {\vec{\sigma}},v]_{L_{\vec {\Phi},q}}^{1/q}.$$

\end{theorem}
\vskip .2cm
Let  $0\leq \alpha < nd$, and let $\sigma_{1},...,\sigma_{n},\omega$ be weights on $\mathbb{R}^{d}$. Let  $\Psi,\Phi,\Phi_{1},...,\Phi_{n}$ be one-to-one correspondences from $\mathbb{R}_{+}$ to itself. Set  $\overrightarrow{\Phi}:=(\Phi_{1},...,\Phi_{n})$,  $\overrightarrow{\sigma}:=(\sigma_{1},...,\sigma_{n})$ and
$\nu_{\overrightarrow{\sigma}}:= \frac{1}{\Phi\left(\prod_{i=1}^{n}\Phi_{i}^{-1}\left(\frac{1}{\sigma_{i}}\right)  \right)}$.  

We say the pair of weights $(\overrightarrow{\sigma},\omega)$  belongs to the class 
 $S_{\overrightarrow{\Phi},\Psi}^{\alpha}$ (or $(\overrightarrow{\sigma}, \omega)\in S_{\overrightarrow{\Phi},\Psi}^{\alpha}$), if  
 \begin{equation}\label{eq:deprmim}
   [\overrightarrow{\sigma},\omega]_{S_{\overrightarrow{\Phi},\Psi}^{\alpha}}:=\sup_{Q\in \mathcal{Q}}\Psi \circ \Phi^{-1}\left( \frac{1}{\nu_{\overrightarrow{\sigma}}(Q)}\right)\int_Q\Psi\left(\mathcal{M}_\alpha(\sigma_{1}\chi_Q,...,\sigma_{n}\chi_Q)(x)\right)\omega(x)dx<\infty.
      \end{equation}
When $n=1$, one recovers the definition of the class $S_{\Phi,\Psi}^{\alpha}$ introduced in \cite{jmtafeuseh}.
 Our next result is the following.
\begin{theorem}\label{pro:main126aq}
Let $\sigma_{1},...,\sigma_{n},\omega$ be weights on $\mathbb{R}^{d}$,  and let $\Phi_{1},...,\Phi_{n}\in \widetilde{\mathscr{U}}\cap \nabla_{2}$. Let $\Psi \in \widetilde{\mathscr{U}}$, and let  $\Phi:\mathbb{R}_{+}\longrightarrow\mathbb{R}_{+}$ be a bijection such that 
 $\Phi^{-1}=\Phi_{1}^{-1}\times...\times \Phi_{n}^{-1}$ and $t\mapsto \frac{\Psi(t)}{\Phi(t)}$ is increasing on $\mathbb{R}_{+}^{*}$. If $(\overrightarrow{\sigma}, \omega)\in S_{\overrightarrow{\Phi},\Psi}^{\alpha}$, then there exists a constant $C:=C_{n,\alpha,d,\Psi,\Phi_{1},...,\Phi_{n}}>0$ such that for any  $(f_{1},...,f_{n}) \in L^{\Phi_{1}}(\sigma_{1})\times...\times L^{\Phi_{n}}(\sigma_{n})$, 
 \begin{equation}\label{eq:eaqoaqicz}
 \|\mathcal{M}_{\alpha}(\sigma_{1}f_{1},...,\sigma_{n}f_{n}) \|_{L_{\omega}^{\Psi}}^{lux} \leq C \Psi^{-1}\left([\overrightarrow{\sigma},\omega]_{S_{\overrightarrow{\Phi},\Psi}^{\alpha}}\right)\prod_{i=1}^{n}\|f_{i}\|_{L_{\sigma_{i}}^{\Phi_{i}}}^{lux}. \end{equation}
\end{theorem}

If $\sigma_{1}=\sigma_2=\cdots=\sigma_n=\sigma$, then Theorem \ref{pro:main126aq} becomes the following.

\begin{corollary}\label{pro:main12sasd6}
Let $\omega,\sigma$ be weights on  $\mathbb{R}^{d}$,and let $\Phi_{1},...,\Phi_{n}\in \mathscr{U}\cap \nabla_{2}$. Let $\Psi \in \widetilde{\mathscr{U}}$, and let  $\Phi:\mathbb{R}_{+}\longrightarrow\mathbb{R}_{+}$ be a bijection such that 
 $\Phi^{-1}=\Phi_{1}^{-1}\times...\times \Phi_{n}^{-1}$ and $t\mapsto \frac{\Psi(t)}{\Phi(t)}$ is increasing on $\mathbb{R}_{+}^{*}$. If $(\overrightarrow{\sigma}, \omega)\in S_{\overrightarrow{\Phi},\Psi}^{\alpha}$, Then the following are equivalent. 
\begin{itemize}
\item[(i)]  $(\overrightarrow{\sigma},\omega)\in S_{\overrightarrow{\Phi},\Psi}^{\alpha}$.
\item[(ii)] $\mathcal{M}_{\alpha}(\overrightarrow{\sigma}. ) : L^{\Phi_{1}}(\sigma)\times...\times L^{\Phi_{n}}(\sigma)\longrightarrow L^{\Psi}(\omega)$ boundedly. 
\end{itemize}
Moreover,
\begin{equation}\label{eq:eazeqaqadsoicz}
\|\mathcal{M}_{\alpha}(\overrightarrow{\sigma}. )\|_{\prod_{i=1}^{n}L_{\sigma}^{\Phi_{i}}\rightarrow L_{\omega}^{\Psi}} \approx \Psi^{-1}\left([\overrightarrow{\sigma}, \omega]_{S_{\overrightarrow{\Phi},\Psi}^{\alpha}}\right). \end{equation}
\end{corollary}


\subsection{Some weighted norm estimates for $\mathcal{M}_\alpha$.}

Let  $0\leq \alpha < nd$, and let $\Psi, \Phi_{1},...,\Phi_{n}$  be one-to-one correspondences from $\mathbb{R}_{+}$ to itself. Let $\sigma_{1},...,\sigma_{n},\omega$ be weights on $\mathbb{R}^{d}$. Set $\overrightarrow{\Phi}:=(\Phi_{1},...,\Phi_{n})$ and $\overrightarrow{\sigma}:=(\sigma_{1},...,\sigma_{1})$. 
We say the pair of weights $(\overrightarrow{\sigma}, \omega)$  belongs to the class 
\begin{itemize}
\item[\textbullet] 
$A_{\overrightarrow{\Phi},\Psi}^{\alpha}$ (or $(\overrightarrow{\sigma}, \omega) \in A_{\overrightarrow{\Phi},\Psi}^{\alpha}$), if  
\begin{equation}\label{eq:de5m45im}
 [\overrightarrow{\sigma},\omega]_{A_{\overrightarrow{\Phi},\Psi}^{\alpha}}:=\sup_{Q\in \mathcal{Q}}\omega(Q)\Psi\left(\prod_{i=1}^{n}\frac{\sigma_{i}(Q)}{|Q|^{1-\frac{\alpha}{nd}}}\right)\Psi \left(\prod_{i=1}^{n}\Phi_{i}^{-1}\left(\frac{1}{\sigma_{i}(Q)}\right)  \right)  <\infty,   \end{equation}
\item[\textbullet] 
$\widetilde{A}_{\overrightarrow{\Phi},\Psi}^{\alpha}$ (or $(\overrightarrow{\sigma}, \omega) \in \widetilde{A}_{\overrightarrow{\Phi},\Psi}^{\alpha}$), if  
\begin{equation}\label{eq:de45im}
 [\overrightarrow{\sigma},\omega]_{\widetilde{A}_{\overrightarrow{\Phi},\Psi}^{\alpha}}:=\sup_{Q\in \mathcal{Q}}\frac{\omega(Q)}{|Q|}\Psi\left(|Q|^{\frac{\alpha}{d}}\prod_{i=1}^{n}\Phi_{i}^{-1}\left(\frac{\sigma_{i}(Q)}{|Q|}\right)\right)<\infty,   \end{equation}
\item[\textbullet] 
$B_{\overrightarrow{\Phi},\Psi}^{\alpha}$ (or $(\overrightarrow{\sigma}, \omega) \in B_{\overrightarrow{\Phi},\Psi}^{\alpha}$), if
\begin{equation}\label{eq:de65665m}
 [\overrightarrow{\sigma},\omega]_{B_{\overrightarrow{\Phi},\Psi}^{\alpha}}:=
 \sup_{Q\in \mathcal{Q}}\omega(Q)\Psi\left(\prod_{i=1}^{n}  \frac{\sigma_{i}(Q)}{|Q|^{1-\frac{\alpha}{nd}}}\right) \Psi\left(\prod_{i=1}^{n}\Phi_{i}^{-1}\left(\frac{1}{|Q|}\exp\left(\frac{1}{|Q|}\int_{Q}\log\frac{1}{\sigma_{i}}\right)\right)\right) < \infty.    \end{equation}
\end{itemize}

When $n=1$, one recovers the definitions of the classes $A_{\Phi,\Psi}^{\alpha}$, $\widetilde{A}_{\Phi,\Psi}^{\alpha}$ and $B_{\Phi,\Psi}^{\alpha}$ introduced in \cite{jmtafeuseh}.
\vskip .2cm
We also observe that if we take $\Psi(t)=t^q$, $\Phi_i(t)=t^{p_i}$, and define $\frac 1p=\frac 1{p_1}+\cdots+\frac 1{p_n}$ and $\vec {P}=(p_1,\cdots,p_n)$, then $A_{\overrightarrow{\Phi},\Psi}^{\alpha}$ and $B_{\overrightarrow{\Phi},\Psi}^{\alpha}$ respectively coincide with the classes $A_{\vec{P},q}$ and $B_{\vec{P},q}$ introduced in \cite{sehbaf}.
\vskip .2cm
We recall that a weight $\omega$ is said to belongs to the class $\mathcal{A}_\infty^{exp}$ of S. V. Hurscev (see \cite{Hurscev}), if $$[\omega]_{\mathcal{A}_\infty^{exp}}:=\sup_{Q\in \mathcal{Q}}\frac{\omega(Q)}{|Q|}\exp\left(\frac{1}{|Q|}\int_Q\log\left(\frac{1}{\omega(x)}\right)dx\right)<\infty.$$
We make the following observations.
\begin{itemize}
\item[(a)] If for $1\leq i \leq n$, $\Phi_{i} \in \widetilde{\mathscr{U}}$, then

\begin{equation}\label{eq:dr56455im}
 [\overrightarrow{\sigma},\omega]_{A_{\overrightarrow{\Phi},\Psi}^{\alpha}} \lesssim [\overrightarrow{\sigma},\omega]_{S_{\overrightarrow{\Phi},\Psi}^{\alpha}}.
 \end{equation}
 
 \item[(b)] If for $1\leq i \leq n$,  $\Phi_{i}$ is convex, and $\sigma_{i} \in \mathcal{A}_\infty^{exp}$ and if  $\Psi \in \mathscr{U}^{q}$ with $q\geq 1$, then 
  
 \begin{equation}\label{eq:dep1r56455im}
  [\overrightarrow{\sigma},\omega]_{A_{\overrightarrow{\Phi},\Psi}^{\alpha}} \leq [\overrightarrow{\sigma},\omega]_{B_{\overrightarrow{\Phi},\Psi}^{\alpha}} \leq C_{q} \left(\prod_{i=1}^{n}[\sigma_{i}]_{\mathcal{A}_\infty^{exp}}\right)^{q} [\overrightarrow{\sigma},\omega]_{A_{\overrightarrow{\Phi},\Psi}^{\alpha}}.
  \end{equation}
 \end{itemize} 
 We say a weight $\omega$ belongs to the Fuji-Wilson class $\mathcal{A}_\infty$ (see \cite{Fujii}), if 
$$
[\omega]_{\mathcal{A}_\infty}:=\sup_{Q\in \mathcal{Q}}\frac{1}{\omega(Q)}\int_Q\mathcal{M}(\sigma\chi_Q(x)))dx<\infty.
$$
Let  $\sigma_{1},...,\sigma_{n}$ be weights on $\mathbb{R}^{d}$ and let   $\Psi,\Phi,\Phi_{1},...,\Phi_{n}$ be one-to-one correspondences from $\mathbb{R}_{+}$ to itself. Put  $\nu_{\overrightarrow{\sigma}}:= \frac{1}{\Phi\left(\prod_{i=1}^{n}\Phi_{i}^{-1}\left(\frac{1}{\sigma_{i}}\right)  \right)}$ and $\overrightarrow{\Phi}:=(\Phi_{1},...,\Phi_{n})$.  We say  $\overrightarrow{\sigma}$ belongs to the class
$W_{\overrightarrow{\Phi},\Psi}$ (or $\overrightarrow{\sigma} \in W_{\overrightarrow{\Phi},\Psi}$), if  
\begin{equation}\label{eq:deaqwx45im}
 [\overrightarrow{\sigma}]_{W_{\overrightarrow{\Phi},\Psi}}:=\sup_{Q\in \mathcal{Q}}\Psi \circ \Phi^{-1}\left( \frac{1}{\nu_{\overrightarrow{\sigma}}(Q)}\right)\int_{Q}\Psi\left(\prod_{i=1}^{n} \Phi_{i}^{-1}\left(\mathcal{M}(\sigma_{i}\chi_{Q})(x)\right)\right)dx<\infty.   \end{equation}
 
Recall that if $\Phi$ is an increasing convex function, then the complementary function of $\Phi$ is the function  ${\Psi}$ defined on $\mathbb{R}_{+}$  by 
  \begin{equation}\label{eq:in45ale1}	{\Psi}(s)=\sup_{t\geq 0}\{st-\Phi(t) \}, ~ \forall~  s \geq 0.   \end{equation}
We note that ${\Psi}$ is also an increasing convex function whose complementary function is $\Phi$ (see for example \cite{raoren,rao68ren}).

We obtain the following estimates.
\begin{theorem}\label{pro:main1aq6}
 Let $\sigma_{1},...,\sigma_{n},\omega$ be weights on $\mathbb{R}^{d}$. Let $\Phi_{1},...,\Phi_{n} \in \widetilde{\mathscr{U}}\cap \nabla_{2}$, $\Psi \in \widetilde{\mathscr{U}}$, and let  $\Phi$ a one-to-one correspondence from $\mathbb{R}_{+}$ to itself such that
 $\Phi^{-1}=\Phi_{1}^{-1}\times...\times \Phi_{n}^{-1}$ and $t\mapsto \frac{\Psi(t)}{\Phi(t)}$ is increasing on $\mathbb{R}_{+}^{*}$. Then the following are satisfied.
\begin{itemize}
\item[(i)] If  $(\overrightarrow{\sigma},\omega) \in B_{\overrightarrow{\Phi},\Psi}^{\alpha}$, then there exists a constant $C_{1}:=C_{n,\Psi,\Phi_{1},...,\Phi_{n}}>0$ such that for any  $(f_{1},...,f_{n}) \in L^{\Phi_{1}}(\sigma_{1})\times...\times L^{\Phi_{n}}(\sigma_{n})$, 
\begin{equation}\label{eq:eazaaqoicz}
\|\mathcal{M}_{\alpha}(\sigma_{1}f_{1},...,\sigma_{n}f_{n}) \|_{L_{\omega}^{\Psi}}^{lux} \leq C_{1}\Psi^{-1}\left([\overrightarrow{\sigma},\omega]_{B_{\overrightarrow{\Phi},\Psi}^{\alpha}}\right)\prod_{i=1}^{n}\|f_{i}\|_{L_{\sigma_{i}}^{\Phi_{i}}}^{lux}. \end{equation}
\item[(ii)] If  $(\overrightarrow{\sigma},\omega) \in A_{\overrightarrow{\Phi},\Psi}^{\alpha}$ and for $1\leq i \leq n$, $\sigma_{i} \in \mathcal{A}_\infty$, then there exists a constant $C_{2}:=C_{n,\Psi,\Phi_{1},...,\Phi_{n}}>0$ such that for any  $(f_{1},...,f_{n}) \in L^{\Phi_{1}}(\sigma_{1})\times...\times L^{\Phi_{n}}(\sigma_{n})$, 
\begin{equation}\label{eq:eazesaqaqoicz}
\|\mathcal{M}_{\alpha}(\sigma_{1}f_{1},...,\sigma_{n}f_{n}) \|_{L_{\omega}^{\Psi}}^{lux} \leq C_{2} \Psi^{-1}\left([\overrightarrow{\sigma},\omega]_{A_{\overrightarrow{\Phi},\Psi}^{\alpha}}\sum_{i=1}^{n}\Psi\circ\Phi^{-1}\left([\sigma_{i}]_{\mathcal{A}_\infty}\right)\right)\prod_{i=1}^{n}\|f_{i}\|_{L_{\sigma_{i}}^{\Phi_{i}}}^{lux}. \end{equation}
\item[(iii)] If $(\overrightarrow{\sigma},\omega) \in \widetilde{A}_{\overrightarrow{\varphi},\Psi}^{\alpha}$ and $\overrightarrow{\sigma} \in W_{\overrightarrow{\Phi},\Psi}$, then there exists a constant  $C_{3}:=C_{n,\Psi,\Phi_{1},...,\Phi_{n}}>0$ such that for any  $(f_{1},...,f_{n}) \in L^{\Phi_{1}}(\sigma_{1})\times...\times L^{\Phi_{n}}(\sigma_{n})$,
\begin{equation}\label{eq:eaiaqcz}
 \|\mathcal{M}_{\alpha}(\sigma_{1}f_{1},...,\sigma_{n}f_{n}) \|_{L_{\omega}^{\Psi}}^{lux} \leq C_{3} \Psi^{-1}\left([\overrightarrow{\sigma}]_{W_{\overrightarrow{\Phi},\Psi}}[\overrightarrow{\sigma},\omega]_{\widetilde{A}_{\overrightarrow{\varphi},\Psi}^{\alpha}}\right)\prod_{i=1}^{n}\|f_{i}\|_{L_{\sigma_{i}}^{\Phi_{i}}}^{lux}, \end{equation} 
 where $\overrightarrow{\varphi}:=(\varphi_{1},...,\varphi_{n})$ with  $\varphi_{i}$ the complementary function of $\Phi_{i}$.
\end{itemize}
\end{theorem}

We also have the following result.

\begin{theorem}\label{pro:mainineq4}
 Let $\sigma_{1},...,\sigma_{n},\omega$ be weights on $\mathbb{R}^{d}$. Let $\Phi_{1},...,\Phi_{n} \in \widetilde{\mathscr{U}}\cap \nabla_{2}$, $\Psi \in \widetilde{\mathscr{U}}$, and let  $\Phi$ a one-to-one correspondence from $\mathbb{R}_{+}$ to itself such that
 $\Phi^{-1}=\Phi_{1}^{-1}\times...\times \Phi_{n}^{-1}$ and $t\mapsto \frac{\Psi(t)}{\Phi_i(t)}$  is increasing on $\mathbb{R}_{+}^{*}$ for each $i\in\{1,2,\cdots,n\}$.
If  $(\overrightarrow{\sigma},\omega) \in A_{\overrightarrow{\Phi},\Psi}^{\alpha}$ and for $1\leq i \leq n$, $\sigma_{i} \in \mathcal{A}_\infty$, then there exists a constant $C:=C_{n,\Psi,\Phi_{1},...,\Phi_{n}}>0$ such that for any  $(f_{1},...,f_{n}) \in L^{\Phi_{1}}(\sigma_{1})\times...\times L^{\Phi_{n}}(\sigma_{n})$, 
\begin{equation}\label{eq:eazesaqaqoicz1}
\|\mathcal{M}_{\alpha}(\sigma_{1}f_{1},...,\sigma_{n}f_{n}) \|_{L_{\omega}^{\Psi}}^{lux} \leq C \Psi^{-1}\left([\overrightarrow{\sigma},\omega]_{A_{\overrightarrow{\Phi},\Psi}^{\alpha}}\prod_{i=1}^{n}[\sigma_{i}]_{\mathcal{A}_\infty}^{q_i}\right)\prod_{i=1}^{n}\|f_{i}\|_{L_{\sigma_{i}}^{\Phi_{i}}}^{lux}, \end{equation}
where $q_i$ is the upper-type of $\Psi\circ\Phi_i^{-1}$.
\end{theorem}
\section{Some definitions and useful properties}

\subsection{Some properties of growth functions.}
Let $\Phi$ be a growth function. We said $\Phi$ satisfies the condition $ \Delta^{\prime\prime}$ (or $\Phi \in \Delta^{\prime\prime}$), if there exists a constant $C > 0$ such that
 \begin{equation}\label{eq:deprprimim}
  \Phi(t)\Phi(s) \leq  \Phi(Cst),~~\forall~s,t > 0. \end{equation} 
Let $\Phi$ be a growth function. If $\Phi$ is strictly increasing on $\mathbb{R}_{+}$, then  $\Phi \in \Delta^{\prime}$ if and only if $\Phi^{-1} \in \Delta^{\prime\prime}$.

\begin{proposition}\label{pro:main55}
Let $\Phi_{1},...,\Phi_{n}$ be strictly increasing growth functions on $\mathbb{R}_{+}$. Let $\Phi$ be a one-to-one correspondence from $\mathbb{R}_{+}$ to itself such that 
$\Phi^{-1}=\Phi_{1}^{-1}\times...\times \Phi_{n}^{-1}.$ If for all $1\leq i \leq n$,  $\Phi_{i}\in \Delta^{\prime}$, then $\Phi \in \Delta^{\prime}$.
\end{proposition}

\begin{proof}
We prove the result only in the case $n=2$ as the general case follows the same way. We have 
\begin{align*}
\Phi_{1},\Phi_{2} \in \Delta^{\prime} &\Rightarrow
\Phi_{1}^{-1},\Phi_{2}^{-1} \in \Delta^{\prime\prime}\\ &\Rightarrow \Phi_{1}^{-1}(s)\Phi_{1}^{-1}(t)\leq \Phi_{1}^{-1}(C_{1}st) \hspace*{0.25cm} \text{and} \hspace*{0.25cm}  \Phi_{2}^{-1}(s)\Phi_{2}^{-1}(t) \leq \Phi_{2}^{-1}(C_{2}st),~~ \forall~ s,t > 0\\
&\Rightarrow \Phi_{1}^{-1}(s)\Phi_{2}^{-1}(s)\Phi_{1}^{-1}(t)\Phi_{2}^{-1}(t) \leq \Phi_{1}^{-1}(Cst)\Phi_{2}^{-1}(Cst),~~ \forall~ s,t > 0~ \text{with}~ C=\max\{C_{1},C_{2}\} \\
&\Rightarrow \Phi^{-1}(s)\Phi^{-1}(t) \leq \Phi^{-1}(Cst),~~ \forall~ s,t > 0\\
&\Rightarrow \Phi^{-1} \in \Delta^{\prime\prime}\\
&\Rightarrow \Phi \in \Delta^{\prime}.
\end{align*}
\end{proof}

We note that the converse of Proposition \ref{pro:main55} does not always hold. Indeed, if we consider the functions  $\Phi$ and $\Psi$ defined on $\mathbb{R}_{+}$ by  $$  \Phi(t)=\exp(t)-t-1 ~\hspace*{0.25cm} \text{and} \hspace*{0.25cm} ~\Psi(t)=(1+t)\ln(1+t)-t,~ \forall~ t \geq 0,                       $$
then    $$  \Phi^{-1}(t)\Psi^{-1}(t) \approx t,~~ \forall~ t>0.            $$
The function $t\mapsto t$  satisfies the condition $\Delta'$ while $\Phi \not\in \Delta'$. Indeed, 
$$  \lim_{t\to +\infty}\frac{\Phi(2t)}{\Phi(t)} =    \lim_{t\to +\infty}\frac{\exp(2t)-2t-1}{\exp(t)-t-1} =\lim_{t\to +\infty} \frac{\left(\exp(t)-\frac{2t}{\exp(t)}-\frac{1}{\exp(t)} \right)}{\left(1-\frac{t}{\exp(t)}-\frac{1}{\exp(t)} \right)}=\infty.             $$

\begin{lemma}\label{pro:main40}
Let $\Psi,\Phi_{1},...,\Phi_{n} \in \mathscr{U}$, and let $\Phi$ be a one-to-one correspondence from $\mathbb{R}_{+}$ to itself such that
$\Phi^{-1}=\Phi_{1}^{-1}\times...\times \Phi_{n}^{-1}$ and $t\mapsto\frac{\Psi(t)}{\Phi(t)}$ is incresing on $\mathbb{R}_{+}^{*}$. Let $\widetilde{\Omega}_{3}$ be the function defined on $\mathbb{R}_{+}$ by
\begin{equation}\label{eq:su556i8n}
\widetilde{\Omega}_{3}(t)=\frac{1}{\Psi\circ \Phi^{-1}\left(\frac{1}{t}\right)}, ~\forall~t>0 ~\hspace*{0.25cm}  ~\text{and}~\hspace*{0.25cm} ~  \widetilde{\Omega}_{3}(0)=0.\end{equation}
If  $\Psi \in \mathscr{U}^{q}$ with $q\geq 1$, then
 $\Psi\circ \Phi^{-1}$ and $\widetilde{\Omega}_{3}$ belong to $\mathscr{U}^{nq}$.
\end{lemma}

 \begin{proof}
As  $t\mapsto\frac{\Psi(t)}{\Phi(t)}$ is increasing on $\mathbb{R}_{+}^{*}$, we deduce that the functions $t\mapsto\frac{\Psi\circ \Phi^{-1}(t)}{t}$ and $t\mapsto\frac{\widetilde{\Omega}_{3}(t)}{t}$ are also increasing on $\mathbb{R}_{+}^{*}$.
\vskip .2cm
Let $1\leq i \leq n$. As  $t\mapsto\frac{ \Phi_{i}(t)}{t}$ is increasing on $\mathbb{R}_{+}^{*}$, it follows that $t\mapsto \frac{\Phi_{i}^{-1}(t)}{t}$ is decreasing on $\mathbb{R}_{+}^{*}$. Consequently, $t\mapsto \frac{\Phi^{-1}(t)}{t^{n}}$ is decreasing on $\mathbb{R}_{+}^{*}$. Indeed, given $0<t_{1}\leq t_{2}$, we have $$ \frac{\Phi^{-1}(t_{2})}{t_{2}^{n}}=\frac{\Phi_{1}^{-1}(t_{2})}{t_{2}}\times...\times \frac{\Phi_{n}^{-1}(t_{2})}{t_{2}} \leq  \frac{\Phi_{1}^{-1}(t_{1})}{t_{1}}\times...\times \frac{\Phi_{n}^{-1}(t_{1})}{t_{1}} = \frac{\Phi^{-1}(t_{1})}{t_{1}^{n}}.     $$
Let $s> 0$ and $t \geq 1$. We have $$  \frac{\Phi^{-1}(st)}{(st)^{n}} \leq \frac{\Phi^{-1}(s)}{s^{n}}\Rightarrow \Phi^{-1}(st) \leq t^{n}\Phi^{-1}(s). $$
Since $\Psi$ is of upper-type $q$, we obtain 
$$ \Psi\left(\Phi^{-1}(st)\right) \leq \Psi\left(t^{n}\Phi^{-1}(s)\right)\leq C_{q}t^{nq}\Psi\left(\Phi^{-1}(s)\right).$$
We then deduce that $\Psi\circ \Phi^{-1} \in \mathscr{U}^{nq}$. Let us check that $\widetilde{\Omega}_{3} \in \mathscr{U}^{nq}$. We have 
\begin{align*}
\Psi\circ\Phi^{-1}\left(\frac{1}{s}\right)\leq C_{q}t^{nq}\Psi\circ \Phi^{-1}\left(\frac{1}{st}\right) &\Rightarrow
\frac{1}{C_{q}t^{nq}\Psi\circ \Phi^{-1}\left(\frac{1}{st}\right)} \leq  \frac{1}{\Psi\circ \Phi^{-1}\left(\frac{1}{s}\right)} \\
&\Rightarrow  \widetilde{\Omega}_{3}(st) \leq C_{q}t^{nq}\widetilde{\Omega}_{3}(s).
\end{align*}
\end{proof}

\begin{lemma}\label{pro:main80jdf}
Let $\Phi\in \widetilde{\mathscr{U}}$.There exists a constant $C:=C_{\Phi}>0$ such that  \begin{equation}\label{eq:ibetdfide1}
\Phi\left(\frac{s}{t}\right)\leq C\frac{\Phi(s)}{\Phi(t)}, ~~ \forall~s,t >0 
.\end{equation}
\end{lemma}

\begin{proof}
Let $q\geq 1$ be such that $\Phi\in \widetilde{\mathscr{U}}^{q}$. Let $t >0$. We consider the two cases $0<t<1$ and $t\geq 1$.
\vskip .2cm
$ \textit{Case 1}:$ Assume $0<t<1$. Then using inequality (\ref{eq:ibide2}), we obtain  
$$  \Phi(t)= \Phi\left(\frac{1}{\frac{1}{t} } \right) \leq C_{1} \frac{\Phi(1)}{\left(\frac{1}{t} \right)^{q}} = C_{1}\Phi(1)t^{q}\Rightarrow \frac{1}{t^{q}} \leq \frac{C_{1}\Phi(1)}{\Phi(t)}.            $$
As $\Phi$ is of upper-type $q$, we have  
$$  \Phi\left(\frac{1}{t}\right)=\Phi\left(\frac{1}{t}\times 1\right) \leq C_{2} \frac{1}{t^{q}}\Phi(1)\Rightarrow \Phi\left(\frac{1}{t}\right) \leq  \frac{C_{1}C_{2}\Phi(1)^{2}}{\Phi(t)}.               $$
$\textit{Cas 2}:$ Assume that $t \geq 1$. Then as $\Phi$ is of upper-type $q$, we obtain
$$  \Phi(t)=\Phi(t\times 1) \leq C_{2}t^{q}\Phi(1) \Rightarrow \frac{1}{t^{q}} \leq \frac{C_{2}\Phi(1)}{\Phi(t)}.         $$
Using  (\ref{eq:ibide2}), we easily obtain
$$  \Phi\left(\frac{1}{t}\right) \leq C_{1} \frac{\Phi(1)}{t^{q}}\Rightarrow  \Phi\left(\frac{1}{t}\right) \leq  \frac{C_{1}C_{2}\Phi(1)^{2}}{\Phi(t)}.       $$
Taking $C_{3}:=C_{1}C_{2}\Phi(1)^{2}$, we deduce from the above analysis that 
 $  \Phi\left(\frac{1}{t}\right) \leq  \frac{C_{3}}{\Phi(t)},$ for any $t >0$. 
 \vskip .1cm
As $\Phi\in \Delta'$, we finally obtain $$ \Phi\left(\frac{s}{t}\right)\leq  C\Phi(s) \Phi\left(\frac{1}{t}\right) \leq CC_{3} \frac{\Phi(s)}{\Phi(t)},~~ \forall~s,t >0.       $$ 
 \end{proof}
\subsection{Dyadic grids and sparse families}
Recall that the standard dyadic grid $\mathcal{D}$ in $\mathbb{R}^d$ is the collection of all cubes of the form $$2^{-k}\left([0,1)^d+m\right),\quad k\in \mathbb{Z}, m\in \mathbb{Z}^d.$$
We also recall the following definition of a general dyadic grid.
\begin{definition}
In general, a dyadic grid $\mathcal{D}^\beta$ in $\mathbb{R}^d$ is any collection of cubes such that
\begin{itemize}
\item[(i)] the sidelength $\ell Q$ of any cube $Q\in \mathcal{D}^\beta$ is $2^k$ for some $k\in \mathbb{Z}$;
\item[(ii)] for $Q,Q'\in \mathcal{D}^\beta$, $Q\cap Q'\in\{Q,Q',\emptyset\}$;
\item[(iii)] for each $k\in \mathbb{Z}$, the family $\mathcal{D}_k^\beta:=\{Q\in \mathcal{D}^\beta:\ell Q=2^k\}$ forms a partition of $\mathbb{R}^d$.
\end{itemize}
\end{definition}
\begin{definition}
We say a collection of dyadic cubes $\mathcal{S}^\beta=\{Q_{j,k}\}_{j,k\in\mathbb{Z}}\subset \mathcal{D}^\beta$ is a sparse family if 
\begin{itemize}
\item[(i)] for each fixed $k$, the family $\{Q_{j,k}\}_{j\in \mathbb{Z}}$ is pairwise disjoint;
\item[(ii)] if $A_k=\cup_{j\in \mathbb{Z}}Q_{j,k}$, then $A_{k+1}\subset A_k$.
\item[(iii)] $|A_{k+1}\cap Q_{j.k}|\le \frac 12|Q_{j,k}|$.
\end{itemize}
\end{definition}
Given a sparse family $\mathcal{S}^\beta=\{Q_{j,k}\}_{j,k\in\mathbb{Z}}\subset \mathcal{D}^\beta$, for each $Q_{j,k}\in \mathcal{S}^\beta$, define $E_{Q_{j,k}}:=Q_{j,k}\setminus A_{k+1}$.
Then the sets in the family $\{E_Q\}_{Q\in \mathcal{S}^\beta}$ are pairwise disjoint.
\vskip .2cm
We refer to \cite{HytPerez} for the following.
\begin{lemma}\label{pro:main80}
There are $2^d$ dyadic grids $\mathcal{D}^{\beta}$ such that for any cube $Q\in \mathbb{R}^d$, there exists a cube $R\in \mathcal{D}^\beta$ for some $\beta$ such that $Q\subset R$ and $\ell R\le 6\ell Q$.
\end{lemma}
The $2^d$ dyadic grids in our case will be considered by taking $\beta\in\{0,\frac 13\}^d$.
\subsection{Some properties of weighted Orlicz spaces.}

Let $q\geq 1$ such that  $\Phi \in \mathscr{U}^{q}$. For $\sigma$ a weight on $\mathbb R^d$, let us define on $L^{\Phi}(\sigma)$ the following quantity 
\begin{equation}\label{eq:foanmaxi85mal}
\|f\|_{L_{\sigma}^{\Phi}}:=\int_{\mathbb{R}^{d}}\Phi\left(|f(x)|\right)\sigma(x)dx.\end{equation}
We note that the mapping $f\mapsto \|f\|_{L_{\sigma}^{\Phi}}$ is not in general a norm on $L^{\Phi}(\sigma)$ (see for example \cite{raoren,rao68ren}).
\vskip .1cm
We refer to \cite{djesehb} for the following inequalities between $\|\cdot\|_{L_{\sigma}^{\Phi}}$ and $\|\cdot\|_{L_{\sigma}^{\Phi}}^{lux}$.
\begin{equation}\label{eq:condgitxxedinis}
 \|f\|_{L_{\sigma}^{\Phi}} \lesssim \max\left\{ \|f\|_{L_{\sigma}^{\Phi}}^{lux} ; \left( \|f\|_{L_{\sigma}^{\Phi}}^{lux}\right)^{q}\right\} ~ \text{and}~ \|f\|_{L_{\sigma}^{\Phi}}^{lux} \lesssim \max\left\{ \|f\|_{L_{\sigma}^{\Phi}} ; \left( \|f\|_{L_{\sigma}^{\Phi}}\right)^{\frac{1}{q}}\right\}.  \end{equation}
\vskip .1cm
If in the definition of the weighted maximal function $\mathcal{M}_\sigma$, we restrict the supremum to dyadic cubes in the grid $\mathcal{D}^{\beta}$, we obtain the dyadic weighted maximal function denoted  $\mathcal{M}_{\sigma}^{\mathcal{D}^{\beta}}$.
 \vskip .1cm
 Let us recall the following (see \cite[Theorem 3.4]{jmtafeuseh}).
\begin{theorem}\label{pro:main88}
Let $\Phi\in \mathscr{U}$ and $\sigma$ a weight on $\mathbb{R}^{d}$.
If $\Phi\in \nabla_2$, then there exists a constant $C:=C_{\Phi}>0$ such that for any $f\in L^{\Phi}(\sigma)$,   
\begin{equation}\label{eq:foncinnmaximal}
\int_{\mathbb{R}^{d}}\Phi\left( \mathcal{M}_{\sigma}^{\mathcal{D}^{\beta}}(f)(x)\right)\sigma(x)dx \leq C \int_{\mathbb{R}^{d}}\Phi\left(|f(x)|\right)\sigma(x)dx
.\end{equation}
\end{theorem}

Kokilashvili and Krbec gave in \cite{kokokrbec} a necessary and sufficient condition for the Hardy-Littlewood maximal function $\mathcal{M}$ to be bounded on a Orlicz space. More precisely, they proved that $\mathcal{M}:L^{\Phi}\longrightarrow L^{\Phi}$ boundedly if and only if $\Phi\in \nabla_2$. The condition $\nabla_2$ is then relevant in Theorem \ref{pro:main88}. 
\vskip .2cm
 Let  $\lambda > 0$ and $\overrightarrow{f}:=(f_{1},...,f_{n})\in L^{\Phi_{1}}(\sigma_{1})\times...\times L^{\Phi_{n}}(\sigma_{n})$. We denote by   $\mathcal{D}_{\lambda}^{*}(\overrightarrow{f})$  the set of all maximal dyadic cubes  $Q$ in $\mathcal{D}^{\beta}$ with respect to the inclusion satisfying the condition
\begin{equation}\label{eq:foncinimal}
\prod_{i=1}^{n}\frac{1}{\sigma_{i}(Q)}\int_{Q}|f_{i}(y)|\sigma_{i}(y) dy > \lambda.\end{equation}
The cubes in $\mathcal{D}_{\lambda}^{*}(\overrightarrow{f})$ are pairwise disjoint and moreover,
\begin{equation}\label{eq:fonciaqal}
\left\{x\in \mathbb{R}^{d} : \mathcal{M}_{\overrightarrow{\sigma}}^{\mathcal{D}^{\beta}}(\overrightarrow{f})(x)> \lambda \right\}=\bigcup_{Q \in \mathcal{D}_{\lambda}^{*}(\overrightarrow{f})}Q.\end{equation}
The following estimates are useful for our purpose.
\begin{proposition}\label{pro:main0aqk5}
Let  $\sigma_{1},...,\sigma_{n}$ be weights on $\mathbb{R}^{d}$. Let $\Phi_{1},...,\Phi_{n}\in \widetilde{\mathscr{U}}$,    and  let $\Phi$ be a one-to-one correspondence from  $\mathbb{R}_{+}$ to itself such that $\Phi^{-1}=\Phi_{1}^{-1}\times...\times \Phi_{n}^{-1}$. Define  $\nu_{\overrightarrow{\sigma}}:= \frac{1}{\Phi\left(\prod_{i=1}^{n}\Phi_{i}^{-1}\left(\frac{1}{\sigma_{i}}\right)  \right)}$. Then the following are satisfied.  
\begin{itemize}
\item[(i)] There exists a constant $C_{1}:=C_{n,\Phi, \Phi_{1},...,\Phi_{n}} >0$ such that for any $\lambda > 0$ and for any $1\leq i \leq n$,   $0\not\equiv f_{i}  \in L^{\Phi_{i}}(\sigma_{i})$, 
\begin{equation}\label{eq:foncijhimal}
\nu_{\overrightarrow{\sigma}}\left(\left\{x\in \mathbb{R}^{d} : \mathcal{M}_{\overrightarrow{\sigma}}^{\mathcal{D}^{\beta}}\left(\frac{f_{1}}{\|f_{1}\|_{L_{\sigma_{1}}^{\Phi_{1}}}^{lux}},...,\frac{f_{n}}{\|f_{n}\|_{L_{\sigma_{n}}^{\Phi_{n}}}^{lux}}\right)(x)> \lambda \right\}\right) \leq  \dfrac{C_{1}}{\Phi(\lambda)}.\end{equation}
\item[(ii)] Let $\Psi\in \mathscr{U}$. If the function $t\mapsto \frac{\Psi(t)}{\Phi(t)}$ is increasing on $\mathbb{R}_{+}^{*}$, then there exists a constant $C_{2}:= C_{n,\Psi,\Phi, \Phi_{1},...,\Phi_{n}} > 0$ such that for any $\lambda > 0$ and any   $(f_{1},...,f_{n}) \in L^{\Phi_{1}}(\sigma_{1})\times...\times
 L^{\Phi_{n}}(\sigma_{n})$ with  $f_{i} \not\equiv 0$, for $i=1,...,n$, 
\begin{equation}\label{eq:ineg56a3jsen}
\sum_{R \in \mathcal{D}_{\lambda}^{*}(\overrightarrow{g})} \frac{1}{\Psi\circ \Phi^{-1}\left(\frac{1}{\nu_{\overrightarrow{\sigma}}(R)}\right)} \leq C_{2} \frac{\Phi(\lambda)}{\Psi(\lambda)} \nu_{\overrightarrow{\sigma}}\left( \left\{ x\in \mathbb{R}^{d} : \mathcal{M}^{\mathcal{D}^{\beta}}_{\overrightarrow{\sigma}}(\overrightarrow{g})(x) >  \lambda \right\} \right),
\end{equation} 
where  $\overrightarrow{g}:=(g_{1},...,g_{n})$, with $g_{i}=\frac{f_{i}}{\|f_{i}\|_{L_{\sigma_{i}}^{\Phi_{i}}}^{lux}}.$
\item[(iii)] If for any $1\leq i \leq n$,  $\Phi_{i} \in \nabla_{2}$, then there exists a constant $C_{3}:= C_{\Phi, \Phi_{1},...,\Phi_{n}} > 0$ such that for any $\overrightarrow{f}:=(f_{1},...,f_{n}) \in L^{\Phi_{1}}(\sigma_{1})\times...\times
 L^{\Phi_{n}}(\sigma_{n})$,  
\begin{equation}\label{eq:foaxiqdmal}
\int_{\mathbb{R}^{d}}\Phi\left(\mathcal{M}_{\overrightarrow{\sigma}}^{\mathcal{D}^{\beta}}(\overrightarrow{f})(x)\right)\nu_{\overrightarrow{\sigma}}(x) dx \leq C_{3} \sum_{i=1}^{n} \int_{\mathbb{R}^{d}}\Phi_{i}(|f_{i}(x)|)\sigma_{i}(x)dx.
\end{equation}
\end{itemize}
\end{proposition}

\begin{proof}
$i)$ Let  $\lambda > 0$, and $(f_{1},...,f_{n})\in L^{\Phi_{1}}(\sigma_{1})\times...\times L^{\Phi_{n}}(\sigma_{n})$. Suppose that for any $1\leq i \leq n$,  $f_{i}\not\equiv 0$.  Put $\overrightarrow{g}:=\left(g_{1},...,g_{n}\right)$, where $g_{i}:=\frac{f_{i}}{\|f_{i}\|_{L_{\sigma_{i}}^{\Phi_{i}}}^{lux}},$ for $1\leq i \leq n$. Let  $\mathcal{D}_{\lambda}^{*}(\overrightarrow{g})$ be the set of all maximal cubes $Q$ in $\mathcal D^\beta$ with respect to the inclusion such that 
$$  \prod_{i=1}^{n}\frac{1}{\sigma_{i}(Q)}\int_{Q}|g_{i}(y)|\sigma_{i}(y) dy > \lambda.    $$
We have from (\ref{eq:fonciaqal}) that
$$
E_{\lambda}:=\left\{x\in \mathbb{R}^{d} : \mathcal{M}_{\overrightarrow{\sigma}}^{\mathcal{D}^{\beta}}(\overrightarrow{g})(x)> \lambda \right\}=\bigcup_{Q \in \mathcal{D}_{\lambda}^{*}(\overrightarrow{g})}Q.  $$
Let $Q \in \mathcal{D}_{\lambda}^{*}(\overrightarrow{g})$. Using Proposition \ref{pro:main55} and Lemma \ref{pro:main80jdf}, we obtain
\begin{align*}
\Phi(\lambda)\nu_{\overrightarrow{\sigma}}(Q)
&\lesssim   \int_{Q}\Phi\left(\prod_{i=1}^{n}\frac{\Phi_{i}^{-1}\left(\frac{1}{\sigma_{i}(x)}\right)\times \frac{1}{\sigma_{i}(Q)}\int_{Q}|g_{i}(y)|\sigma_{i}(y) dy}{\Phi_{i}^{-1}\left(\frac{1}{\sigma_{i}(x)}\right)} \right)\nu_{\overrightarrow{\sigma}}(x)dx\\
& \lesssim \int_{Q} \Phi\left(\prod_{i=1}^{n}\Phi_{i}^{-1}\left(\frac{1}{\sigma_{i}(x)}\right)\right)\Phi\left(\prod_{i=1}^{n}\frac{\frac{1}{\sigma_{i}(Q)}\int_{Q}|g_{i}(y)|\sigma_{i}(y) dy}{\Phi_{i}^{-1}\left(\frac{1}{\sigma_{i}(x)}\right)} \right)\nu_{\overrightarrow{\sigma}}(x)dx\\
&= \int_{Q} \Phi\left(\prod_{i=1}^{n}\frac{\frac{1}{\sigma_{i}(Q)}\int_{Q}|g_{i}(y)|\sigma_{i}(y) dy}{\Phi_{i}^{-1}\left(\frac{1}{\sigma_{i}(x)}\right)} \right)dx \\
&\lesssim \sum_{i=1}^{n}\int_{Q} \Phi_{i}\left(\frac{\frac{1}{\sigma_{i}(Q)}\int_{Q}|g_{i}(y)|\sigma_{i}(y) dy}{\Phi_{i}^{-1}\left(\frac{1}{\sigma_{i}(x)}\right)} \right)dx \\
& \lesssim  \sum_{i=1}^{n} \int_{Q}\frac{\Phi_{i}\left(\frac{1}{\sigma_{i}(Q)}\int_{Q}|g_{i}(y)|\sigma_{i}(y) dy \right)}{\Phi_{i}\left(\Phi_{i}^{-1}\left(\frac{1}{\sigma_{i}(x)}\right)\right)} dx \\
&= \sum_{i=1}^{n}\Phi_{i}\left(\frac{1}{\sigma_{i}(Q)}\int_{Q}|g_{i}(y)|\sigma_{i}(y) dy \right)\sigma_{i}(Q) \\ &\lesssim  \sum_{i=1}^{n}\int_{Q}\Phi_{i}(|g_{i}(y)|)\sigma_{i}(y) dy. 
\end{align*}
As the cubes in $\mathcal{D}_{\lambda}^{*}(\overrightarrow{g})$ are pairwise disjoint, we deduce that 
\begin{align*}
\nu_{\overrightarrow{\sigma}}\left(E_{\lambda}\right) &= \nu_{\overrightarrow{\sigma}}\left(\bigcup_{Q \in \mathcal{D}_{\lambda}^{*}(\overrightarrow{g})}Q\right) =\sum_{Q \in \mathcal{D}_{\lambda}^{*}(\overrightarrow{g})}\nu_{\overrightarrow{\sigma}}(Q)\\
&\lesssim \frac{1}{\Phi(\lambda)}\sum_{i=1}^{n}\sum_{Q \in \mathcal{D}_{\lambda}^{*}(\overrightarrow{g})}\int_{Q}\Phi_{i}(|g_{i}(y)|)\sigma_{i}(y) dy \\
&= \frac{1}{\Phi(\lambda)}\sum_{i=1}^{n}\int_{\cup_{Q \in \mathcal{D}_{\lambda}^{*}(\overrightarrow{g})} Q}\Phi_{i}(|g_{i}(y)|)\sigma_{i}(y) dy \\ 
&\lesssim \frac{1}{\Phi(\lambda)}\sum_{i=1}^{n}\int_{\mathbb{R}^{d}}\Phi_{i}(|g_{i}(y)|)\sigma_{i}(y) dy\\
&\lesssim \frac{1}{\Phi(\lambda)}.
\end{align*}
$ii)$ Let us set $$  \widetilde{\Omega}_{3}(t)=\frac{1}{\Psi\circ \Phi^{-1}\left(\frac{1}{t}\right)},~~\forall~t>0 ~ \hspace*{0.25cm} ~\text{and}~\hspace*{0.25cm} ~  \widetilde{\Omega}_{3}(0)=0.     $$  
Since $\Psi \in \mathscr{U}$ and  $t\mapsto \frac{\Psi(t)}{\Phi(t)}$ is increasing on $\mathbb{R}_{+}^{*}$, it follows from Lemma \ref{pro:main40} that $\widetilde{\Omega}_{3} \in \mathscr{U}$.
We then obtain
\begin{align*}
\sum_{R \in \mathcal{D}_{\lambda}^{*}(\overrightarrow{g})}\widetilde{\Omega}_{3}\left(\nu_{\overrightarrow{\sigma}}(R)\right) &\lesssim \widetilde{\Omega}_{3}\left(\sum_{R \in \mathcal{D}_{\lambda}^{*}(\overrightarrow{g})}\nu_{\overrightarrow{\sigma}}(R)\right)\\ 
&= \widetilde{\Omega}_{3}\left(\nu_{\overrightarrow{\sigma}}\left(\bigcup_{R \in \mathcal{D}_{\lambda}^{*}(\overrightarrow{g})} R\right)\right) =\widetilde{\Omega}_{3}\left(\nu_{\overrightarrow{\sigma}}\left( E_{\lambda}\right)\right)
= \frac{\widetilde{\Omega}_{3}\left(\nu_{\overrightarrow{\sigma}}\left( E_{\lambda}\right)\right)}{\nu_{\overrightarrow{\sigma}}\left( E_{\lambda}\right)} \times \nu_{\overrightarrow{\sigma}}\left( E_{\lambda}\right)\\
&\lesssim \frac{\widetilde{\Omega}_{3}\left(\frac{1}{\Phi(\lambda)}\right)}{\frac{1}{\Phi(\lambda)}}\times \nu_{\overrightarrow{\sigma}}\left( E_{\lambda}\right)\\ 
&=\Phi(\lambda)\frac{1}{\Psi \circ \Phi^{-1}\left(\Phi(\lambda)\right)}\times \nu_{\overrightarrow{\sigma}}\left( E_{\lambda}\right)
=\frac{\Phi(\lambda)}{\Psi(\lambda)}\times \nu_{\overrightarrow{\sigma}}\left( E_{\lambda}\right). 
\end{align*}
$iii)$ Using Theorem \ref{pro:main88}, we obtain 
\begin{align*}
\int_{\mathbb{R}^{d}}\Phi\left(\mathcal{M}_{\overrightarrow{\sigma}}^{\mathcal{D}^{\beta}}(\overrightarrow{f})(x)\right)\nu_{\overrightarrow{\sigma}}(x) dx &\lesssim \int_{\mathbb{R}^{d}}\Phi\left(\prod_{i=1}^{n}\mathcal{M}_{\sigma_{i}}^{\mathcal{D}^{\beta}}(f_{i})(x)\right)\nu_{\overrightarrow{\sigma}}(x) dx\\
&=\int_{\mathbb{R}^{d}}\Phi\left(\prod_{i=1}^{n} \frac{ \Phi_{i}^{-1}\left(\frac{1}{\sigma_{i}(x)}\right)\times\mathcal{M}_{\sigma_{i}}^{\mathcal{D}^{\beta}}(f_{i})(x)}{\Phi_{i}^{-1}\left(\frac{1}{\sigma_{i}(x)}\right)}  \right)\nu_{\overrightarrow{\sigma}}(x) dx \\
&\lesssim \int_{\mathbb{R}^{d}}\Phi\left(\prod_{i=1}^{n}\Phi_{i}^{-1}\left(\frac{1}{\sigma_{i}(x)}\right)  \right)\Phi\left(\prod_{i=1}^{n}\frac{\mathcal{M}_{\sigma_{i}}^{\mathcal{D}^{\beta}}(f_{i})(x)}{\Phi_{i}^{-1}\left(\frac{1}{\sigma_{i}(x)}\right)}\right)\nu_{\overrightarrow{\sigma}}(x) dx \\
&=\int_{\mathbb{R}^{d}} \Phi\left(\prod_{i=1}^{n}\frac{\mathcal{M}_{\sigma_{i}}^{\mathcal{D}^{\beta}}(f_{i})(x)}{\Phi_{i}^{-1}\left(\frac{1}{\sigma_{i}(x)}\right)}\right)dx \\
&\lesssim \sum_{i=1}^{n}\int_{\mathbb{R}^{d}} \Phi_{i}\left(\frac{\mathcal{M}_{\sigma_{i}}^{\mathcal{D}^{\beta}}(f_{i})(x)}{\Phi_{i}^{-1}\left(\frac{1}{\sigma_{i}(x)}\right)}\right)dx \\
&\lesssim \sum_{i=1}^{n}\int_{\mathbb{R}^{d}} \frac{\Phi_{i}\left(\mathcal{M}_{\sigma_{i}}^{\mathcal{D}^{\beta}}(f_{i})(x)\right)}{\Phi_{i}\left(\Phi_{i}^{-1}\left(\frac{1}{\sigma_{i}(x)}\right)\right)}dx \\
&=\sum_{i=1}^{n}\int_{\mathbb{R}^{d}}\Phi_{i}\left(\mathcal{M}_{\sigma_{i}}^{\mathcal{D}^{\beta}}(f_{i})(x)\right)\sigma_{i}(x)dx\\
&\lesssim \sum_{i=1}^{n} \int_{\mathbb{R}^{d}}\Phi_{i}(|f_{i}(x)|)\sigma_{i}(x)dx.
\end{align*}
The proof is complete.
\end{proof}

Let $\Psi,\Phi_{1},...,\Phi_{n}\in \mathscr{U}$  and let  $\Phi$ be a one-to-one correspondence from $\mathbb{R}_{+}$ to itself such that $\Phi^{-1}=\Phi_{1}^{-1}\times...\times \Phi_{n}^{-1}$.\\ 
Note that if for any $1\leq i \leq n$, $\sigma_{i}\equiv \sigma$  then  the $\nu_{\overrightarrow{\sigma}}$ reduces to $\sigma$. In this case, the proof of the above proposition can be handled as follows.

\begin{itemize}
\item[\textbullet] The convexity of each $\Phi_{i}$ and the generalized Young's inequality are enough to obtain assertion $(i)$ and the constant $C_{1}$ in the inequality  (\ref{eq:foncijhimal}) is $n$.
\item[\textbullet] Let  $q\geq 1$ be such that $\Psi \in  \mathscr{U}^{q}$. The fact that the function $t\mapsto \frac{\Psi(t)}{\Phi(t)}$ defined $\mathbb{R}_{+}^{*}$ is increasing is sufficient to prove assertion $(ii)$ and the constant $C_{2}$ in (\ref{eq:ineg56a3jsen}) is $C_{q}n^{nq}$.
\item[\textbullet] Assertion $(iii)$ follows from the fact that each $\Phi_i$ satisfies the $\nabla_{2}-$condition, and the generalized Young's inequality.
\end{itemize}
In this last case, one doesn't need the $\Phi_i$s to belong to the class $\widetilde{\mathscr{U}}$. 
In the above case, Proposition \ref{pro:main0aqk5} takes the following form.

\begin{corollary}\label{pro:main0aaqqk5}
Let $\sigma$ be a weight on $\mathbb{R}^{d}$. Let $\Psi,\Phi_{1},...,\Phi_{n}\in \mathscr{U}$, and  let $\Phi$ be a one-to-one correspondence from  $\mathbb{R}_{+}$ to itself such that $\Phi^{-1}=\Phi_{1}^{-1}\times...\times \Phi_{n}^{-1}$. Then the following are satisfied.  
\begin{itemize}
\item[(i)] For any $\lambda > 0$ and any $1\leq i \leq n$,   $0\not\equiv f_{i}  \in L^{\Phi_{i}}(\sigma)$, 
\begin{equation}\label{eq:fjhimal}
\sigma\left(\left\{x\in \mathbb{R}^{d} : \mathcal{M}_{\overrightarrow{\sigma}}^{\mathcal{D}^{\beta}}\left(\frac{f_{1}}{\|f_{1}\|_{L_{\sigma}^{\Phi_{1}}}^{lux}},...,\frac{f_{n}}{\|f_{n}\|_{L_{\sigma}^{\Phi_{n}}}^{lux}}\right)(x)> \lambda \right\}\right) \leq  \dfrac{n}{\Phi(\lambda)}.\end{equation}
\item[(ii)] Let $q\geq 1$ be such that $\Psi \in  \mathscr{U}^{q}$. If  the function $t\mapsto \frac{\Psi(t)}{\Phi(t)}$ is increasing on $\mathbb{R}_{+}^{*}$, then for any $\lambda > 0$ and any $(f_{1},...,f_{n}) \in L^{\Phi_{1}}(\sigma)\times...\times
 L^{\Phi_{n}}(\sigma)$ with  $f_{i} \not\equiv 0$, for $i=1,...,n$,
\begin{equation}\label{eq:inaq6a3jsen}
\sum_{R \in \mathcal{D}_{\lambda}^{*}(\overrightarrow{g})} \frac{1}{\Psi\circ \Phi^{-1}\left(\frac{1}{\sigma(R)}\right)} \leq C_{q}n^{nq} \frac{\Phi(\lambda)}{\Psi(\lambda)} \sigma\left( \left\{ x\in \mathbb{R}^{d} : \mathcal{M}^{\mathcal{D}^{\beta}}_{\overrightarrow{\sigma}}(\overrightarrow{g})(x) >  \lambda \right\} \right),
\end{equation} 
where  $\overrightarrow{g}:=(g_{1},...,g_{n})$, with $g_{i}=\frac{f_{i}}{\|f_{i}\|_{L_{\sigma}^{\Phi_{i}}}^{lux}}.$
\item[(iii)] If for any $1\leq i \leq n$,  $\Phi_{i} \in  \nabla_{2}$, then there exists a  constant $C:= C_{ \Phi_{1},...,\Phi_{n}} > 0$ such that for any $(f_{1},...,f_{n}) \in L^{\Phi_{1}}(\sigma)\times...\times
 L^{\Phi_{n}}(\sigma)$,  
\begin{equation}\label{eq:fodmal}
\int_{\mathbb{R}^{d}}\Phi\left(\mathcal{M}_{\overrightarrow{\sigma}}^{\mathcal{D}^{\beta}}(f_{1},...,f_{n})(x)\right)\sigma(x) dx \leq C \sum_{i=1}^{n} \int_{\mathbb{R}^{d}}\Phi_{i}(|f_{i}(x)|)\sigma(x)dx.
\end{equation}
\end{itemize}
\end{corollary}
We finish this section with the following useful result.
\begin{proposition}\label{pro:main17jo}
Let $\sigma_{1},...,\sigma_{n},\omega$ be weights on $\mathbb{R}^{d}$,   $\Psi,\Phi_{1},...,\Phi_{n} \in \mathscr{U}$, and  $T : L^{\Phi_{1}}(\sigma_{1})\times...\times L^{\Phi_{n}}(\sigma_{n})\longrightarrow L^{\Psi}(\omega)$ a nonnegative operator. Suppose that there are constants $C_{1},C_{2}>0$ such that for  $1 \leq i \leq n$, $ 0\not\equiv f_{i}\in L^{\Phi_{i}}(\sigma_{i})$,  
\begin{equation}\label{eq:cdinis}
\int_{\mathbb{R}^{d}}\Psi\left(\frac{T(f_{1},...,f_{n})(x)}{\prod_{i=1}^{n}\|f_{i}\|_{L_{\sigma_{i}}^{\Phi_{i}}}^{lux}}\right)\omega(x)dx \leq C_{1}C_{2}.
\end{equation}
If $\Psi \in   \widetilde{\mathscr{U}}$, then there exists a constant  $C_{3}>0$ such that for any  $(f_{1},...,f_{n}) \in L^{\Phi_{1}}(\sigma_{1})\times...\times L^{\Phi_{n}}(\sigma_{n})$, 
\begin{equation}\label{eq:cs}
 \|T(f_{1},...,f_{n})\|_{L_{\omega}^{\Psi}}^{lux} \leq C_{3}\Psi^{-1}\left(C_{2}\right)\prod_{i=1}^{n}\|f_{i}\|_{L_{\sigma_{i}}^{\Phi_{i}}}^{lux}.
\end{equation}
\end{proposition}

\begin{proof}
As  $\Psi \in   \widetilde{\mathscr{U}}$, it follows from Lemma \ref{pro:main80jdf}, that there exists a constant  $C$ such that 
$$\Psi\left(\frac{s}{t}\right)\leq C\frac{\Psi(s)}{\Psi(t)}, ~~ \forall~s,t >0.$$
As $\Psi \in \Delta'$, there exists a constant $C'>0$ such that $$ \Psi\left(st\right)\leq C'\Psi(s)\Psi(t), ~~ \forall~s,t >0.     $$
We set $$C_{3}:=\max\{1; CC'C_{1}\Psi(1)\}.$$
Let $(f_{1},...,f_{n}) \in L^{\Phi_{1}}(\sigma_{1})\times...\times L^{\Phi_{n}}(\sigma_{n})$. Let us suppose that for $1\leq i \leq n$, $f_{i}\not\equiv 0$. As  (\ref{eq:cdinis}) holds, we have    
\begin{align*}
\int_{\mathbb{R}^{d}}\Psi\left( \frac{T(f_{1},...,f_{n})(x)}{C_{3}\Psi^{-1}\left(C_{2}\right)\prod_{i=1}^{n}\|f_{i}\|_{L_{\sigma_{i}}^{\Phi_{i}}}^{lux}}\right)\omega(x)dx 
 &\leq \frac{1}{C_{3}}\int_{\mathbb{R}^{d}}\Psi\left( \frac{T(f_{1},...,f_{n})(x)}{\Psi^{-1}\left(C_{2}\right)\prod_{i=1}^{n}\|f_{i}\|_{L_{\sigma_{i}}^{\Phi_{i}}}^{lux}}\right)\omega(x)dx
 \\
&\leq \frac{1}{CC'C_{1}\Psi(1)}\int_{\mathbb{R}^{d}}\Psi\left( \frac{T(f_{1},...,f_{n})(x)}{\Psi^{-1}\left(C_{2}\right)\prod_{i=1}^{n}\|f_{i}\|_{L_{\sigma_{i}}^{\Phi_{i}}}^{lux}}\right)\omega(x)dx
  \\
&\leq \frac{1}{CC'C_{1}\Psi(1)}\int_{\mathbb{R}^{d}}C'\Psi\left( \frac{1}{\Psi^{-1}\left(C_{2}\right)}\right)\Psi\left( \frac{T(f_{1},...,f_{n})(x)}{\prod_{i=1}^{n}\|f_{i}\|_{L_{\sigma_{i}}^{\Phi_{i}}}^{lux}}\right)\omega(x)dx\\
&\leq \frac{1}{CC'C_{1}\Psi(1)}\int_{\mathbb{R}^{d}} \frac{C'C\Psi(1)}{\Psi\left(\Psi^{-1}\left(C_{2}\right)\right)}\Psi\left( \frac{T(f_{1},...,f_{n})(x)}{\prod_{i=1}^{n}\|f_{i}\|_{L_{\sigma_{i}}^{\Phi_{i}}}^{lux}}\right)\omega(x)dx\\
&= \frac{1}{C_{1}C_{2}}\int_{\mathbb{R}^{d}} \Psi\left( \frac{T(f_{1},...,f_{n})(x)}{\prod_{i=1}^{n}\|f_{i}\|_{L_{\sigma_{i}}^{\Phi_{i}}}^{lux}}\right)\omega(x)dx\\ 
&\leq \frac{1}{C_{1}C_{2}} \times C_{1}C_{2}=1.
\end{align*}
\end{proof}
\subsection{Example of weights for (\ref{eq:surmaron})}
In the following, we prove that the reverse of (\ref{eq:surmaron}) always holds. We also prove that if each of the involved weights is in the Muckenhoupt class $A_1$, then (\ref{eq:surmaron}) holds.
\begin{proposition}
Let $\sigma_{1},...,\sigma_{n}$ be weights on $\mathbb{R}^{d}$,  and let $\Phi_{1},...,\Phi_{n}\in \widetilde{\mathscr{U}}$. Assume that $\Phi$ is a one-to-one correspondence from $\mathbb{R}_{+}$ to itself such that
 $\Phi^{-1}=\Phi_{1}^{-1}\times...\times \Phi_{n}^{-1}$. Define $\nu_{\overrightarrow{\sigma}}:=  \frac{1}{\Phi\left(\prod_{i=1}^{n}\Phi_{i}^{-1}\left(\frac{1}{\sigma_{i}}\right)  \right)}$. Then the following assertions hold.
\begin{itemize}
\item[(i)] There exists a constant $C_{1}:=C_{n,\Phi_{1},...,\Phi_{n}}>0$ such that 
\begin{equation}\label{eq:dr55im}
\Phi\left(\prod_{i=1}^{n}\Phi_{i}^{-1}\left(\frac{1}{\sigma_{i}(Q)}\right)  \right)\nu_{\overrightarrow{\sigma}}(Q) \leq C_{1},~~ \forall~Q \in \mathcal{Q}.
\end{equation} 
\item[(ii)] If for each $1\leq i \leq n$, $\sigma_{i} \in A_{1}$, then there exists a constant $C_{2}:=C_{n, \Phi_{1},...,\Phi_{n}}>0$ such that
\begin{equation}\label{eq:drim}
\frac{C_{2}}{\sum_{i=1}^{n}[\sigma_{i}]_{A_{1}}}\leq\Phi\left(\prod_{i=1}^{n}\Phi_{i}^{-1}\left(\frac{1}{\sigma_{i}(Q)}\right)  \right)\nu_{\overrightarrow{\sigma}}(Q),~~ \forall~Q \in \mathcal{Q}.
\end{equation} 
\end{itemize} 
 \end{proposition}

\begin{proof}
$i)$  Let $1\leq i \leq n$. Since $\Phi_{i} \in \widetilde{\mathscr{U}}$, we deduce that $$ \Phi_{i}\left(\frac{s}{t}\right)\leq C_{i}\frac{\Phi_{i}(s)}{\Phi_{i}(t)}, ~~ \forall~s,t >0. 
      $$
It follows that $$  \dfrac{\Phi_{i}^{-1}(s)}{\Phi_{i}^{-1}(t)}\leq \Phi_{i}^{-1}\left(C_{i}\dfrac{s}{t}\right), ~~ \forall~s,t >0. 
       $$
As  $\Phi_{i}\in \Delta'$ for any $1\le 1\le n$, we have that $\Phi\in \Delta'$. 
Let  $Q\in \mathcal{Q}$. We obtain   
\begin{align*}
\Phi\left(\prod_{i=1}^{n}\Phi_{i}^{-1}\left(\frac{1}{\sigma_{i}(Q)}\right)  \right)\nu_{\overrightarrow{\sigma}}(Q)
&= \int_{Q}\Phi\left(\prod_{i=1}^{n}\Phi_{i}^{-1}\left(\frac{1}{\sigma_{i}(Q)}\right)  \right)\nu_{\overrightarrow{\sigma}}(x)dx\\
&= \int_{Q}\Phi\left(\prod_{i=1}^{n}\frac{\Phi_{i}^{-1}\left(\frac{1}{\sigma_{i}(x)}\right)  \times \Phi_{i}^{-1}\left(\frac{1}{\sigma_{i}(Q)}\right)  }{\Phi_{i}^{-1}\left(\frac{1}{\sigma_{i}(x)}\right)}  \right)\nu_{\overrightarrow{\sigma}}(x)dx   \\ 
&\lesssim\int_{Q} \Phi\left(\prod_{i=1}^{n}\Phi_{i}^{-1}\left(\frac{1}{\sigma_{i}(x)}\right)   \right)\Phi\left(\prod_{i=1}^{n}\frac{ \Phi_{i}^{-1}\left(\frac{1}{\sigma_{i}(Q)}\right)  }{\Phi_{i}^{-1}\left(\frac{1}{\sigma_{i}(x)}\right)}  \right)\nu_{\overrightarrow{\sigma}}(x)dx\\
&= \int_{Q} \Phi\left(\prod_{i=1}^{n}\frac{ \Phi_{i}^{-1}\left(\frac{1}{\sigma_{i}(Q)}\right)  }{\Phi_{i}^{-1}\left(\frac{1}{\sigma_{i}(x)}\right)}  \right)dx 
\lesssim \int_{Q} \Phi\left(\prod_{i=1}^{n}\Phi_{i}^{-1}\left(\frac{\frac{1}{\sigma_{i}(Q)}}{\frac{1}{\sigma_{i}(x)}}\right)\right)dx \\ &=\int_{Q}\Phi\left(\prod_{i=1}^{n}\Phi_{i}^{-1}\left(\frac{\sigma_{i}(x)}{\sigma_{i}(Q)}\right)  \right)dx \\
&\lesssim
\int_{Q}\sum_{i=1}^{n}\Phi_{i}\left(\Phi_{i}^{-1}\left(\frac{\sigma_{i}(x)}{\sigma_{i}(Q)}\right)  \right)dx
=\sum_{i=1}^{n}\int_{Q}\frac{\sigma_{i}(x)}{\sigma_{i}(Q)}dx= n.
\end{align*}
$ii)$  
Let us assume that for each  $1\leq i \leq n$, $\sigma_{i} \in A_{1}$. 
Let  $1\leq i \leq n$. As   $\Phi_{i}\in \Delta'$, we have that  $\Phi_{i}^{-1}\in \Delta''$.
Now let  $Q\in \mathcal{Q}$. Let $x\in Q$. We have
\begin{align*}
\frac{1}{|Q|}\frac{\Phi\left(\prod_{i=1}^{n}\Phi_{i}^{-1}\left(\frac{1}{\sigma_{i}(x)}\right)  \right)}{\Phi\left(\prod_{i=1}^{n}\Phi_{i}^{-1}\left(\frac{1}{\sigma_{i}(Q)}\right)  \right)}&= \frac{1}{|Q|}\frac{\Phi\left(\prod_{i=1}^{n}\Phi_{i}^{-1}\left(\frac{1}{\sigma_{i}(Q)}\right)\times\frac{\Phi_{i}^{-1}\left(\frac{1}{\sigma_{i}(x)}\right)  }{\Phi_{i}^{-1}\left(\frac{1}{\sigma_{i}(Q)}\right)  }  \right)}{\Phi\left(\prod_{i=1}^{n}\Phi_{i}^{-1}\left(\frac{1}{\sigma_{i}(Q)}\right)  \right)} \\
&\lesssim \frac{1}{|Q|} \Phi\left(\prod_{i=1}^{n}\frac{\Phi_{i}^{-1}\left(\frac{1}{\sigma_{i}(x)}\right)  }{\Phi_{i}^{-1}\left(\frac{1}{\sigma_{i}(Q)}\right)  }  \right) \lesssim \frac{1}{|Q|} \Phi\left(\prod_{i=1}^{n}\Phi_{i}^{-1}\left(\frac{\sigma_{i}(Q)}{\sigma_{i}(x)}\right)\right)\\
&\lesssim \frac{1}{|Q|} \sum_{i=1}^{n}\Phi_{i}\left(\Phi_{i}^{-1}\left(\frac{\sigma_{i}(Q)}{\sigma_{i}(x)}\right)  \right)= \frac{1}{|Q|} \sum_{i=1}^{n}\frac{\sigma_{i}(Q)}{\sigma_{i}(x)}=  \sum_{i=1}^{n}\frac{\sigma_{i}(Q)}{|Q|}\frac{1}{\sigma_{i}(x)} \\
&\lesssim \sum_{i=1}^{n}\frac{\sigma_{i}(Q)}{|Q|}\left(\inf ess_{x \in Q} \sigma_{i}(x)\right)^{-1} \lesssim \sum_{i=1}^{n}[\sigma_{i}]_{A_{1}}. 
\end{align*}
We deduce that $$  
\frac{1}{|Q|\Phi\left(\prod_{i=1}^{n}\Phi_{i}^{-1}\left(\frac{1}{\sigma_{i}(Q)}\right)  \right)}\lesssim \left(\sum_{i=1}^{n}[\sigma_{i}]_{A_{1}}\right)\frac{1}{\Phi\left(\prod_{i=1}^{n}\Phi_{i}^{-1}\left(\frac{1}{\sigma_{i}(x)}\right)  \right)}=\left(\sum_{i=1}^{n}[\sigma_{i}]_{A_{1}}\right)\nu_{\overrightarrow{\sigma}}(x), ~~ \forall~x\in Q.      $$
It follows that $$ \frac{1}{\sum_{i=1}^{n}[\sigma_{i}]_{A_{1}}} \lesssim  \Phi\left(\prod_{i=1}^{n}\Phi_{i}^{-1}\left(\frac{1}{\sigma_{i}(Q)}\right)  \right)\nu_{\overrightarrow{\sigma}}(Q),~~ \forall~Q \in \mathcal{Q}.    $$  
The proof is complete.
\end{proof}

\section{Proofs of the results}
\subsection{Proofs of Carleson embedding results}

Let $\sigma$ be a weight on $\mathbb{R}^{d}$, $f$ a measurable function on $\mathbb{R}^{d}$, and $Q\in \mathcal{D}^{\beta}$. 
We will be using the notation  
\begin{equation}\label{eq:s56son}
m_{\sigma}(f,Q):=\frac{1}{\sigma(Q)}\int_{Q}|f(x)|\sigma(x)dx.    
 \end{equation}
 
\begin{proposition}[\cite{raoren}, Corollary 7]\label{pro:main92}
Let  $\Phi \in  \mathscr{U}$, and let $\sigma$ be a weight on  $\mathbb{R}^{d}$. For any cube $Q\in \mathcal{Q}$,
\begin{equation}\label{eq:egali6orinlf}
   \dfrac{1}{\|\chi_{Q}\|_{L_{\sigma}^{\Phi}}^{lux}}= \Phi^{-1}\left(\dfrac{1}{\sigma(Q)}\right).\end{equation}\end{proposition}
 We first prove Proposition \ref{pro:main16aqyh6}.

 \begin{proof}[Proof of Proposition \ref{pro:main16aqyh6}]
  Let us assume that (\ref{eq:surmaron}) and (\ref{eq:suitc5faqrmaron}) hold. Let
 $R \in \mathcal{D}^{\beta}$. Let  $Q \in \mathcal{D}^{\beta}$ such that  $Q \subset R$.
 Then by Proposition \ref{pro:main92}, we have
 $$ \Phi_{i}^{-1}\left(\frac{1}{\sigma_{i}(R)}\right)=\dfrac{1}{\|\chi_{R}\|_{L_{\sigma_{i}}^{\Phi_{i}}}^{lux}} ,~~\forall~1\leq i\leq n\Rightarrow  \prod_{i=1}^{n}\Phi_{i}^{-1}\left(\frac{1}{\sigma_{i}(R)}\right) =\prod_{i=1}^{n}\frac{1}{\sigma_{i}(Q)}\int_{Q}\frac{\chi_{R}(x)}{\|\chi_{R}\|_{L_{\sigma_{i}}^{\Phi_{i}}}^{lux} }\sigma_{i}(x)dx.   $$
 As $ \Psi\in \mathscr{U}$ and $t\mapsto \frac{\Psi(t)}{\Phi(t)}$ is increasing on $\mathbb{R}_{+}^{*}$, we can deduce with the help of Lemma \ref{pro:main40}, that $\Psi \circ \Phi^{-1} \in \mathscr{U}.$  
 As (\ref{eq:surmaron}) and (\ref{eq:suitc5frmaron}) are satisfied, we then obtain
 \begin{align*}
 \sum_{Q \subset R,  Q \in \mathcal{D}^{\beta}}\lambda_{Q}\Psi \circ \Phi^{-1}\left( \frac{1}{\nu_{\overrightarrow{\sigma}}(R)}\right)
 &\lesssim \sum_{Q \subset R,  Q \in \mathcal{D}^{\beta}}\lambda_{Q}\Psi \circ \Phi^{-1}\left( \Phi\left(\prod_{i=1}^{n}\Phi_{i}^{-1}\left(\frac{1}{\sigma_{i}(R)}\right)  \right)\right)\\
 &= \sum_{Q \subset R,  Q \in \mathcal{D}^{\beta}}\lambda_{Q}\Psi\left(\prod_{i=1}^{n}\frac{1}{\sigma_{i}(Q)}\int_{Q}\frac{\chi_{R}(x)}{\|\chi_{R}\|_{L_{\sigma_{i}}^{\Phi_{i}}}^{lux} }\sigma_{i}(x)dx\right)\\
 &\lesssim 1.
 \end{align*}
 \end{proof}

We next prove Theorem \ref{pro:main16yh6}

\begin{proof}[Proof of Theorem \ref{pro:main16yh6}]
Let us suppose that (\ref{eq:s564c25son}) holds.\\
Let $\mathscr{P}(\mathcal{D}^{\beta})$ be the power set of $\mathcal{D}^{\beta}$. Consider the mapping $\mu$ defined on $\mathscr{P}(\mathcal{D}^{\beta})$ by
 $$  \mu(A)=  \sum_{Q \in \mathcal{D}^{\beta}}\lambda_{Q}\chi_{A}(Q),~~\forall~A \in \mathscr{P}(\mathcal{D}^{\beta}).  $$
By construction, $\mu$ is a counting measure on $\mathscr{P}(\mathcal{D}^{\beta})$ such that $$  \mu(\{Q\})= \lambda_{Q}, ~~\forall~ Q \in \mathcal{D}^{\beta}.        $$
Let $1\leq i \leq n$ and $f_{i} \in L^{\Phi_{i}}(\sigma_{i})$ such that  $f_{i} \not\equiv 0$. Put $g_{i}=\frac{f_{i}}{\|f_{i}\|_{L_{\sigma_{i}}^{\Phi_{i}}}^{lux}}$ and
  $\overrightarrow{g}:=(g_{1},...,g_{n})$.
Let us fix $t> 0$ and consider $\mathcal{D}_{t}(\overrightarrow{g})$ (resp.  $\mathcal{D}_{t}^{*}(\overrightarrow{g})$) the set of dyadic  cubes $Q$  (resp. maximal dyadic cubes $Q$ with respect to the inclusion) such that 
$$\prod_{i=1}^{n}\frac{1}{\sigma_{i}(Q)}\int_{Q}|g_{i}(x)|\sigma_{i}(x)dx> t.$$
As (\ref{eq:s564c25son}) holds and since any element of $\mathcal{D}_{t}(\overrightarrow{g})$ is a subset of a unique element of $\mathcal{D}_{t}^{*}(\overrightarrow{g})$,  it follows from assertion $(ii)$ in Proposition \ref{pro:main0aqk5}, that
\begin{align*}
\mu\left(\mathcal{D}_{t}(\overrightarrow{g})\right)&=\mu\left(\bigcup_{Q \in \mathcal{D}_{t}(\overrightarrow{g})}\{Q\}\right)\\{\tiny } &\leq \mu\left(\bigcup_{R \in \mathcal{D}_{t}^{*}(\overrightarrow{g})}\bigcup_{Q \in \mathcal{D}_{t}(\overrightarrow{g}): Q \subset R}\{Q\}\right) 
\leq \sum_{ R\in \mathcal{D}^{*}_{t}(\overrightarrow{g})} \sum_{ Q \subset R : Q \in \mathcal{D}^{\beta}}\lambda_{Q} \\
&\leq \sum_{ R\in \mathcal{D}^{*}_{t}(\overrightarrow{g})}\frac{\Lambda}{\Psi \circ \Phi^{-1}\left( \frac{1}{\nu_{\overrightarrow{\sigma}}(R)}\right)}
\leq \Lambda C \frac{\Phi(t)}{\Psi(t)} \nu_{\overrightarrow{\sigma}}\left( \left\{ x\in \mathbb{R}^{d} : \mathcal{M}^{\mathcal{D}^{\beta}}_{\overrightarrow{\sigma}}(\overrightarrow{g})(x) >  t \right\}\right).
\end{align*}
As each $\Phi_{i} \in \mathscr{U}$ and $\Phi^{-1}=\Phi_{1}^{-1}\times...\times \Phi_{n}^{-1}$, we deduce that  $\Phi$ is in  $\mathscr{C}^{1}$ and $\Phi'(t)\approx \frac{\Phi(t)}{t}$, for all $ t > 0.$ Indeed, we have 
\begin{align*}
\Phi_{i} \in \mathscr{U},~~\forall~ 1\leq i \leq n &\Rightarrow \Phi_{i}'(t)\approx \frac{\Phi_{i}(t)}{t},~~\forall~t>0, ~~\forall~ 1\leq i \leq n \\
&\Rightarrow \left(\Phi_{i}^{-1}\right)'(t)\approx \frac{\Phi_{i}^{-1}(t)}{t},~~\forall~t>0, ~~\forall~ 1\leq i \leq n \\
&\Rightarrow   \left(\Phi_{1}^{-1}\times...\times \Phi_{n}^{-1}\right)'(t)\approx\sum_{i=1}^{n} \frac{\Phi_{i}^{-1}(t)}{t}\times \prod_{j=1,j\not=i}^{n}\Phi_{j}^{-1}(t),~~\forall~t>0 \\
&\Rightarrow \left(\Phi_{1}^{-1}\times...\times \Phi_{n}^{-1}\right)'(t)\approx\sum_{i=1}^{n} \dfrac{\prod_{i=1}^{n}\Phi_{i}^{-1}(t)}{t} ,~~\forall~t>0 \\
&\Rightarrow \left(\Phi^{-1}\right)'(t)\approx \frac{\Phi^{-1}(t)}{t}, ~~\forall~t>0 \\
&\Rightarrow \Phi'(t)\approx \frac{\Phi(t)}{t} ,~~\forall~t>0.
\end{align*}
It follows that $\frac{\Phi(t)}{\Psi(t)} \approx \frac{\Phi^{\prime}(t)}{\Psi^{\prime}(t)}$, for all $ t > 0.$
We deduce that $$  \mu\left(\mathcal{D}_{t}(\overrightarrow{g})\right) \leq \Lambda C \frac{\Phi'(t)}{\Psi'(t)} \nu_{\overrightarrow{\sigma}}\left( \left\{ x\in \mathbb{R}^{d} : \mathcal{M}^{\mathcal{D}^{\beta}}_{\overrightarrow{\sigma}}(\overrightarrow{g})(x) >  t \right\}\right), ~~\forall~t>0.         $$
If follows from the above and the boundedness of the weighted dyadic maximal function that
\begin{align*}
\sum_{Q\in \mathcal{D}^{\beta}}\lambda_{Q}\Psi\left(\prod_{i=1}^{n}m_{\sigma_{i}}(g_{i},Q)\right)
&=\int_{0}^{\infty}\Psi^{\prime}(t)\mu\left( \left\{ Q\in \mathcal{D}^{\beta} :\prod_{i=1}^{n} m_{\sigma_{i}}(g_{i},Q)   > t \right\}  \right)dt \\ 
&=\int_{0}^{\infty}\Psi^{\prime}(t)\mu\left(\mathcal{D}_{t}(\overrightarrow{g})\right)dt \\
&\leq \int_{0}^{\infty}\Psi^{\prime}(t)\left( \Lambda C\frac{\Phi'(t)}{\Psi'(t)} \nu_{\overrightarrow{\sigma}}\left( \left\{ x\in \mathbb{R}^{d} : \mathcal{M}^{\mathcal{D}^{\beta}}_{\overrightarrow{\sigma}}(\overrightarrow{g})(x) >  t \right\}\right)\right)dt \\
&=\Lambda C \int_{0}^{\infty}\Phi^{\prime}(t)\nu_{\overrightarrow{\sigma}}\left( \left\{ x\in \mathbb{R}^{d} : \mathcal{M}^{\mathcal{D}^{\beta}}_{\overrightarrow{\sigma}}(\overrightarrow{g})(x) >  t \right\}\right)dt \\
&=\Lambda C\int_{\mathbb{R}^{d}}\Phi\left( \mathcal{M}_{\overrightarrow{\sigma}}^{\mathcal{D}^{\beta}}(\overrightarrow{g})(x)\right)\nu_{\overrightarrow{\sigma}}(x) dx \\
&\leq \Lambda C \sum_{i=1}^{n}\int_{\mathbb{R}^{d}}\Phi_{i}\left(|g_{i}(x)|\right)\sigma_{i}(x)dx \\
&\leq n\Lambda C.
\end{align*}
The proof is complete.
\end{proof}


Let $\Phi_{1},...,\Phi_{n}\in \mathscr{U}$. Let  $\Phi$ be a one-to-one correspondence from $\mathbb{R}_{+}$ to itself such that $\Phi^{-1}=\Phi_{1}^{-1}\times...\times \Phi_{n}^{-1}$.
If for any $1\leq i \leq n$, $\sigma_{i}\equiv \sigma$, then we know that $\nu_{\overrightarrow{\sigma}}$ is just $\sigma$. In this case, inequalities (\ref{eq:s564c25son}) and (\ref{eq:suitc5frmaron}) are equivalent since the condition (\ref{eq:surmaron}) is satisfied.  Indeed, that condition (\ref{eq:suitc5frmaron}) impliques (\ref{eq:s564c25son}) follows from Proposition \ref{pro:main16aqyh6}. To prove that condition (\ref{eq:s564c25son}) impliques (\ref{eq:suitc5frmaron}), it suffices to replace in the proof of Theorem \ref{pro:main16yh6}, the assertions (ii) and (iii) in Proposition \ref{pro:main0aqk5} by the assertions  (ii ) and (iii) of Corollairy \ref{pro:main0aaqqk5} respectively.
Theorem \ref{pro:main16yh6} then reduces to the following.
 
 \begin{corollary}\label{pro:main16aaqq6}
 Let $\{\lambda_{Q}\}_{Q \in \mathcal{D}^{\beta}}$ be a sequence of positive real numbers, $\sigma$ a weight on $\mathbb{R}^{d}$,  $\Phi_{1},...,\Phi_{n} \in \mathscr{U}\cap \nabla_{2}$, $\Psi\in \mathscr{U}$. Let $\Phi$ be a one-to-one correspondence from $\mathbb{R}_{+}$ to itself such that $\Phi^{-1}=\Phi_{1}^{-1}\times...\times \Phi_{n}^{-1}$, and  $t\mapsto \frac{\Psi(t)}{\Phi(t)}$ is increasing on $\mathbb{R}_{+}^{*}$. Then the following  assertions are equivalent. 
 \begin{itemize}
 \item[(i)] $\{\lambda_{Q}\}_{Q \in \mathcal{D}^{\beta}}$ is a  $(\sigma,\Psi \circ \Phi^{-1})-$Carleson sequence.
 \item[(ii)] There exists a constant $C:=C_{n,\Psi,\Phi_{1},...,\Phi_{n}}>0$ such that for any 
 $0\not\equiv f_{i}\in L^{\Phi_{i}}(\sigma)$ with $i=1,...,n$,
 \begin{equation}\label{eq:suitaqc5frmaron}
  \sum_{Q\in \mathcal{D}^{\beta}}\lambda_{Q}\Psi\left(\prod_{i=1}^{n}m_{\sigma}\left(\frac{f_{i}}{\|f_{i}\|_{L_{\sigma}^{\Phi_{i}}}^{lux}},Q\right)\right) \leq C. 
 \end{equation}
 \end{itemize}
 \end{corollary}
 
 When $n=1$,  the above corollary is the following (see also \cite{jmtafeuseh}).
 
 \begin{corollary}\label{pro:main16aaqaqq6}
 Let $\{\lambda_{Q}\}_{Q \in \mathcal{D}^{\beta}}$ be a sequence of positive real numbers, $\sigma$ a weight on $\mathbb{R}^{d}$,  $\Phi,\Psi\in \mathscr{U}$ such that  $\Phi\in  \nabla_2$  and     $t\mapsto \frac{\Psi(t)}{\Phi(t)}$ is increasing on $\mathbb{R}_{+}^{*}$. Then the following  assertions are equivalent. 
 \begin{itemize}
 \item[(i)] $\{\lambda_{Q}\}_{Q \in \mathcal{D}^{\beta}}$ is a  $(\sigma,\Psi \circ \Phi^{-1})-$Carleson sequence.
 \item[(ii)] There exists a constant $C:=C_{\Psi,\Phi}>0$ such that for any 
 $0\not\equiv f\in L^{\Phi}(\sigma)$,
 \begin{equation}\label{eq:suiqcaraqmaron}
\sum_{Q\in \mathcal{D}^{\beta}}\lambda_{Q}\Psi\left(m_{\sigma}\left(\frac{f}{\|f\|_{L_{\sigma}^{\Phi}}^{lux}},Q\right)\right) \leq C.\end{equation} \end{itemize}
 \end{corollary}

\begin{corollary}\label{pro:main2aq09}
Let $\{\lambda_{Q}\}_{Q \in \mathcal{D}^{\beta}}$ be a sequence of positive real numbers,  $\sigma$ a weight on $\mathbb{R}^{d}$, and $\Phi\in \mathscr{U}\cap \nabla_2$.
Then the following are equivalent. 
\begin{itemize}
\item[(i)] $\{\lambda_{Q}\}_{Q \in \mathcal{D}^{\beta}}$ is a $\sigma-$Carleson sequence.
\item[(ii)] There exists a constant $C>0$ such that for any $0\not\equiv f\in L^{\Phi}(\sigma)$,
\begin{equation}\label{eq:sui56carmaron}
\sum_{Q\in \mathcal{D}^{\beta}}\lambda_{Q}\Phi\left(m_{\sigma}\left(\frac{f}{\|f\|_{L_{\sigma}^{\Phi}}^{lux}},Q\right)\right) \leq C.\end{equation}
\end{itemize}
\end{corollary}

\subsection{Proof of the weighted inequalities for the weighted multilinear maximal function}

We next prove Theorem \ref{pro:maing1}.

\begin{proof}[Proof of Theorem \ref{pro:maing1}]
Let us suppose that for each $1 \leq i \leq n$, $\sigma_{i}$ is doubling and $(\overrightarrow{\sigma}, \omega)\in M_{\overrightarrow{\Phi},\Psi}$.
\vskip .1cm
Let $(f_{1},...,f_{n})\in L^{\Phi_{1}}(\sigma_{1})\times...\times L^{\Phi_{n}}(\sigma_{n})$. Suppose that for $1\leq i \leq n$,  $f_{i}\not\equiv 0$.  Put   $\overrightarrow{g}:=(g_{1},...,g_{n})$, with $g_{i}=\frac{f_{i}}{\|f_{i}\|_{L_{\sigma_{i}}^{\Phi_{i}}}^{lux}}$,  for $1\leq i \leq n$.
As each $\sigma_{i}$ is doubling, it follows from Lemma \ref{pro:main80} that
$$  \mathcal{M}_{\overrightarrow{\sigma}} (\overrightarrow{g}) \lesssim  \sum_{\beta \in \{0,  1/3\}^{d} }   \mathcal{M}_{\overrightarrow{\sigma}}^{\mathcal{D}^{\beta}} (\overrightarrow{g}). 
    $$
 Let $q\geq 1$ such that $\Psi \in \widetilde{\mathscr{U}}^{q}$. As $\Psi$ is convex and of upper-type $q$,  we deduce that
$$ \int_{\mathbb{R}^{d}}\Psi\left( \mathcal{M}_{\overrightarrow{\sigma}}(\overrightarrow{g})(x)\right)\omega(x)dx \lesssim  \sum_{\beta \in \{0,  1/3\}^{d} }  \int_{\mathbb{R}^{d}}\Psi\left( \mathcal{M}_{\overrightarrow{\sigma}}^{\mathcal{D}^{\beta}}(\overrightarrow{g})(x)\right)\omega(x)dx.        $$
It suffices then to prove that for any $\beta \in \{0,  1/3\}^{d}$ fixed,   
$$ \int_{\mathbb{R}^{d}}\Psi\left( \mathcal{M}^{\mathcal{D}^{\beta}}_{\overrightarrow{\sigma}}(\overrightarrow{g})(x)\right)\omega(x)dx  \lesssim [\overrightarrow{\sigma},\omega]_{M_{\overrightarrow{\Phi},\Psi}}.     $$
    
Let $\mathcal{S}^\beta=\{Q_{k,j}\}_{j,k\in \mathbb{Z}}\subset \mathcal{D}^\beta$ be the family of all maximal cubes $Q_{k,j}$ with respect to the inclusion such that
$$   \prod_{i=1}^{n}\frac{1}{\sigma_{i}(Q_{k,j})  }\int_{Q_{k,j}}|g_{i}(x)|\sigma_{i}(x)dx > 2^{k}, ~~\forall~j \in \mathbb{Z}.           $$
Then $\mathcal{S}^\beta$ is a sparse family and
$$  A_{k}:= \left\{ x \in \mathbb{R}^{d} : \mathcal{M}_{\overrightarrow{\sigma}}^{\mathcal{D}^{\beta}} (\overrightarrow{g})(x) > 2^{k}    \right\}=\bigcup_{j\in \mathbb{Z}}Q_{k,j}.       $$
Using that  $\Psi$ is of upper-type $q$, we easily obtain
\begin{align*}
\int_{\mathbb{R}^{d}}\Psi\left( \mathcal{M}_{\overrightarrow{\sigma}}^{\mathcal{D}^{\beta}} \overrightarrow{g}(x)\right)\omega(x)dx 
&\lesssim \sum_{Q\in \mathcal{D}^{\beta}}\lambda_{Q}\Psi\left(\prod_{i=1}^{n}\frac{1}{\sigma_{i}(Q)  }\int_{Q}|g_{i}(x)|\sigma_{i}(x)dx\right),    
\end{align*}
where 
\[ \lambda_{Q} :=
\left\{
\begin{array}{ll} \omega(E_{Q}) & \mbox{ if $Q\in \mathcal{S}^\beta$}\\ 0 & \mbox{ otherwise}.
\end{array}
\right. \]
Let us put $$\nu_{\overrightarrow{\sigma}}:=  \frac{1}{\Phi\left(\prod_{i=1}^{n}\Phi_{i}^{-1}\left(\frac{1}{\sigma_{i}}\right)  \right)}.$$
Let $R \in \mathcal{D}^{\beta}$.
As  $(\overrightarrow{\sigma}, \omega) \in M_{\overrightarrow{\Phi},\Psi}$ and since the $E_{Q}$ are pairwise disjoint, we obtain 
 
$$ \sum_{Q\subset R, Q \in \mathcal{D}^{\beta}}\lambda_{Q}= \sum_{ Q\subset R, Q\in\mathcal{S}^\beta}\omega(E_{Q})  =\omega \left(\bigcup_{Q\subset R, Q\in\mathcal{S}^\beta} E_{Q}\right) \leq \omega(R) \leq \frac{[\overrightarrow{\sigma},\omega]_{M_{\overrightarrow{\Phi},\Psi}}}{\Psi \circ \Phi^{-1}\left( \frac{1}{\nu_{\overrightarrow{\sigma}}(R)}\right)}. $$
Thus by Theorem \ref{pro:main16yh6}, we have 
$$  \sum_{Q\in \mathcal{D}^{\beta}}\lambda_{Q}\Psi\left(\prod_{i=1}^{n}\frac{1}{\sigma_{i}(Q)  }\int_{Q}|g_{i}(x)|\sigma_{i}(x)dx\right) \lesssim [\overrightarrow{\sigma},\omega]_{M_{\overrightarrow{\Phi},\Psi}}.        $$
The proof is complete.
\end{proof}

\begin{proposition}\label{pro:mainaaqwqyh6}
Let $\sigma_{1},...,\sigma_{n},\omega$ be weights on $\mathbb{R}^{d}$, and $\Phi_{1},...,\Phi_{n},\Psi\in \mathscr{U}$. Let $\Phi$ be a one-to-one correspondence from $\mathbb{R}_{+}$ to itself such that 
 $\Phi^{-1}=\Phi_{1}^{-1}\times...\times \Phi_{n}^{-1}$, and $t\mapsto \frac{\Psi(t)}{\Phi(t)}$ is increasing on $\mathbb{R}_{+}^{*}$. Let us put  $\nu_{\overrightarrow{\sigma}}:=  \frac{1}{\Phi\left(\prod_{i=1}^{n}\Phi_{i}^{-1}\left(\frac{1}{\sigma_{i}}\right)  \right)}.$ If (\ref{eq:surmaron}) is satisfied and  $\mathcal{M}_{\overrightarrow{\sigma}} : L^{\Phi_{1}}(\sigma_{1})\times...\times L^{\Phi_{n}}(\sigma_{n})\longrightarrow L^{\Psi}(\omega)$ boundedly, then $(\overrightarrow{\sigma}, \omega) \in M_{\overrightarrow{\Phi},\Psi}$.
\end{proposition}

\begin{proof}
Let $R\in \mathcal{Q}$. 
As (\ref{eq:surmaron}) holds and $\mathcal{M}_{\overrightarrow{\sigma}}: L^{\Phi_{1}}(\sigma_{1}) \times...\times L^{\Phi_{n}}(\sigma_{n})\longrightarrow  L^{\Psi}(\omega)    $ boundedly, we have
\begin{align*}
\omega(R)\Psi \circ \Phi^{-1}\left( \frac{1}{\nu_{\overrightarrow{\sigma}}(R)}\right) &= \int_{R}\Psi \circ \Phi^{-1}\left( \frac{1}{\nu_{\overrightarrow{\sigma}}(R)}\right)\omega(x)dx \\
&\lesssim \int_{R}\Psi \circ \Phi^{-1}\left( \Phi\left(\prod_{i=1}^{n}\Phi_{i}^{-1}\left(\frac{1}{\sigma_{i}(R)}\right)\right)\right)\omega(x)dx\\
&=\int_{R}\Psi\left(\prod_{i=1}^{n}\frac{1}{\sigma_{i}(R)}\int_{R}\frac{\chi_{R}(y)}{\|\chi_{R}\|_{L_{\sigma_{i}}^{\Phi_{i}}}^{lux} }\sigma_{i}(y)dy\right)\omega(x)dx\\
&\lesssim \int_{R}\Psi\left( \mathcal{M}_{\overrightarrow{\sigma}}\left(\frac{\chi_{R}}{\|\chi_{R}\|_{L_{\sigma_{1}}^{\Phi_{1}}}^{lux}},...,\frac{\chi_{R}}{\|\chi_{R}\|_{L_{\sigma_{n}}^{\Phi_{n}}}^{lux}}\right)(x)\right)\omega(x)dx\\
&\lesssim 1. 
\end{align*}
\end{proof}

Recall that if for any $1\leq i \leq n$, $\sigma_{i}\equiv \sigma$,
then the weights $\nu_{\overrightarrow{\sigma}}$ and $\sigma$ are identical. In this case, under our hypotheses, (\ref{eq:surmaron}) is always satisfied. Also  $\sigma$ is doubling and so $\mathcal{M}_{\overrightarrow{\sigma}} : L^{\Phi_{1}}(\sigma)\times...\times L^{\Phi_{n}}(\sigma)\longrightarrow L^{\Psi}(\omega)$ boundedly if and only if  $(\overrightarrow{\sigma},\omega)\in M_{\overrightarrow{\Phi},\Psi}$. Indeed, if $\mathcal{M}_{\overrightarrow{\sigma}} : L^{\Phi_{1}}(\sigma)\times...\times L^{\Phi_{n}}(\sigma)\longrightarrow L^{\Psi}(\omega)$ boundedly, then  $(\overrightarrow{\sigma},\omega)\in M_{\overrightarrow{\Phi},\Psi}$ (see Proposition \ref{pro:mainaaqwqyh6}). Conversely, by replacing Theorem \ref{pro:main16yh6} by Corollary \ref{pro:main16aaqq6} in the arguments in the proof of Theorem \ref{pro:maing1}, we obtain that if  $(\overrightarrow{\sigma},\omega)\in M_{\overrightarrow{\Phi},\Psi}$, then $\mathcal{M}_{\overrightarrow{\sigma}} : L^{\Phi_{1}}(\sigma)\times...\times L^{\Phi_{n}}(\sigma)\longrightarrow L^{\Psi}(\omega)$ boundedly.  This proves Corollary \ref{pro:maingp1}.
\vskip .2cm
Let us now prove Theorem \ref{thm:LiSuaqn}.
\begin{proof}[Proof of Theorem \ref{thm:LiSuaqn}]
Suppose that     $\mathcal{M}_{\overrightarrow{\sigma}}$ is bounded from $L^{\Phi_1}(\sigma_1)\times\cdots\times L^{\Phi_n}(\sigma_n)$ to $L^\Psi(\omega)$.\\
Let $R\in \mathcal{Q}$. Then we have 
\begin{align*}
&\omega(R)\left(\prod_{i=1}^n\Psi\circ\Phi_i^{-1}\left(\frac{1}{\sigma_i(R)}\right)\right) \\
&= \left(\prod_{i=1}^n\Psi\circ\Phi_i^{-1}\left(\frac{1}{\sigma_i(R)}\right)\right)\int_{R}\omega(x)dx\\
&=\left(\prod_{i=1}^n\Psi\circ\Phi_i^{-1}\left(\frac{1}{\sigma_i(R)}\right)\right)\int_{R}\Psi\left(\prod_{i=1}^{n}\|\chi_{R}\|_{L_{\sigma_{i}}^{\Phi_{i}}}^{lux} \frac{1}{\sigma_{i}(R)}\int_{R}\frac{\chi_{R}(y)}{\|\chi_{R}\|_{L_{\sigma_{i}}^{\Phi_{i}}}^{lux} }\sigma_{i}(y)dy\right)\omega(x)dx\\
&\lesssim \left(\prod_{i=1}^n\Psi\circ\Phi_i^{-1}\left(\frac{1}{\sigma_i(R)}\right)\right)\times \Psi\left(\prod_{i=1}^{n}\|\chi_{R}\|_{L_{\sigma_{i}}^{\Phi_{i}}}^{lux} \right)\int_{R}\Psi\left( \mathcal{M}_{\overrightarrow{\sigma}}\left(\frac{\chi_{R}}{\|\chi_{R}\|_{L_{\sigma_{1}}^{\Phi_{1}}}^{lux}},...,\frac{\chi_{R}}{\|\chi_{R}\|_{L_{\sigma_{n}}^{\Phi_{n}}}^{lux}}\right)(x)\right)\omega(x)dx\\
&= \left(\prod_{i=1}^n\Psi\circ\Phi_i^{-1}\left(\frac{1}{\sigma_i(R)}\right)\right)\times \Psi\left(\prod_{i=1}^{n}\frac{1}{\Phi_i^{-1}\left(\frac{1}{\sigma_i(R)}\right)}\right)\int_{R}\Psi\left( \mathcal{M}_{\overrightarrow{\sigma}}\left(\frac{\chi_{R}}{\|\chi_{R}\|_{L_{\sigma_{1}}^{\Phi_{1}}}^{lux}},...,\frac{\chi_{R}}{\|\chi_{R}\|_{L_{\sigma_{n}}^{\Phi_{n}}}^{lux}}\right)(x)\right)\omega(x)dx\\
&\lesssim \left(\prod_{i=1}^n\Psi\circ\Phi_i^{-1}\left(\frac{1}{\sigma_i(R)}\right)\right)\times \prod_{i=1}^{n}\Psi\left(\frac{1}{\Phi_i^{-1}\left(\frac{1}{\sigma_i(R)}\right)}\right)\int_{R}\Psi\left( \mathcal{M}_{\overrightarrow{\sigma}}\left(\frac{\chi_{R}}{\|\chi_{R}\|_{L_{\sigma_{1}}^{\Phi_{1}}}^{lux}},...,\frac{\chi_{R}}{\|\chi_{R}\|_{L_{\sigma_{n}}^{\Phi_{n}}}^{lux}}\right)(x)\right)\omega(x)dx\\
&\lesssim \left(\prod_{i=1}^n\Psi\circ\Phi_i^{-1}\left(\frac{1}{\sigma_i(R)}\right)\right)\times \prod_{i=1}^{n}\frac{1}{\Psi\circ\Phi_i^{-1}\left(\frac{1}{\sigma_i(R)}\right)}\int_{R}\Psi\left( \mathcal{M}_{\overrightarrow{\sigma}}\left(\frac{\chi_{R}}{\|\chi_{R}\|_{L_{\sigma_{1}}^{\Phi_{1}}}^{lux}},...,\frac{\chi_{R}}{\|\chi_{R}\|_{L_{\sigma_{n}}^{\Phi_{n}}}^{lux}}\right)(x)\right)\omega(x)dx \\
&\lesssim 1.
\end{align*}
Conversely, assume that  $[ {\vec{\sigma}},\omega]_{K_{\vec {\Phi},\Psi}}$ is finite. Without loss of generality, we can fix $n=2$. \\
Let $(f_{1},f_{2})\in L^{\Phi_{1}}(\sigma_{1})\times L^{\Phi_{2}}(\sigma_{2})$. Suppose that for $ i=1,2$,  $f_{i}\not\equiv 0$ and put $g_{i}=\frac{f_{i}}{\|f_{i}\|_{L_{\sigma_{i}}^{\Phi_{i}}}^{lux}}$.
As in the previous result, we only have to estimate 
$$\int_{\mathbb{R}^{d}}\Psi\left( \mathcal{M}_{\overrightarrow{\sigma}}^{\mathcal{D}^{\beta}}(g_{1},g_{2})(x)\right)\omega(x)dx.$$
Let $\mathcal{S}^\beta=\{Q_{k,j}\}_{j\in \mathbb{Z}}$ be the family of maximal (with respect to inclusion) dyadic cubes such that 
$$   \prod_{i=1}^{2}\frac{1}{\sigma_{i}(Q_{k,j})  }\int_{Q_{k,j}}|g_{i}(x)|\sigma_{i}(x)dx > 2^{k}, ~~\forall~j \in \mathbb{Z}.           $$
We have seen in the previous theorem that $\mathcal{S}^\beta$ is a sparse family.
As  $\Psi$ is of upper-type $q$, we obtain
\begin{align*}
&\int_{\mathbb{R}^{d}}\Psi\left( \mathcal{M}_{\overrightarrow{\sigma}}^{\mathcal{D}^{\beta}}(g_{1},g_{2})(x)\right)\omega(x)dx 
= \sum_{Q\in\mathcal{S}^\beta}\int_{E_Q}\Psi\left( \mathcal{M}_{\overrightarrow{\sigma}}^{\mathcal{D}^{\beta}}(g_{1},g_{2})(x)\right)\omega(x)dx \\
&\lesssim \sum_{Q\in\mathcal{S}^\beta}\Psi\left(\prod_{i=1}^{2}\frac{1}{\sigma_{i}(Q)  }\int_{Q}|g_{i}(x)|\sigma_{i}(x)dx\right)\omega(E_Q)
\lesssim \sum_{Q\in\mathcal{S}^\beta}\prod_{i=1}^{2}\Psi\left(\frac{1}{\sigma_{i}(Q)  }\int_{Q}|g_{i}(x)|\sigma_{i}(x)dx\right)\omega(E_Q)\\
&= \sum_{Q\in \mathcal{S}^\beta}\Psi\left(\frac{1}{\sigma_{1}(Q)  }\int_{Q}|g_{1}(x)|\sigma_{1}(x)dx\right)\Psi\left(\frac{1}{\sigma_{2}(Q)  }\int_{Q}|g_{2}(x)|\sigma_{2}(x)dx\right)\omega(E_Q)\\
&= \sum_{Q\in \mathcal{D}^{\beta}}\lambda_{Q}\Psi\left(\frac{1}{\sigma_{1}(Q)  }\int_{Q}|g_{1}(x)|\sigma_{1}(x)dx\right),    
\end{align*}
where
\[ \lambda_{Q} :=
\left\{
\begin{array}{ll} \omega(E_{Q})\Psi\left(\frac{1}{\sigma_{2}(Q)  }\int_{Q}|g_{2}(x)|\sigma_{2}(x)dx\right) & \mbox{ if $Q\in\mathcal{S}^\beta$}\\ 0 & \mbox{ otherwise}.
\end{array}
\right. \]
Let $R \in \mathcal{D}^{\beta}$ be fixed. We have 
\begin{eqnarray*}  \sum_{Q\subset R, Q \in \mathcal{D}^{\beta}}\lambda_{Q} &=& \sum_{ Q\subset R, Q\in\mathcal{S}^\beta}\Psi\left(\frac{1}{\sigma_{2}(Q)  }\int_{Q}|g_{2}(x)|\sigma_{2}(x)dx\right)\omega(E_Q)\\ &=& \sum_{Q\subset R, Q \in \mathcal{D}^{\beta}}\lambda_{Q}'\Psi\left(\frac{1}{\sigma_{2}(Q)  }\int_{Q}|g_{2}(x)|\sigma_{2}(x)dx\right),    
\end{eqnarray*}
where  

\[ \lambda_{Q}' :=
\left\{
\begin{array}{ll} \omega(E_{Q}) & \mbox{ if $Q\in\mathcal{S}^\beta$}\\ 0 & \mbox{ otherwise}.
\end{array}
\right. \]
As $t\mapsto\frac{\Psi(t)}{\Phi_1(t)}$ is nondecreasing on $\mathbb{R}^{*}_{+}$, we have that $t\mapsto\frac{1}{\Psi\circ\Phi_{1}^{-1}(\frac{1}{t})}$ is also nondecreasing on  $\mathbb{R}^{*}_{+}$. 
Let $B \in \mathcal{D}^{\beta}$ and such that $B\subset R$. 
Since  $[ {\vec{\sigma}},\omega]_{K_{\vec {\Phi},\Psi}}$ is finite and as the $E_{Q}$ are pairwise disjoint, 
we have
\begin{align*}
\sum_{Q\subset B, Q \in \mathcal{D}^{\beta}}\lambda_{Q}'&
= \sum_{ Q\subset B, Q\in\mathcal{S}^\beta}\omega(E_Q)  =\omega \left(\bigcup_{Q\subset B, Q\in\mathcal{S}^\beta} E_Q\right) \\
&\leq \omega(B) \leq \frac{[\overrightarrow{\sigma},\omega]_{K_{\overrightarrow{\Phi},\Psi}}}{\prod_{i=1}^{2}\Psi \circ \Phi_{i}^{-1}\left( \frac{1}{\sigma_{i}(B)}\right)} \leq  \frac{[\overrightarrow{\sigma},\omega]_{K_{\overrightarrow{\Phi},\Psi}}}{\Psi \circ \Phi_{1}^{-1}\left( \frac{1}{\sigma_{1}(R)}\right)}\times \frac{1}{\Psi \circ \Phi_{2}^{-1}\left( \frac{1}{\sigma_{2}(B)}\right)}.
\end{align*}
Hence $\{\lambda_{Q}'\}_{Q \in \mathcal{D}^{\beta}}$ is a $(\sigma_{2}, \Psi \circ \Phi_{2}^{-1})-$Carleson sequence with Carleson constant $$   \Lambda'_{Carl} \lesssim  \frac{[\overrightarrow{\sigma},\omega]_{K_{\overrightarrow{\Phi},\Psi}}}{\Psi \circ \Phi_{1}^{-1}\left( \frac{1}{\sigma_{1}(R)}\right)}.     $$ 
Thus by Theorem \ref{pro:main16aaqaqq6}, we have that 
$$  \sum_{Q\subset R, Q \in \mathcal{D}^{\beta}}\lambda_{Q}= \sum_{Q\subset R, Q \in \mathcal{D}^{\beta}}\lambda_{Q}'\Psi\left(\frac{1}{\sigma_{2}(Q)  }\int_{Q}|g_{2}(x)|\sigma_{2}(x)dx\right) \lesssim   \frac{[\overrightarrow{\sigma},\omega]_{K_{\overrightarrow{\Phi},\Psi}}}{\Psi \circ \Phi_{1}^{-1}\left( \frac{1}{\sigma_{1}(R)}\right)}.  $$
We deduce that $\{\lambda_{Q}\}_{Q \in \mathcal{D}^{\beta}}$ is a $(\sigma_{1}, \Psi \circ \Phi_{1}^{-1})-$Carleson sequence with Carleson constant $$   \Lambda_{Carl} \lesssim [\overrightarrow{\sigma},\omega]_{K_{\overrightarrow{\Phi},\Psi}}.    $$
It follows from Theorem \ref{pro:main16aaqaqq6} that 
$$ \int_{\mathbb{R}^{d}}\Psi\left( \mathcal{M}_{\overrightarrow{\sigma}}^{\mathcal{D}^{\beta}}(g_{1},g_{2})(x)\right)\omega(x)dx \lesssim \sum_{Q\in \mathcal{D}^{\beta}}\lambda_{Q}\Psi\left(\frac{1}{\sigma_{1}(Q)  }\int_{Q}|g_{1}(x)|\sigma_{1}(x)dx\right) \lesssim [\overrightarrow{\sigma},\omega]_{K_{\overrightarrow{\Phi},\Psi}}.        $$
Hence $$ \|\mathcal{M}_{\overrightarrow{\sigma}}(f_{1},f_{2})\|_{L_{\omega}^{\Psi}}^{lux} \lesssim \Psi^{-1}\left( [\overrightarrow{\sigma},\omega]_{K_{\overrightarrow{\Phi},\Psi}}\right)\prod_{i=1}^{2}\|f_{i}\|_{L_{\sigma_{i}}^{\Phi_{i}}}^{lux}.       $$     
The proof is complete.
\end{proof}

\subsection{Proofs of Sawyer-type results}
Let us start by observing that from Lemma \ref{pro:main80} one has the following. 
\begin{equation}\label{eq:of895ua1l}
 \mathcal{M}_{\alpha} (\overrightarrow{\sigma}.\overrightarrow{g}) \leq 6^{(nd-\alpha)} \sum_{\beta \in \{0, 1/3\}^{d} }   \mathcal{M}^{\mathcal{D}^{\beta}}_{\alpha} (\overrightarrow{\sigma}.\overrightarrow{g}).
\end{equation}
It follows from the above and the convexity of the Orlicz functions in consideration that to estimate $\mathcal{M}_\alpha$, it is enough to estimate its dyadic versions $\mathcal{M}_\alpha^{\mathcal D^\beta}$.
\vskip .2cm
Let us prove Theorem \ref{thm:LiSun}.

\begin{proof}[Proof of Theorem \ref{thm:LiSun}]
We restrict ourself to the bilinear case  as the general case follows the same. We start by proving the following.
\begin{lemma}\label{lem:LiSun1}
Suppose that $0\le \alpha<2d$. Let $\Phi_1,\Phi_2, \Psi\in \mathscr U$ be such that $\Phi_2\in \nabla_2$ ,   $\Psi\in \Delta'$  and  $t\mapsto \frac{\Psi(t)}{\Phi_i(t)}$ is nondecreasing for any i=1,2. Let $\sigma_1,\sigma_2,\omega$ be three weights. Then if $g$ is a function with $\supp g\subset R\in \mathcal D^\beta$ and $\|g\|_{\Phi_2,\sigma_2}^{lux}=1$, then
\Be\label{eq:LiSun}
\int_R\Psi\left(\mathcal {M}_{\alpha}^{\mathcal D^\beta}(\chi_R\sigma_1,g\sigma_2)(x)\right)\omega(x)dx\lesssim \frac{[\vec {\sigma},\omega]_{ L_{\vec \Phi,\Psi}}}{\Psi\circ\Phi_1^{-1}\left(\frac{1}{\sigma_1(R)}\right)}
\Ee
\end{lemma}
\begin{proof}
Let $a>2^{2d-\alpha}$. To each integer $k$, associate the set $$\Omega_k:=\{x\in \mathbb R^d:a^k<\mathcal M_{\alpha}^{\mathcal D^\beta}(\chi_R\sigma_1,g\sigma_2)(x)\le a^{k+1}\}.$$
There exists a family $\mathcal{S}=\{Q_{k,j}\}_{j\in \mathbb N}$   
of dyadic cubes maximal with respect to the inclusion and such that  $$\frac{1}{|Q_{k,j}|^{2-\frac{\alpha}{d}}}\int_{Q_{k,j}}\chi_R(x)\sigma_1(x)dx\int_{Q_{k,j}}|g(x)|\sigma_2(x)dx>a^k$$
so that $$\Omega_k\subseteq \cup_{j\in \mathbb N}Q_{k,j}.$$
Following the usual arguments and using that $\Psi$ is of some upper-type $q$, we obtain that
\Beas
L_R(g) &:=& \int_R\Psi\left(\mathcal {M}_{\alpha}^{\mathcal D^\beta}(\chi_R\sigma_1,g\sigma_2)(x)\right)\omega(x)dx\\ &\lesssim& a^q\sum_{Q\subseteq R, Q\in \mathcal{S}}\Psi\left(\frac{\int_{Q}\chi_R\sigma_1\int_{Q}|g|\sigma_2}{|Q|^{2-\frac{\alpha}{d}}}\right)\omega(E_Q)\\ &=& a^q\sum_{Q\subset R, Q\in \mathcal{S}}\Psi\left(\frac{\int_{Q}\chi_R\sigma_1\int_{Q}|g|\sigma_2}{|Q|^{2-\frac{\alpha}{d}}}\right)\omega(E_Q)\\ &+& a^q\Psi\left(\frac{\int_{R}\chi_R\sigma_1\int_{R}|g|\sigma_2}{|R|^{2-\frac{\alpha}{d}}}\right)\omega(E_R)\\ &=& a^q(T_1+T_2)
\Eeas
where $$T_1=\sum_{Q\subset R, Q\in \mathcal{S}}\Psi\left(\frac{\int_{Q}\chi_R\sigma_1\int_{Q}|g|\sigma_2}{|Q|^{2-\frac{\alpha}{d}}}\right)\omega(E_Q),$$
and
$$T_2=\Psi\left(\frac{\int_{R}\chi_R\sigma_1\int_{R}|g|\sigma_2}{|R|^{2-\frac{\alpha}{d}}}\right)\omega(E_R).$$
Let us start by estimating the term $T_2$. We denote by $\Psi_2$ the complementary function of $\Phi_2$. We obtain
\Beas
T_2 &=& \Psi\left(\frac{\int_{R}\chi_R\sigma_1\int_{R}|g|\sigma_2}{|R|^{2-\frac{\alpha}{d}}}\right)\omega(E_R)\\ &\lesssim& \Psi\left(\frac{1}{\sigma_2(R)}\int_{R}|g|\sigma_2\right)\int_R\Psi\left(\mathcal {M}_{\alpha}^{\mathcal D^\beta}(\chi_R\sigma_1,\chi_R\sigma_2)\right)\omega(x)dx\\ &\lesssim& \Psi\left(\frac{1}{\sigma_2(R)}\|g\|_{\Phi_2,\sigma_2}^{lux}\|\chi_R\|_{\Psi_2,\sigma_2}^{lux}\right)\int_R\Psi\left(\mathcal {M}_{\alpha}^{\mathcal D^\beta}(\chi_R\sigma_1,\chi_R\sigma_2)\right)\omega(x)dx\\ &\lesssim& \Psi\left(\frac{1}{\sigma_2(R)\Psi_2^{-1}\left(\frac{1}{\sigma_2(R)}\right)}\right)\int_R\Psi\left(\mathcal {M}_{\alpha}^{\mathcal D^\beta}(\chi_R\sigma_1,\chi_R\sigma_2)\right)\omega(x)dx\\ &\lesssim& \Psi\circ\Phi_2^{-1}\left(\frac{1}{\sigma_2(R)}\right)\int_R\Psi\left(\mathcal {M}_{\alpha}^{\mathcal D^\beta}(\chi_R\sigma_1,\chi_R\sigma_2)\right)\omega(x)dx\\ &\lesssim& \frac{[\vec {\sigma},\omega]_{L_{\vec \Phi,\Psi}}}{\Psi\circ\Phi_1^{-1}\left(\frac{1}{\sigma_1(R)}\right)}.
\Eeas
In the above, we have used that $\frac{t}{\Psi_2^{-1}(t)}\approx \Phi_2^{-1}(t)\quad\forall t>0$.
\vskip .2cm
We observe that
\Beas
T_1 &=&\sum_{Q\subset R, Q\in \mathcal{S}}\Psi\left(\frac{\int_{Q}\chi_R\sigma_1\int_{Q}|g|\sigma_2}{|Q|^{2-\frac{\alpha}{d}}}\right)\omega(E_Q)\\ &\lesssim& \sum_{Q\subset R, Q\in \mathcal{S}}\Psi\left(m_{\sigma_2}(|g|,Q)\right)\Psi\left(\frac{\int_{Q}\chi_R\sigma_1\int_{Q}\chi_R\sigma_2}
{|Q|^{2-\frac{\alpha}{d}}}\right)\omega(E_Q)\\ &=& \sum_{Q\in \mathcal D^\beta}\lambda_Q\Psi\left(m_{\sigma_2}(|g|,Q)\right)
\Eeas
where
$$
\lambda_Q:=\left\{ \begin{matrix} \Psi\left(\frac{\int_{Q}\chi_R\sigma_1\int_{Q}\chi_R\sigma_2}
{|Q|^{2-\frac{\alpha}{d}}}\right)\omega(E_Q) &\text{if }& Q\subset R, Q\in \mathcal{S}\\
      0 & \text{ otherwise} .
                                  \end{matrix} \right.
$$
Let us prove that $\{\lambda_Q\}_{Q\in \mathcal D}$ is $(\sigma_2,\Psi\circ\Phi_2^{-1})$-Carleson sequence. Let $K\in \mathcal D^\beta$, $K\subset R$. Then
\Beas
\sum_{Q\in \mathcal D^\beta, Q\subseteq K}\lambda_Q &=& \sum_{Q\in \mathcal D^\beta, Q\subseteq K}\Psi\left(\frac{\int_{Q}\chi_R\sigma_1\int_{Q}\chi_R\sigma_2}
{|Q|^{2-\frac{\alpha}{d}}}\right)\omega(E_Q)\\ &=& \sum_{Q\in \mathcal D^\beta, Q\subseteq K}\int_{E_Q}\Psi\left(\frac{\int_{Q}\chi_R\sigma_1\int_{Q}\chi_R\sigma_2}
{|Q|^{2-\frac{\alpha}{d}}}\right)\omega(x)dx\\ &\le& \int_K\Psi\left(\mathcal {M}_{\alpha}^{\mathcal D^\beta}(\chi_R\sigma_1,\chi_R\sigma_2)(x)\right)\omega(x)dx\\ &\le& \frac{1}{\Psi\circ\Phi_1^{-1}\left(\frac{1}{\sigma_1(K)}\right)}[\vec {\sigma},\omega]_{L_{\vec \Phi,\Psi}}\frac{1}{\Psi\circ\Phi_2^{-1}\left(\frac{1}{\sigma_2(K)}\right)}
\\ &\le& \frac{1}{\Psi\circ\Phi_1^{-1}\left(\frac{1}{\sigma_1(R)}\right)}[\vec {\sigma},\omega]_{L_{\vec \Phi,\Psi}}\frac{1}{\Psi\circ\Phi_2^{-1}\left(\frac{1}{\sigma_2(K)}\right)}
\Eeas
That is $\{\lambda_Q\}_{Q\in \mathcal D}$ is $(\sigma_2,\Psi\circ\Phi_2^{-1})$-Carleson sequence with Carleson constant  $$\Lambda_{Carl}\lesssim \frac{[\vec {\sigma},\omega]_{L_{\vec \Phi,\Psi}}}{\Psi\circ\Phi_1^{-1}\left(\frac{1}{\sigma_1(R)}\right)}.$$ Thus by Corollary \ref{pro:main16aaqaqq6},
$$T_1\lesssim \frac{[\vec {\sigma},\omega]_{L_{\vec \Phi,\Psi}}}{\Psi\circ\Phi_1^{-1}\left(\frac{1}{\sigma_1(R)}\right)}.$$
From the estimates of $T_1$ and $T_2$, we conclude that
$$\int_R\Psi\left(\mathcal {M}_{\alpha}^{\mathcal D^\beta}(\chi_R\sigma_1,g\sigma_2)(x)\right)\omega(x)dx\lesssim \frac{[\vec {\sigma},\omega]_{L_{\vec \Phi,\Psi}}}{\Psi\circ\Phi_1^{-1}\left(\frac{1}{\sigma_1(R)}\right)}.$$
The proof is complete.
\end{proof}
Theorem \ref{thm:LiSun} in the bilinear case is a consequence of its following dyadic version.
\begin{proposition}\label{prop:LiSunsuffdyabili}
Given $\Phi_i\in \mathscr U\cap\nabla_2$, $i=1,2$, and $\Psi\in \widetilde{\mathscr U}$, suppose that $0\le \alpha<2d$, and $t\mapsto\frac{\Psi(t)}{\Phi_i(t)}$ is nondecreasing for $i=1,2$. Let $\sigma_1,\sigma_2$ and $v$ be weights. Define
$$[ {\vec{\sigma}},v]_{L_{\vec {\Phi},\Psi}}:=\sup_{Q\in \mathcal Q}\left(\prod_{i=1}^2\Psi\circ\Phi_i^{-1}\left(\frac{1}{\sigma_i(Q)}\right)\right)\left(\int_Q\Psi\left(\mathcal {M}_\alpha(\sigma_1\chi_Q,\sigma_2\chi_Q)(x)\right)v(x)dx\right).$$
Then $\mathcal M_\alpha^{\mathcal D^\beta}(\vec{\sigma}.)$ is bounded from $L^{\Phi_1}(\sigma_1)\times L^{\Phi_2}(\sigma_2)$ to $L^\Psi(v)$ if $[ {\vec{\sigma}},v]_{L_{\vec {\Phi},\Psi}}$ is finite. Moreover,
$$\|\mathcal M_\alpha^{\mathcal D^\beta}(\vec{\sigma}.)\|_{\left(\prod_{i=1}^2L^{\Phi_i}(\sigma_i)\right)\rightarrow L^\Psi(v)}\lesssim \Psi^{-1}\left([ {\vec{\sigma}},v]_{L_{\vec {\Phi},\Psi}}\right).$$
\end{proposition}
\begin{proof}
Let $(f_1,f_2)\in L^{\Phi_1}(\sigma_1)\times L^{\Phi_2}(\sigma_2)$ with $\|f_1\|_{\Phi_1,\sigma_1}^{lux}=1=\|f_2\|_{\Phi_2,\sigma_2}^{lux}$, Following for example \cite{sehbaf}, we obtain that there is a sparse family $\mathcal{S}^\beta$ such that
\Beas
L(f_1,f_2) &:=& \int_{\mathbb R^d}\Psi\left(\mathcal {M}_{\alpha}^{\mathcal D^\beta}(f_1\sigma_1,f_2\sigma_2)(x)\right)\omega(x)dx\\ &\le& \sum_{Q\in \mathcal{S}^\beta}\Psi\left(\frac{\int_{Q}|f_1|\sigma_1\int_{Q}|f_2|\sigma_2}{|Q|^{2-\frac{\alpha}{d}}}\right)\omega(E_Q)\\ &=& \sum_{Q\in \mathcal{S}^\beta}\Psi\left(m_{\sigma_1}(|f_1|,Q)\right)\Psi\left(\frac{\sigma_1(Q)\int_{Q}|f_2|\sigma_2}{|Q|^{2-\frac{\alpha}{d}}}\right)
\omega(E_Q)\\ &=& \sum_{Q\in \mathcal D^\beta}\lambda_Q\Psi\left(m_{\sigma_1}(|f_1|,Q)\right)
\Eeas
where
$$
\lambda_Q:=\left\{ \begin{matrix} \Psi\left(\frac{\sigma_1(Q)\int_{Q}|f_2|\sigma_2}
{|Q|^{2-\frac{\alpha}{d}}}\right)\omega(E_Q) &\text{if }& Q\in \mathcal{S}^\beta,\\
      0 & \text{ otherwise} .
                                  \end{matrix} \right.
$$
Let us prove that $\{\lambda_Q\}_{Q\in \mathcal D^\beta}$ is $(\sigma_1,\Psi\circ\Phi_1^{-1})$-Carleson sequence.

For  $R\in \mathcal D^\beta$ given, we obtain using Lemma \ref{lem:LiSun1} that
\Beas
\sum_{Q\in \mathcal D^\beta, Q\subseteq R}\lambda_Q &=& \sum_{Q\in \mathcal S^\beta, Q\subseteq R}\Psi\left(\frac{\sigma_1(Q)\int_{Q}|f_2|\sigma_2}
{|Q|^{2-\frac{\alpha}{d}}}\right)\omega(E(Q))\\ &=& \sum_{Q\in \mathcal S^\beta, Q\subseteq R}\Psi\left(\frac{\sigma_1(Q)\int_{Q}\chi_R|f_2|\sigma_2}
{|Q|^{2-\frac{\alpha}{d}}}\right)\omega(E(Q))\\ &\le& \int_R\Psi\left(\mathcal {M}_{\alpha}^{\mathcal D^\beta}(\chi_R\sigma_1,\chi_R|f_2|\sigma_2)(x)\right)\omega(x)dx\\ &\le& [\vec {\sigma},\omega]_{ L_{\vec \Phi,\Psi}}\frac{1}{\Psi\circ\Phi_1^{-1}\left(\frac{1}{\sigma_1(R)}\right)}.
\Eeas
That is $\{\lambda_Q\}_{Q\in \mathcal D}$ is a $(\sigma_1,\Psi\circ\Phi_1^{-1})$-Carlseon sequence with Carleson constant $$\Lambda_{Carl}\lesssim [\vec {\sigma},\omega]_{L_{\vec \Phi,\Psi}}.$$  Thus using Corollary \ref{pro:main2aq09}, we conclude that
$$\int_{\mathbb R^n}\Psi\left(\mathcal {M}_{\alpha}^{\mathcal D^\beta}(f_1\sigma_1,f_2\sigma_2)(x)\right)\omega(x)dx\lesssim [\vec {\sigma},\omega]_{L_{\vec \Phi,\Psi}}.$$
The proof is complete.

\end{proof}
Theorem \ref{thm:LiSun} then follows from the above and Proposition \ref{pro:main17jo}. The proof is complete.
\end{proof}


\begin{proof}[Proof of Theorem \ref{pro:main126aq}]
Let us assume that $(\overrightarrow{\sigma}, \omega) \in S_{\overrightarrow{\Phi},\Psi}^{\alpha}$.
Let $(f_{1},...,f_{n})\in L^{\Phi_{1}}(\sigma_{1})\times...\times L^{\Phi_{n}}(\sigma_{n})$. Suppose that for any $1\leq i \leq n$,  $f_{i}\not\equiv 0$. Define   $\overrightarrow{g}:=(g_{1},...,g_{n})$ with $g_{i}=\frac{f_{i}}{\|f_{i}\|_{L_{\sigma_{i}}^{\Phi_{i}}}^{lux}}$,  for $1\leq i \leq n$. 
For $\beta \in \left\{0, 1/3\right\}^{d}$ fixed, using the same sparse family $\mathcal{S}^\beta$ as in the case of power functions in \cite{sehbaf}, and that  
    $\Psi\in \Delta'$, we obtain  
\begin{eqnarray*}
\int_{\mathbb{R}^{d}}\Psi\left( \mathcal{M}_{\alpha}^{\mathcal{D}^{\beta}}\left(\overrightarrow{\sigma}.\overrightarrow{g}\right)(x)\right)\omega(x)dx   &\lesssim&  \sum_{Q\in \mathcal{S}^\beta}\Psi\left(\prod_{i=1}^{n}\frac{1}{|Q|^{1-\frac{\alpha}{nd}}}\int_{Q}|g_{i}(x)|\sigma_{i}(x)dx \right)\omega(E_Q)\\
&=& \sum_{Q\in \mathcal{S}^\beta}\Psi\left(\prod_{i=1}^{n}\frac{\sigma_{i}(Q)}{|Q|^{1-\frac{\alpha}{nd}}}\times \frac{1}{\sigma_{i}(Q)}\int_{Q}|g_{i}(x)|\sigma_{i}(x)dx \right)\omega(E_{Q})\\
&=& \sum_{Q\in \mathcal{S}^\beta}\Psi\left(\prod_{i=1}^{n}\frac{\sigma_{i}(Q)}{|Q|^{1-\frac{\alpha}{nd}}}\times m_{\sigma_{i}}(g_{i},Q) \right)\omega(E_Q)\\
&\lesssim&  \sum_{Q\in \mathcal{S}^\beta}\Psi\left(\prod_{i=1}^{n} m_{\sigma_{i}}(g_{i},Q) \right)\Psi\left(\prod_{i=1}^{n}\frac{\sigma_{i}(Q)}{|Q|^{1-\frac{\alpha}{nd}}} \right)\omega(E_Q) \\
&=& \sum_{Q\in \mathcal{S}^\beta}\Psi\left(\prod_{i=1}^{n} m_{\sigma_{i}}(g_{i},Q) \right)\int_{E_Q}\Psi\left(\prod_{i=1}^{n}  \frac{\sigma_{i}(Q)}{|Q|^{1-\frac{\alpha}{nd}}}\right)\omega(x)dx\\
&\lesssim&  \sum_{Q\in \mathcal{S}^\beta}\Psi\left(\prod_{i=1}^{n} m_{\sigma_{i}}(g_{i},Q) \right)\int_{E_Q}\Psi\left(\mathcal{M}_\alpha(\overrightarrow{\sigma}.\overrightarrow{\chi_{Q}} )(x)\right)\omega(x)dx \\
&=& \sum_{Q\in \mathcal{D}^{\beta}}\lambda_{Q}\Psi\left(\prod_{i=1}^{n}m_{\sigma_{i}}(g_{i},Q)\right),
\end{eqnarray*} 
where  
\[ \lambda_{Q} :=
  \left\{
  \begin{array}{ll} \int_{E_Q}\Psi\left(\mathcal{M}_\alpha(\overrightarrow{\sigma}.\overrightarrow{\chi_Q} )(x)\right)\omega(x)dx & \mbox{ if $Q\in \mathcal{S}^\beta$}    \\ 0 & \mbox{  otherwise}.
  \end{array}
  \right. \]
Fix $R \in \mathcal{D}^{\beta}$.
 As the $E_Q$s are pairwise disjoint and $E_Q\subset Q$, using that $(\overrightarrow{\sigma},\omega) \in S_{\overrightarrow{\Phi},\Psi}^{\alpha}$, we obtain
\begin{align*}
 \sum_{Q\subset R : Q \in \mathcal{D}^{\beta}}\lambda_{Q}&= \sum_{Q\subset R, Q\in \mathcal{S}^\beta }\,\int_{E_Q}\Psi\left(\mathcal{M}_\alpha(\overrightarrow{\sigma}.\overrightarrow{\chi_{Q}} )(x)\right)\omega(x)dx \\ &\leq \sum_{Q\subset R, Q\in \mathcal{S}^\beta }\,\int_{E_Q}\Psi\left(\mathcal{M}_\alpha(\overrightarrow{\sigma}.\overrightarrow{\chi_R} )(x)\right)\omega(x)dx \\
  &\leq \int_{R}\Psi\left(\mathcal{M}_\alpha(\overrightarrow{\sigma}.\overrightarrow{\chi_R} )(x)\right)\omega(x)dx\\
 &\leq \frac{[\overrightarrow{\sigma},\omega]_{S_{\overrightarrow{\Phi},\Psi}^{\alpha}}}{\Psi \circ \Phi^{-1}\left( \frac{1}{\nu_{\overrightarrow{\sigma}}(R)}\right)}.
\end{align*}
It then follows from Theorem \ref{pro:main16yh6} that
$$  \sum_{Q\in \mathcal{D}^{\beta}}\lambda_{Q}\Psi\left(\prod_{i=1}^{n}m_{\sigma_{i}}(g_{i},Q)\right)\lesssim [\overrightarrow{\sigma},\omega]_{S_{\overrightarrow{\Phi},\Psi}^{\alpha}}.        $$
Hence
$$\int_{\mathbb{R}^{d}}\Psi\left( \mathcal{M}_{\alpha}^{\mathcal{D}^{\beta}}\left(\overrightarrow{\sigma}.\overrightarrow{g}\right)(x)\right)\omega(x)dx\lesssim [\overrightarrow{\sigma},\omega]_{S_{\overrightarrow{\Phi},\Psi}^{\alpha}}.$$
This is enough to conclude that Theorem \ref{pro:main126aq} is true.
\end{proof}
The following provides a converse of Theorem \ref{pro:main126aq} under the condition (\ref{eq:surmaron}).

\begin{proposition}\label{pro:maiaqwqyh6}
Let $\sigma_{1},...,\sigma_{n},\omega$ be weights on $\mathbb{R}^{d}$, $\Phi_{1},...,\Phi_{n}\in \mathscr{U}$,  and $\Psi \in \widetilde{\mathscr{U}}$. Let $\Phi$ be a one-to-one correspondence from $\mathbb{R}_{+}$ to itself such that 
 $\Phi^{-1}=\Phi_{1}^{-1}\times...\times \Phi_{n}^{-1}$, and $t\mapsto \frac{\Psi(t)}{\Phi(t)}$ is increasing on $\mathbb{R}_{+}^{*}$.  If (\ref{eq:surmaron}) holds and  $\mathcal{M}_{\alpha}(\overrightarrow{\sigma}. ) : L^{\Phi_{1}}(\sigma_{1})\times...\times L^{\Phi_{n}}(\sigma_{n})\longrightarrow L^{\Psi}(\omega)$ boundedly, then $(\overrightarrow{\sigma}, \omega) \in S_{\overrightarrow{\Phi},\Psi}^{\alpha}$.
\end{proposition}

\begin{proof}
Let us put $$  \nu_{\overrightarrow{\sigma}}:=  \frac{1}{\Phi\left(\prod_{i=1}^{n}\Phi_{i}^{-1}\left(\frac{1}{\sigma_{i}}\right)  \right)}.    $$
Let $R \in \mathcal{Q}$.  As (\ref{eq:surmaron}) is satisfied, and  $\Psi \in \widetilde{\mathscr{U}}$, it follows from Lemma \ref{pro:main80jdf}, that
\begin{align*}
\Psi \circ \Phi^{-1}\left(\frac{1}{\nu_{\overrightarrow{\sigma}}(R)}\right) &\lesssim \Psi \circ \Phi^{-1}\left(\Phi\left(\prod_{i=1}^{n}\Phi_{i}^{-1}\left(\frac{1}{\sigma_{i}(R)}\right)\right)\right) \\
&=\Psi \left(\prod_{i=1}^{n}\Phi_{i}^{-1}\left(\frac{1}{\sigma_{i}(R)}\right)\right)
=  \Psi\left(\frac{1}{\prod_{i=1}^{n}\|\chi_{R}\|_{L_{\sigma_{i}}^{\Phi_{i}}}^{lux}}\right) \\
&\lesssim \frac{\Psi(1)}{\Psi\left(\prod_{i=1}^{n}\|\chi_{R}\|_{L_{\sigma_{i}}^{\Phi_{i}}}^{lux}\right)}.
\end{align*}
As $\mathcal{M}_{\alpha}(\overrightarrow{\sigma}.) : L^{\Phi_{1}}(\sigma_{1})\times...\times L^{\Phi_{n}}(\sigma_{n})\longrightarrow L^{\Psi}(\omega)$  boundedly, we have
\begin{align*}
\int_R\Psi\left(\mathcal{M}_\alpha(\sigma_{1}\chi_R,...,\sigma_{n}\chi_R )(x)\right)\omega(x)dx&=\int_R\Psi\left(\prod_{i=1}^{n}\|\chi_{R}\|_{L_{\sigma_{i}}^{\Phi_{i}}}^{lux}\times\frac{\mathcal{M}_\alpha(\sigma_{1}\chi_R,...,\sigma_{n}\chi_R)(x)}{\prod_{i=1}^{n}\|\chi_{R}\|_{L_{\sigma_{i}}^{\Phi_{i}}}^{lux}}\right)\omega(x)dx\\
&\lesssim \Psi\left(\prod_{i=1}^{n}\|\chi_{R}\|_{L_{\sigma_{i}}^{\Phi_{i}}}^{lux}\right) \int_R\Psi\left(\frac{\mathcal{M}_\alpha(\sigma_{1}\chi_R,...,\sigma_{n}\chi_R)(x)}{\prod_{i=1}^{n}\|\chi_{R}\|_{L_{\sigma_{i}}^{\Phi_{i}}}^{lux}}\right)\omega(x)dx\\
&\lesssim \Psi\left(\prod_{i=1}^{n}\|\chi_{R}\|_{L_{\sigma_{i}}^{\Phi_{i}}}^{lux}\right).
\end{align*}
It follows that $$  \Psi \circ \Phi^{-1}\left( \frac{1}{\nu_{\overrightarrow{\sigma}}(R)}\right)\int_R\Psi\left(\mathcal{M}_\alpha(\sigma_{1}\chi_R,...,\sigma_{n}\chi_R )(x)\right)\omega(x)dx \lesssim 1.      $$  
\end{proof}

Let  $\Phi_{1},...,\Phi_{n} \in \mathscr{U}$, and let $\Phi$ be a one-to-one correspondence from $\mathbb{R}_{+}$ to itself such that 
 $\Phi^{-1}=\Phi_{1}^{-1}\times...\times \Phi_{n}^{-1}$.
 If for any $1\leq i \leq n$,  $\sigma_{i}\equiv \sigma$, then  $\nu_{\overrightarrow{\sigma}}\equiv \sigma$ and (\ref{eq:surmaron})
holds. In this case, $\mathcal{M}_{\alpha}(\overrightarrow{\sigma}. ) : L^{\Phi_{1}}(\sigma)\times...\times L^{\Phi_{n}}(\sigma)\longrightarrow L^{\Psi}(\omega)$ boundedly if and only if $(\sigma,\omega)\in S_{\overrightarrow{\Phi},\Psi}^{\alpha}$. Indeed, if $\mathcal{M}_{\alpha}(\overrightarrow{\sigma}. ) : L^{\Phi_{1}}(\sigma)\times...\times L^{\Phi_{n}}(\sigma)\longrightarrow L^{\Psi}(\omega)$ boundedly, then  $(\sigma,\omega)\in S_{\overrightarrow{\Phi},\Psi}^{\alpha}$ (see Proposition \ref{pro:maiaqwqyh6}). Conversely, by replacing Theorem \ref{pro:main16yh6} by Corollary \ref{pro:main16aaqq6} in the proof of Theorem \ref{pro:main126aq}, we obtain that if   $(\sigma,\omega)\in S_{\overrightarrow{\Phi},\Psi}^{\alpha}$, then $\mathcal{M}_{\alpha}(\overrightarrow{\sigma}. ) : L^{\Phi_{1}}(\sigma)\times...\times L^{\Phi_{n}}(\sigma)\longrightarrow L^{\Psi}(\omega)$ boundedly. Hence Corollary \ref{pro:main12sasd6} is proved.

\subsection{Proofs of the weighted norm estimates}
Let us recall the following.
\begin{lemma}(\cite[Lemma 2.1]{HytPerez})\label{pro:main063}
Let $f$ be a measurable function on  $\mathbb{R}^{d}$.
The logarithmic maximal function of $f$ defined by
\begin{equation}\label{eq:onc5ax1l}
\mathcal{M}_0(f)(x):=\sup_{Q\in \mathcal{Q}}\exp\left(\frac{\chi_Q(x)}{|Q|}\int_Q\log\left(|f(t)|\right)dt\right),\end{equation}
satisfies 
\begin{equation}\label{eq:onc35o65ax1l}
\|\mathcal{M}_0(f)\|_{L^p}\le c_d^{\frac 1p}\|f\|_{L^p}, ~~\forall~ p>0.\end{equation}
\end{lemma} 
\vskip .2cm
We next prove Theorem \ref{pro:main1aq6}.
\vskip .1cm
 \begin{proof}[Proof of Theorem \ref{pro:main1aq6}]
Let $(f_{1},...,f_{n})\in L^{\Phi_{1}}(\sigma_{1})\times...\times L^{\Phi_{n}}(\sigma_{n})$. Suppose that for any $1\leq i \leq n$,  $f_{i}\not\equiv 0$. Set  $\overrightarrow{g}:=(g_{1},...,g_{n})$ with $g_{i}=\frac{f_{i}}{\|f_{i}\|_{L_{\sigma_{i}}^{\Phi_{i}}}^{lux}}$,  for $1\leq i \leq n$.\\ 
As  $\Psi$ is convex and of upper-type $q$, from (\ref{eq:of895ua1l}), we deduce that we only have to estimate
$$ \int_{\mathbb{R}^{d}}\Psi\left( \mathcal{M}_{\alpha}^{\mathcal{D}^{\beta}}\left(\overrightarrow{\sigma}.\overrightarrow{g}\right)(x)\right)\omega(x)dx.        $$
Let $\beta \in \left\{0, 1/3\right\}^{d}$ be fixed. Following \cite{sehbaf}, we have that there is a sparse family $\mathcal{S}^\beta$ such that for some $0<\epsilon<1$ depending only on $n,d$ and a fixed parameter,
\begin{equation}\label{eq:qardsax1l}
|Q| \leq \frac{1}{1-\epsilon}|E_Q|, ~~\forall~~Q\in\mathcal{S}^\beta.\end{equation}
As in Theorem \ref{pro:main126aq}, we obtain   
\begin{equation}\label{eq:qsax1l}
\int_{\mathbb{R}^{d}}\Psi\left( \mathcal{M}_{\alpha}^{\mathcal{D}^{\beta}}\left(\overrightarrow{\sigma}.\overrightarrow{g}\right)(x)\right)\omega(x)dx \lesssim \sum_{Q\in\mathcal{S}^\beta}\Psi\left(\prod_{i=1}^{n} m_{\sigma_{i}}(g_{i},Q) \right)\Psi\left(\prod_{i=1}^{n}\frac{\sigma_{i}(Q)}{|Q|^{1-\frac{\alpha}{nd}}} \right)\omega(E_Q).
\end{equation}
$i)$ Suppose that  $(\overrightarrow{\sigma},\omega) \in B_{\overrightarrow{\Phi},\Psi}^{\alpha}$.\\  
Let $1\leq i \leq n$. Put $$  \Psi_{i}=\Psi\circ \Phi^{-1}\circ \Phi_{i}.      $$       
As $t\mapsto \frac{\Psi(t)}{\Phi(t)}$ is increasing on $\mathbb{R}_{+}^{*}$ and  $\Psi \in \mathscr{U}$, from Lemma \ref{pro:main40}, we deduce that $\Psi\circ \Phi^{-1} \in \mathscr{U}$. In particular, $\Psi\circ \Phi^{-1}$ is a convex growth function. It follows that $\Psi_{i}$ is also a convex growth function.\\
Let  $t>0$. As $\Phi^{-1}=\Phi_{1}^{-1}\times...\times \Phi_{n}^{-1}$, we have
\begin{align*} 
\Psi_{1}^{-1}(t)\times...\times\Psi_{n}^{-1}(t)&= \prod_{i=1}^{n}\Psi_{i}^{-1}(t)=\prod_{i=1}^{n}\left(\Psi\circ \Phi^{-1}\circ \Phi_{i}\right)^{-1}(t)=\prod_{i=1}^{n}\Phi_{i}^{-1}\circ \Phi\circ \Psi^{-1}(t)\\
&= \prod_{i=1}^{n}\Phi_{i}^{-1}\left(\Phi\circ \Psi^{-1}(t)\right)  =  \Phi^{-1}\left(\Phi\circ \Psi^{-1}(t)\right) = \Psi^{-1}(t). 
\end{align*} 
Thus $\Psi^{-1}=\Psi_{1}^{-1}\times...\times\Psi_{n}^{-1}.$
\vskip .1cm
We first observe that
  \begin{align*}
L_{1}&:=\int_{\mathbb{R}^{d}}\Psi\left( \mathcal{M}_{\alpha}^{\mathcal{D}^{\beta}}\left(\overrightarrow{\sigma}.\overrightarrow{g}\right)(x)\right)\omega(x)dx \\ 
& \lesssim\sum_{Q\in\mathcal{S}^\beta}\Psi\left(\prod_{i=1}^{n} m_{\sigma_{i}}(g_{i},Q) \right)\Psi\left(\prod_{i=1}^{n}\frac{\sigma_{i}(Q)}{|Q|^{1-\frac{\alpha}{nd}}} \right)\omega(E_{Q}) \\
& \lesssim\sum_{Q\in\mathcal{S}^\beta}\Psi\left(\prod_{i=1}^{n} m_{\sigma_{i}}(g_{i},Q) \right)\Psi\left(\prod_{i=1}^{n}\frac{\sigma_{i}(Q)}{|Q|^{1-\frac{\alpha}{nd}}} \right)\omega(Q).
\end{align*}
Since $\Phi_{i} \in \widetilde{\mathscr{U}}$, for  $1\leq i \leq n$,  it follows from Lemma \ref{pro:main80jdf} that
\begin{align*}
L_{1}
&\lesssim[\overrightarrow{\sigma},\omega]_{B_{\overrightarrow{\Phi},\Psi}^{\alpha}}\sum_{Q\in\mathcal{S}^\beta}\Psi\left(\prod_{i=1}^{n} \frac{ m_{\sigma_{i}}(g_{i},Q)}{\Phi_{i}^{-1}\left(\frac{1}{|Q|}\exp\left(\frac{1}{|Q|}\int_{Q}\log\frac{1}{\sigma_{i}}\right)\right)} \right)\\
&\lesssim[\overrightarrow{\sigma},\omega]_{B_{\overrightarrow{\Phi},\Psi}^{\alpha}} \sum_{Q\in\mathcal{S}^\beta}\sum_{i=1}^{n}\Psi_{i}\left(\frac{m_{\sigma_{i}}(g_{i},Q)}{\Phi_{i}^{-1}\left(\frac{1}{|Q|}\exp\left(\frac{1}{|Q|}\int_{Q}\log\sigma_{i}^{-1}\right)\right) }\right) \\
&= [\overrightarrow{\sigma},\omega]_{B_{\overrightarrow{\Phi},\Psi}^{\alpha}}\sum_{i=1}^{n}\sum_{Q\in\mathcal{S}^\beta}
\Psi\circ \Phi^{-1}\circ \Phi_{i}\left(\frac{m_{\sigma_{i}}(g_{i},Q)}{\Phi_{i}^{-1}\left(\frac{1}{|Q|}\exp\left(\frac{1}{|Q|}\int_{Q}\log\sigma_{i}^{-1}\right)\right) }\right) \\
&\lesssim[\overrightarrow{\sigma},\omega]_{B_{\overrightarrow{\Phi},\Psi}^{\alpha}}\sum_{i=1}^{n}\Psi\circ \Phi^{-1}\left(\sum_{Q\in\mathcal{S}^\beta} \frac{\Phi_{i}\left(m_{\sigma_{i}}(g_{i},Q)\right)}{\Phi_{i}\circ\Phi_{i}^{-1}\left(\frac{1}{|Q|}\exp\left(\frac{1}{|Q|}\int_{Q}\log\sigma_{i}^{-1}\right)\right)}\right)  \\
& =[\overrightarrow{\sigma},\omega]_{B_{\overrightarrow{\Phi},\Psi}^{\alpha}}\sum_{i=1}^{n}\Psi\circ \Phi^{-1}\left(\sum_{Q\in\mathcal{S}^\beta}|Q|\exp\left(\frac{1}{|Q|}\int_{Q}\log\sigma_{i}\right) \Phi_{i}\left(m_{\sigma_{i}}(g_{i},Q)\right)\right) \\
& =[\overrightarrow{\sigma},\omega]_{B_{\overrightarrow{\Phi},\Psi}^{\alpha}} \sum_{i=1}^{n}\Psi\circ \Phi^{-1}\left(\sum_{Q\in \mathcal{D}^{\beta}} \lambda_{Q}^{i}\Phi_{i}\left(m_{\sigma_{i}}(g_{i},Q)\right)\right),
\end{align*}
where
\[ \lambda_{Q}^{i} :=
\left\{
\begin{array}{ll} |Q|\exp\left(\frac{1}{|Q|}\int_{Q}\log\sigma_{i}\right) & \mbox{ if $Q\in\mathcal{S}^\beta$}    \\ 0 & \mbox{ otherwise}.
\end{array}
\right. \]
Fix $R \in \mathcal{D}^{\beta}$. Following for example \cite{ChenDamian}, one easily obtains with the help of Lemma \ref{pro:main063}, that
$$\sum_{Q\subset R : Q \in \mathcal{D}^{\beta}}\lambda_{Q}^{i}\lesssim \frac{\sigma_{i}(R)}{1-\epsilon}.$$
We conclude that $\{\lambda_{Q}^{i}\}_{Q \in \mathcal{D}^{\beta}}$ is a $\sigma_{i}-$Carleson sequence. As 
$\Phi_{i}\in \nabla_{2}$,  it follows from Corollary \ref{pro:main2aq09},   that
$$  \sum_{Q\in \mathcal{D}^{\beta}} \lambda_{Q}^{i} \Phi_{i}\left(\frac{1}{\sigma_{i}(Q)}\int_{Q}|g_{i}(x)|\sigma_{i}(x)dx\right) \lesssim 1.        $$
Consequently, 
$$  \int_{\mathbb{R}^{d}}\Psi\left( \mathcal{M}_{\alpha}^{\mathcal{D}^{\beta}}\left(\overrightarrow{\sigma}.\overrightarrow{g}\right)(x)\right)\omega(x)dx \lesssim  [\overrightarrow{\sigma},\omega]_{B_{\overrightarrow{\Phi},\Psi}^{\alpha}}.$$
 $ii)$ Suppose that for any $1\leq i \leq n$, $\sigma_{i} \in \mathcal{A}_\infty$ and  $(\overrightarrow{\sigma},\omega) \in A_{\overrightarrow{\Phi},\Psi}^{\alpha}$.\\   
As $\Psi$ is of upper-type $q$ and satisfies the $\Delta'-$condition, we obtain
\begin{align*}
L_{2}&:=\int_{\mathbb{R}^{d}}\Psi\left( \mathcal{M}_{\alpha}^{\mathcal{D}^{\beta}}\left(\overrightarrow{\sigma}.\overrightarrow{g}\right)(x)\right)\omega(x)dx \\  
&\lesssim 
 \sum_{Q\in\mathcal{S}^\beta}\Psi\left(\prod_{i=1}^{n} m_{\sigma_{i}}(g_{i},Q) \right)\Psi\left(\prod_{i=1}^{n}\frac{\sigma_{i}(Q)}{|Q|^{1-\frac{\alpha}{nd}}} \right)\omega(E_Q) \\
& \lesssim \sum_{Q\in\mathcal{S}^\beta}\Psi\left(\prod_{i=1}^{n}\frac{\Phi_{i}^{-1}\left(\frac{1}{\sigma_{i}(Q)}\right)\times m_{\sigma_{i}}(g_{i},Q)}{\Phi_{i}^{-1}\left(\frac{1}{\sigma_{i}(Q)}\right)}  \right)\Psi\left(\prod_{i=1}^{n}\frac{\sigma_{i}(Q)}{|Q|^{1-\frac{\alpha}{nd}}} \right)\omega(Q) \\
& \lesssim  [\overrightarrow{\sigma},\omega]_{A_{\overrightarrow{\Phi},\Psi}^{\alpha}} \sum_{Q\in\mathcal{S}^\beta}\Psi\left(\prod_{i=1}^{n}\frac{ m_{\sigma_{i}}(g_{i},Q)}{\Phi_{i}^{-1}\left(\frac{1}{\sigma_{i}(Q)}\right)}  \right) \\
&\lesssim [\overrightarrow{\sigma},\omega]_{A_{\overrightarrow{\Phi},\Psi}^{\alpha}} \sum_{Q\in\mathcal{S}^\beta}\sum_{i=1}^{n}\Psi_{i}\left(\frac{ m_{\sigma_{i}}(g_{i},Q)}{\Phi_{i}^{-1}\left(\frac{1}{\sigma_{i}(Q)}\right)}  \right) \\
&= [\overrightarrow{\sigma},\omega]_{A_{\overrightarrow{\Phi},\Psi}^{\alpha}} \sum_{i=1}^{n}\sum_{Q\in\mathcal{S}^\beta}\Psi\circ \Phi^{-1}\circ \Phi_{i}\left(\frac{ m_{\sigma_{i}}(g_{i},Q)}{\Phi_{i}^{-1}\left(\frac{1}{\sigma_{i}(Q)}\right)}  \right) \\
&\lesssim [\overrightarrow{\sigma},\omega]_{A_{\overrightarrow{\Phi},\Psi}^{\alpha}} \sum_{i=1}^{n}\Psi\circ \Phi^{-1}\left( \sum_{Q\in\mathcal{S}^\beta} \Phi_{i}\left(\frac{ m_{\sigma_{i}}(g_{i},Q)}{\Phi_{i}^{-1}\left(\frac{1}{\sigma_{i}(Q)}\right)}  \right) \right)\\
&\lesssim [\overrightarrow{\sigma},\omega]_{A_{\overrightarrow{\Phi},\Psi}^{\alpha}} \sum_{i=1}^{n}\Psi\circ \Phi^{-1}\left(\sum_{Q\in\mathcal{S}^\beta} \frac{\Phi_{i}\left( m_{\sigma_{i}}(g_{i},Q)\right)}{\Phi_{i}\left(\Phi_{i}^{-1}\left(\frac{1}{\sigma_{i}(Q)}\right)\right)}  \right)\\
& \lesssim [\overrightarrow{\sigma},\omega]_{A_{\overrightarrow{\Phi},\Psi}^{\alpha}} \sum_{i=1}^{n}\Psi\circ \Phi^{-1}\left(\sum_{Q\in\mathcal{S}^\beta} \sigma_{i}(Q)\Phi_{i}\left( m_{\sigma_{i}}(g_{i},Q)\right)\right)\\
& = [\overrightarrow{\sigma},\omega]_{A_{\overrightarrow{\Phi},\Psi}^{\alpha}} \sum_{i=1}^{n}\Psi\circ \Phi^{-1}\left(\sum_{Q \in \mathcal{D}^{\beta}} \lambda_{Q}^{i}\Phi_{i}\left( m_{\sigma_{i}}(g_{i},Q)\right)\right),
 \end{align*} 
where
\[ \lambda_{Q}^{i} :=
\left\{
\begin{array}{ll} \sigma_{i}(Q) & \mbox{ if $Q\in\mathcal{S}^\beta$}    \\ 0 & \mbox{ otherwise}.
\end{array}
\right. \]
For any $R \in \mathcal{D}^{\beta}$ fixed, it is known (see for example \cite{sehbaf}) that
$$
\sum_{Q\subset R : Q \in \mathcal{D}^{\beta}}\lambda_{Q}^{i}
\lesssim  \frac{[\sigma_{i}]_{\mathcal{A}_\infty}}{1-\epsilon}\sigma_{i}(R).
$$
That is 
 $\{\lambda_{Q}^{i}\}_{Q \in \mathcal{D}^{\beta}}$ is a  $\sigma_{i}-$Carleson sequence.
 As $\Phi_{i}\in \nabla_{2}$, using Corollary \ref{pro:main2aq09},   we deduce that
$$  \sum_{Q\in \mathcal{D}^{\beta}} \lambda_{Q}^{i} \Phi_{i}\left(m_{\sigma_{i}}(g_{i},Q)\right) \lesssim [\sigma_{i}]_{\mathcal{A}_\infty}.        $$
Thus
$$  \int_{\mathbb{R}^{d}}\Psi\left( \mathcal{M}_{\alpha}^{\mathcal{D}^{\beta}}\left(\overrightarrow{\sigma}.\overrightarrow{g}\right)(x)\right)\omega(x)dx \lesssim [\overrightarrow{\sigma},\omega]_{A_{\overrightarrow{\Phi},\Psi}^{\alpha}}\sum_{i=1}^{n}\Psi\circ\Phi^{-1}\left([\sigma_{i}]_{\mathcal{A}_\infty}\right).
       $$ 

       
$iii)$ Suppose $(\overrightarrow{\sigma},\omega) \in \widetilde{A}_{\overrightarrow{\varphi},\Psi}^{\alpha}$ and $\overrightarrow{\sigma} \in W_{\overrightarrow{\Phi},\Psi}$.
Let $1\leq i \leq n$.
As $\varphi_{i}$ is the complementary function of $\Phi_{i}$, we have that $$  t<\Phi_{i}^{-1}(t)\varphi_{i}^{-1}(t)\leq 2t,~~\forall~t>0.       $$
We obtain
\begin{align*}
L_{3}&:=\int_{\mathbb{R}^{d}}\Psi\left( \mathcal{M}_{\alpha}^{\mathcal{D}^{\beta}}\left(\overrightarrow{\sigma}.\overrightarrow{g}\right)(x)\right)\omega(x)dx \\ 
&\lesssim \sum_{Q\in \mathcal{S}^\beta}\Psi\left(\prod_{i=1}^{n} m_{\sigma_{i}}(g_{i},Q) \right)\Psi\left(\prod_{i=1}^{n}\frac{\sigma_{i}(Q)}{|Q|^{1-\frac{\alpha}{nd}}} \right)\omega(E_Q)\\
&=\sum_{Q\in \mathcal{S}^\beta}\Psi\left(\prod_{i=1}^{n} m_{\sigma_{i}}(g_{i},Q) \right)\Psi\left(|Q|^{\frac{\alpha}{d}}\prod_{i=1}^{n}\frac{\sigma_{i}(Q)}{|Q|} \right)\omega(E_Q)\\
&\lesssim \sum_{Q\in \mathcal{S}^\beta}\Psi\left(\prod_{i=1}^{n} m_{\sigma_{i}}(g_{i},Q) \right)\Psi\left(|Q|^{\frac{\alpha}{d}}\prod_{i=1}^{n}\varphi_{i}^{-1}\left(\frac{\sigma_{i}(Q)}{|Q|}\right) \right)\Psi\left(\prod_{i=1}^{n}\Phi_{i}^{-1}\left(\frac{\sigma_{i}(Q)}{|Q|}\right) \right)\omega(Q)\\
&\lesssim [\overrightarrow{\sigma},\omega]_{\widetilde{A}_{\overrightarrow{\varphi},\Psi}^{\alpha}}\sum_{Q\in \mathcal{S}^\beta}\Psi\left(\prod_{i=1}^{n} m_{\sigma_{i}}(g_{i},Q) \right)\Psi\left(\prod_{i=1}^{n}\Phi_{i}^{-1}\left(\frac{\sigma_{i}(Q)}{|Q|}\right) \right)|Q| \\
&\lesssim [\overrightarrow{\sigma},\omega]_{\widetilde{A}_{\overrightarrow{\varphi},\Psi}^{\alpha}}\sum_{Q\in \mathcal{S}^\beta}\Psi\left(\prod_{i=1}^{n} m_{\sigma_{i}}(g_{i},Q) \right)\Psi\left(\prod_{i=1}^{n}\Phi_{i}^{-1}\left(\frac{\sigma_{i}(Q)}{|Q|}\right) \right)|E_Q| \\
&\lesssim [\overrightarrow{\sigma},\omega]_{\widetilde{A}_{\overrightarrow{\varphi},\Psi}^{\alpha}}\sum_{Q\in \mathcal{S}^\beta}\Psi\left(\prod_{i=1}^{n} m_{\sigma_{i}}(g_{i},Q) \right)\int_{E_Q}\Psi\left(\prod_{i=1}^{n}\Phi_{i}^{-1}\left(\mathcal{M}(\sigma_{i}\chi_{Q})(x)\right) \right)dx \\
&=[\overrightarrow{\sigma},\omega]_{\widetilde{A}_{\overrightarrow{\varphi},\Psi}^{\alpha}}\sum_{Q\in \mathcal{D}^{\beta}}\lambda_{Q}\Psi\left(\prod_{i=1}^{n}m_{\sigma_{i}}(g_{i},Q)\right),
\end{align*}
where 
\[ \lambda_{Q} :=
  \left\{
  \begin{array}{ll} \int_{E_Q}\Psi\left(\prod_{i=1}^{n}\Phi_{i}^{-1}\left(\mathcal{M}(\sigma_{i}\chi_{Q})(x)\right) \right)dx & \mbox{ if $Q\in \mathcal{S}^{\beta}$}    \\ 0 & \mbox{otherwise}.
  \end{array}
  \right. \]
  
Recall the notation $$\nu_{\overrightarrow{\sigma}}:=  \frac{1}{\Phi\left(\prod_{i=1}^{n}\Phi_{i}^{-1}\left(\frac{1}{\sigma_{i}}\right)  \right)}.$$
Let us fix $R \in \mathcal{D}^{\beta}$.
As $\overrightarrow{\sigma} \in W_{\overrightarrow{\Phi},\Psi}$, we obtain 
\begin{align*}
 \sum_{Q\subset R : Q \in \mathcal{D}^{\beta}}\lambda_{Q}&= \sum_{Q\subset R, Q \in \mathcal{S}^{\beta} }\,\int_{E_Q}\Psi\left(\prod_{i=1}^{n}\Phi_{i}^{-1}\left(\mathcal{M}(\sigma_{i}\chi_{Q})(x)\right) \right)dx\\
 &\leq \sum_{Q\subset R, Q \in \mathcal{S}^\beta }\,\int_{E_Q}\Psi\left(\prod_{i=1}^{n}\Phi_{i}^{-1}\left(\mathcal{M}(\sigma_{i}\chi_{Q})(x)\right) \right)dx \\
  &\leq \int_{R}\Psi\left(\prod_{i=1}^{n}\Phi_{i}^{-1}\left(\mathcal{M}(\sigma_{i}\chi_{Q})(x)\right) \right)dx\\
 &\leq \frac{[\overrightarrow{\sigma}]_{W_{\overrightarrow{\Phi},\Psi}}}{\Psi \circ \Phi^{-1}\left( \frac{1}{\nu_{\overrightarrow{\sigma}}(R)}\right)}. 
\end{align*}
It follows from Theorem \ref{pro:main16yh6}, that 
$$ \int_{\mathbb{R}^{d}}\Psi\left( \mathcal{M}_{\alpha}^{\mathcal{D}^{\beta}}\left(\overrightarrow{\sigma}.\overrightarrow{g}\right)(x)\right)\omega(x)dx \lesssim [\overrightarrow{\sigma}]_{W_{\overrightarrow{\Phi},\Psi}}[\overrightarrow{\sigma},\omega]_{\widetilde{A}_{\overrightarrow{\varphi},\Psi}^{\alpha}}.        $$
The proof is complete.
\end{proof}
We next prove Theorem \ref{pro:mainineq4}.
\begin{proof}[Proof of Theorem \ref{pro:mainineq4}]
Suppose that for any $1\leq i \leq n$, $\sigma_{i} \in \mathcal{A}_\infty$ and  $(\overrightarrow{\sigma},\omega) \in A_{\overrightarrow{\Phi},\Psi}^{\alpha}$.\\   
Recall that $\Psi$ is of upper-type $q$ and satisfies the $\Delta'-$condition. Above, we have obtained that
\begin{align*}
L_{2}&:=\int_{\mathbb{R}^{d}}\Psi\left( \mathcal{M}_{\alpha}^{\mathcal{D}^{\beta}}\left(\overrightarrow{\sigma}.\overrightarrow{g}\right)(x)\right)\omega(x)dx \\  
& \lesssim  [\overrightarrow{\sigma},\omega]_{A_{\overrightarrow{\Phi},\Psi}^{\alpha}} \sum_{Q\in\mathcal{S}^\beta}\Psi\left(\prod_{i=1}^{n}\frac{ m_{\sigma_{i}}(g_{i},Q)}{\Phi_{i}^{-1}\left(\frac{1}{\sigma_{i}(Q)}\right)}  \right).
\end{align*}
Hence using that $\Psi\in \widetilde{\mathscr{U}}$, we get
\begin{eqnarray*}
L_2&\lesssim& [\overrightarrow{\sigma},\omega]_{A_{\overrightarrow{\Phi},\Psi}^{\alpha}} \sum_{Q\in\mathcal{S}^\beta}\prod_{i=1}^{n}\frac{\Psi( m_{\sigma_{i}}(g_{i},Q))}{\Psi\circ\Phi_{i}^{-1}\left(\frac{1}{\sigma_{i}(Q)}\right)}   \\
&=& [\overrightarrow{\sigma},\omega]_{A_{\overrightarrow{\Phi},\Psi}^{\alpha}} \prod_{i=1}^{n}\sum_{Q\in\mathcal{S}^\beta}\frac{\Psi( m_{\sigma_{i}}(g_{i},Q))}{\Psi\circ\Phi_{i}^{-1}\left(\frac{1}{\sigma_{i}(Q)}\right)} \\
&=& [\overrightarrow{\sigma},\omega]_{A_{\overrightarrow{\Phi},\Psi}^{\alpha}} \prod_{i=1}^{n}\sum_{Q \in \mathcal{D}^{\beta}} \lambda_{Q}^{i}\Psi\left( m_{\sigma_{i}}(g_{i},Q)\right),
 \end{eqnarray*} 
where
\[ \lambda_{Q}^{i} :=
\left\{
\begin{array}{ll} \frac{1}{\Psi\circ\Phi_{i}^{-1}\left(\frac{1}{\sigma_{i}(Q)}\right)} & \mbox{ if $Q\in\mathcal{S}^\beta$}    \\ 0 & \mbox{ otherwise}.
\end{array}
\right. \]
Define $\widetilde{\Phi}_i$ by $$\widetilde{\Phi}_i(0)=0,\textrm{and}\quad \widetilde{\Phi}_i(t)=\frac{1}{\Psi\circ\Phi_i^{-1}\left(\frac 1t\right)}\quad \forall t>0.$$
Then as $t\longrightarrow \frac{\Psi(t)}{\Phi_i(t)}$ is nondecreasing, we have from \cite[Lemma 3.1]{jmtafeuseh} that $\widetilde{\Phi}_i\in \mathscr{U}^{q_i}$ for some $q_i>1$. Hence from \cite[Page 20]{jmtafeuseh}, we have that for 
any $R \in \mathcal{D}^{\beta}$ fixed, 
$$
\sum_{Q\subset R : Q \in \mathcal{D}^{\beta}}\lambda_{Q}^{i}
\lesssim  [\sigma_{i}]_{\mathcal{A}_\infty}^{q_i}\widetilde{\Phi}_i\left(\sigma_{i}(R)\right).
$$
Using Corollary \ref{pro:main2aq09},   we conclude that
$$  \int_{\mathbb{R}^{d}}\Psi\left( \mathcal{M}_{\alpha}^{\mathcal{D}^{\beta}}\left(\overrightarrow{\sigma}.\overrightarrow{g}\right)(x)\right)\omega(x)dx \lesssim [\overrightarrow{\sigma},\omega]_{A_{\overrightarrow{\Phi},\Psi}^{\alpha}}\prod_{i=1}^{n}[\sigma_{i}]_{\mathcal{A}_\infty}^{q_i}.
       $$ 
       The proof is complete.
\end{proof}

\bibliographystyle{plain}

\begin{thebibliography}{1}
\bibitem{Bloom}
\textsc{Bloom,~S.,~Kerman,~R.}, \emph{Weighted Orlicz integral inequalities for the Hardy-Littlewood maximal operator,} Studia Math. {\bf 110} (1994), no. 2, 149--167.

\bibitem{CaoXue}
\textsc{ Cao,~M.,~Tanaka,~H.,~Xue,~Q.,~Yabuta,~K.}, A note on Carleson embedding theorems. Plub. Math. Debrecen \textbf{93/3-4} (2018), 517-523. 

\bibitem{ChenDamian}
\textsc{Chen,~W.,~Dami\'an,~W.}, \emph{Weighted estimates for the multisublinear maximal function,} Rend. Circ. Mat. Palermo {\bf 62} (2013), no. 3, 379--391.

\bibitem{CruzMoen}
\textsc{ Cruz-Uribe,~ D.,~ Moen,~ K.} \emph{A multilinear reverse Hölder inequality with applications to multilinear weighted norm inequalities.} Georgian Math. J. {\bf 27} (2020), no. 1, 37-42.

\bibitem{jmtafeuseh}
	\textsc{J.M. Tanoh Dje, J. Feuto and B.F. Sehba}, \emph{A generalization of the Carleson lemma and Weighted norm inequaliies for the maximal functions in the Orlicz setting.}  J. Math. Anal. Appl. {\bf 491} (2020), no. 1, 124248, 24 pp.
    	
\bibitem{djesehb}
    \textsc{J.M. Tanoh Dje and B.F. Sehba}, \emph{Carleson embeddings for Hardy-Orlicz and Bergman-Orlicz spaces of the upper-half plane}.  Funct. Approx. Comment. Math. {\bf 64} (2021), no. 2, 163–201.

\bibitem{Fujii}
\textsc{Fujii,~N.}, \emph{Weighted bounded mean oscillation and singular integrals,} Math. Japon {\bf 22} (1977/78), no. 5, 529--534.

\bibitem{Grafakos1}
\textsc{Grafakos,~L.}, \emph{On multilinear fractional integrals}, Studia Math., {\bf 102} (1992), 49-56.

\bibitem{Grafakos2}
\textsc{Grafakos,~L.,~Kalton,~N.}, \emph{Some remarks on multilinear maps and interpolation}, Math. Ann., {\bf 319} (2001), 151-180.

\bibitem{Grafakos3}
\textsc{Grafakos,~L.,~Torres,~R.~H.}, \emph{Maximal operator and weighted norm inequalities for multilinear singular integrals}, Indiana Math. J., {\bf 51} (2002), 1261-1276.


\bibitem{Hurscev}
\textsc{Hur\v{s}\v{c}ev,~S.~V.}, \emph{A description of weights satisfying the $A_\infty$ condition of Muckenhoupt,} Proc. Amer. Math. Soc. {\bf 90} (1984), no. 2, 253--257.

 \bibitem{HytPerez}
 \textsc{Hyt\"onen,~T.,~P\'erez,~C.}, \emph{Sharp weighted bounds involving $A_\infty$,}  Anal. and P.D.E. {\bf 6} (2013), no. 4, 777--818.

\bibitem{Kerman}
\textsc{Kerman,~R.,~Torchinsky,~A.}, \emph{Integral inequalities with weights for Hardy maximal function,} Studia Math. {\bf 71} (1981/82), 277--284.

\bibitem{KS}
\textsc{Kenig,~C.~E.,~Stein,~E.~M.}, \emph{Multilinear estimates and fractional integration}, Math. Res. Lett., {\bf 6} (1999), 1-15.

\bibitem{LOPTT}
\textsc{Lerners,~A.~K,~Ombrosi,~S.,~P\'erez,~C.~M.,~Torres,~R.~H.,~Trujillo-Gonz\'alez,~R.}, \emph{New maximal functions and multiple weights for multilinear Calder\'on-Zygmund theory}, Adv. Math., {\bf 220} (2009), 1222-1264.


\bibitem{kokokrbec}
  \textsc{V. Kokilashvili, M. Krbec}, \emph{Weighted Inequalities in Lorentz and Orlicz spaces},World Scientific publishing. C.O.Pte.Ltd (\textbf{1991}).
    


\bibitem{LiSun}
 \textsc{Li,~K.,~Sun,~W.}, \emph{Characterization of two weight inequality for multilinear fractional maximal function}, available at
http://arxiv.org/abs/1305.4267..

\bibitem{Moen1}
\textsc{Moen,~K.}, \emph{Sharp one-weight and two-weight bounds for maximal operators,} Studia Math. {\bf 194} (2009), no. 2, 163--180.

\bibitem{Moen2}
\textsc{Moen,~K.}, \emph{Weighted inequalities for multilinear fractional integral operators,} Collect. Math. {\bf 60} (2009), no. 2, 213--238.

\bibitem{Pereyra}
\textsc{Pereyra,~M.~C.}, \emph{Lecture notes on dyadic harmonic analysis,} Contemp. Math. {\bf 289} (2001), 1--60.

\bibitem{LQ}
\textsc{Quinsheng,~L.}, \emph{Two weight $\Phi$-inequalities for Hardy operator, Hardy-Littlewood maximal operator and fractional integrals,} Proc. Amer. Math. Soc. {\bf 118} (1992), no. 1, 129--142.


\bibitem{RahmSpencer}
\textsc{R. Rahm and S. Spencer }, \emph{Entropy bump conditions for fractional maximal and integral operators.} Concr. Oper. {\bf 3} (2016), no. 1, 112–121.

\bibitem{Israel}
\textsc{Rivera-Rios,~I.~P.}, \emph{A quantitative approach to weighted Carleson condition,} Concr. Oper. {\bf 4} (2017), no. 1, 58--75. 

    
\bibitem{raoren}
	\textsc{M.M. Rao and Z.D. Ren}, \emph{Theory of Orlicz Spaces}, Marcel Dekker,INC, \textbf{270}.(\textbf{1991}).	

\bibitem{rao68ren}
	\textsc{M.M. Rao and Z.D. Ren}, \emph{Applications of Orlicz Spaces}, Marcel Dekker,INC, All Rights Reserved \textbf{270}.(\textbf{2002}).
	
\bibitem{Sawyer5}
\textsc{E. Sawyer}, \emph{A characterization of two-weight norm inequality for maximal operators.} Studia Math.\textbf{75}.(\textbf{1982}), no. 1, p.p.1-11.
	
	
\bibitem{sehbaf}
    \textsc{B.F. Sehba}, \emph{On two-weight norm estimates for multilinear fractional maximal function},J. Math. Soc. Japan \textbf{70} no. 1 (\textbf{2018}), 71-94.

 \end{thebibliography}

\end{document}